\documentclass[12pt]{amsart}

\usepackage{mathtools}
\usepackage{nicefrac}
\usepackage[shortlabels]{enumitem}
\usepackage{tikz-cd}
\usepackage[bookmarksopen,bookmarksdepth=3,hidelinks]{hyperref}
\usepackage[T1]{fontenc}
\usepackage[capitalise]{cleveref}

\calclayout

\newcommand{\w}{\omega}
\newcommand{\R}{\mathbb{R}}
\newcommand{\C}{\mathbb{C}}
\newcommand{\N}{\mathbb{N}}
\newcommand{\Z}{\mathbb{Z}}
\newcommand{\T}{\mathbb{T}}
\newcommand{\G}{\mathcal{G}}
\newcommand{\E}{\mathcal{E}}
\renewcommand{\L}{\mathcal{L}}
\renewcommand{\i}{\iota}
\renewcommand{\S}{\mathcal{S}}

\newcommand\norm[1]{\lVert#1\rVert}

\DeclareMathOperator{\Id}{Id}

\DeclareMathOperator{\supp}{supp}
\DeclareMathOperator{\interior}{int}

\DeclareMathOperator{\vectorspan}{span}

\newtheorem{thm}{Theorem}[section]
\newtheorem{lemma}[thm]{Lemma}

\newtheorem{cor}[thm]{Corollary}
\newtheorem{proposition}[thm]{Proposition}
\newtheorem{mainTheorem}{Theorem}

\theoremstyle{remark}
\newtheorem{examplex}[thm]{Example}
\newtheorem{remarkx}[thm]{Remark}

\newtheorem{notation}[thm]{Notation}

\newenvironment{remark}
  {\pushQED{\qed}\remarkx}
  {\popQED\endremarkx}
\newenvironment{example}
  {\pushQED{\qed}\examplex}
  {\popQED\endexamplex}

\theoremstyle{definition}
\newtheorem{definition}[thm]{Definition}
\newtheorem{question}[thm]{Question}

\crefname{subsection}{subsection}{subsections}
\crefname{mainTheorem}{theorem}{theorems}
\crefformat{enumi}{#2\textup{(#1)}#3}

\numberwithin{equation}{section}
\numberwithin{figure}{section}

\begin{document}

\title{Weinstein neighbourhood theorems for stratified subspaces}

\author{Yael Karshon}
\author{Sara B. Tukachinsky}
\author{Yoav Zimhony}
\address{Tel Aviv University}

\email{yaelkarshon@tauex.tau.ac.il,
sarabt1@gmail.com,
yoavzimhony@gmail.com}


\begin{abstract}
By analogy with Weinstein's neighbourhood theorem, we prove a uniqueness result for symplectic neighbourhoods of a large family of stratified subspaces.
This result generalizes existing constructions, e.g., in the search for exotic Lagrangians.
Along the way, we prove a strong version of Moser's trick and a (non-symplectic) tubular neighbourhood theorem for these stratified subspaces.
\end{abstract}

\maketitle
\tableofcontents

\section{Introduction}

\subsection{Normal form theorems}\label{subsection:normal_form_theorems}

Weinstein's Lagrangian neighbourhood theorem~\cite{WEINSTEIN1971329} states that a neighbourhood of a Lagrangian submanifold $L$ in a symplectic manifold is symplectomorphic to a neighbourhood of the zero section in the cotangent bundle $T^*L$. More generally, Weinstein~\cite[Lecture~5]{weinstein_lectures} used the tubular neighbourhood theorem and Moser's trick \cite{mosers_trick} to prove the following symplectic normal form theorem. Let $\left(M_0, \w_0\right)$ and $\left(M_1, \w_1\right)$ be symplectic manifolds, $N_0,N_1$ submanifolds, and $g:N_0 \to N_1$ a diffeomorphism. Then there exists a symplectomorphism extending $g$ if and only if there exists a symplectic bundle isomorphism
\begin{equation*}
    G:\left(TM_0\Big\lvert_{N_0}, \w_0\right) \to \left(TM_1\Big\lvert_{N_1}, \w_1\right)
\end{equation*}
over $g$ which restricts to $Dg$ on $TN_0$. If we futher assume that $N_i$ are coisotropic, then the existence of such $G$ is equivalent to $g^*\w_1 = \w_0$ as forms on $N_0$~\cite{coisotropic_imbedding}.

These symplectic normal form theorems provide local models for neighbourhoods of submanifolds. In the literature, there are also some normal form theorems for neighbourhoods of specific singular subspaces; see \Cref{subsubsection:existing_normal_forms}. Such models enable explicit constructions and calculations. They are used in the study of Hamiltonian group actions, in the search for exotic Lagrangians (``Lagrangian knots'' \cite{first_steps_in_symplectic_top,polterovich2024lagrangian}), and more.

\subsection{Results}\label{subsection:introduction_results}

In this paper, we provide local normal form theorems for a broad class of singular subspaces.
This class consists of the \emph{stratified subspaces} (\Cref{definition:stratified_space}) that are \emph{smoothly locally trivial with conical fibers} (\Cref{definition:smoothly_locally_trivial}), introduced in \cite{zimhony2024commutative}. This condition controls singularities and can be checked locally. For a subset $A\subset M$ and a decomposition
\begin{equation*}
    A = \bigsqcup_{X\in \S}X \subset M
\end{equation*}
into \emph{strata} $X$, the condition is that for every stratum $X\in \S$ and every $p\in X$, a neighbourhood $U$ of $p$ in $M$ can be identified with $\R^k \times \R^{n-k}$ such that
\begin{itemize}
    \item $X\cap U$ is identified with $\R^k\times\{0\}$;
    \item For every stratum $Y\in \S$ with $\overline{Y}\supset X$, the intersection $Y\cap U$ is identified with $\R^k\times C_Y$ for a subset $C_Y\subset \R^{n-k}\setminus \{0\}$ invariant under multiplication by all $t\in(0,1)$.
\end{itemize}

Stratified subspaces that are smoothly locally trivial with conical fibers include all the examples in \Cref{subsubsection:existing_normal_forms} below, as well as zero level sets of momentum maps $\mu$ for compact group actions and, more generally, components of the critical set of $\norm{\mu}^2$ for such~$\mu$ (\cite[Section~7]{zimhony2024commutative}).

Throughout this introduction, all our stratified subspaces are implicitly assumed to satisfy this condition.
Throughout the paper, we endow these subspaces with the induced subspace differential structure in the sense of Sikorski. See \Cref{section:stratified_differential_spaces} for formal definitions.

Our main results are \Cref{thm:main_coisotropic,thm:main_symplectic}. \Cref{thm:main_coisotropic} is a symplectic normal form theorem around \emph{strongly coisotropic} stratified subspaces.
See \Cref{section:stratified_differential_spaces} for the definitions of the terms that appear in its statement.

\begin{mainTheorem}\label{thm:main_coisotropic}
    For $i=0,1$, let $\left(M_i, \w_i\right)$ be symplectic manifolds and $\left(A_i, \S_i\right)$ stratified subspaces of them, smoothly locally trivial with conical fibers.
    Let $g:\left(A_0, \S_0\right) \to \left(A_1, \S_1\right)$ be a stratified diffeomorphism with $g^*\w_1 = \w_0$ as Zariski forms on $A_0$.
    Assume that $A_i$ are strongly coisotropic.  
    Then there exist neighbourhoods $U_i\subset M_i$ of $A_i$ for $i=0,1$ and a symplectomorphism
    \begin{equation*}
        \G:(U_0, \w_0) \to (U_1, \w_1)
    \end{equation*}
    such that $\G$ restricts on $A_0$ to a stratified diffeomorphism
    \begin{equation*}
        \G\Big\lvert_{A_0}:\left(A_0, \S_0\right) \to \left(A_1, \S_1\right).
    \end{equation*}
    Moreover, this restriction is isotopic to $g:A_0 \to A_1$ through a smooth family of stratified diffeomorphisms
    \begin{equation*}
        g_t:\left(A_0, \S_0\right) \to \left(A_1, \S_1\right) \quad t\in [0,1],
    \end{equation*}
    with $g_t^*\w_1 = \w_0$ as Zariski forms on $A_0$ for all $t\in [0,1]$.
\end{mainTheorem}
Note that the symplectomorphism $\G$ need not restrict to the given stratified diffeomorphism $g$, but rather to a stratified diffeomorphism isotopic to $g$. This is a limitation of the proof, and it remains unclear whether one can construct a symplectomorphism whose restriction is exactly $g$. 
One can arrange the isotopy to be supported in an arbitrarly small neighbourhood of the singular part of $(A_0, \S_0)$. Thus, \Cref{thm:main_coisotropic} recovers the classical neighbourhood theorem for coisotropic submanifolds \cite{coisotropic_imbedding}.

For stratified subspaces $A_0$ and $A_1$ that are not strongly coisotropic, we require the additional data of a \emph{symplectic tangent bundle isomorphism} ${G:TM_0\Big\lvert_{A_0} \to TM_1\Big\lvert_{A_1}}$ over $g$ (\Cref{definition:bundle_isomorphism_over_g}) such that the pair $(g,G)$ is \emph{locally extendable} (\Cref{definition:extendable_pairs}).

\begin{mainTheorem}\label{thm:main_symplectic}
    For $i=0,1$, let $\left(M_i, \w_i\right)$ be symplectic manifolds and $\left(A_i, \S_i\right)$ stratified subspaces of them, smoothly locally trivial with conical fibers.
    Let $g:\left(A_0, \S_0\right) \to \left(A_1, \S_1\right)$ be a stratified diffeomorphism and let
    \begin{equation*}
        G:\left(TM_0\Big\lvert_{A_0}, \w_0\right) \to \left(TM_1\Big\lvert_{A_1}, \w_1\right)
    \end{equation*}
    be a symplectic tangent bundle isomorphism over $g$, such that the pair $(g, G)$ is locally extendable.
    Then there exist neighbourhoods $U_i\subset M_i$ of $A_i$ for $i=0,1$ and a symplectomorphism
    \begin{equation*}
        \G:(U_0, \w_0) \to (U_1, \w_1)
    \end{equation*}
    with $\G\big\lvert_{A_0}$ = $g$.
\end{mainTheorem}

The assumptions that $G$ is a symplectic tangent bundle isomorphism over $g$ and that $(g,G)$ is locally extendable imply that $g^*\w_1 = \w_0$ as Zariski forms on $A_0$, which is explicit in \Cref{thm:main_coisotropic}. When the stratified subspaces $A_i$ are submanifolds, local extendability is equivalent to $G\Big\lvert_{TA_0} = Dg$ (see \Cref{remark:locally_extendable_analog}), and \Cref{thm:main_symplectic} recovers the classical result~\cite[Lecture~5]{weinstein_lectures}. For general stratified subspaces, the condition $G\Big\lvert_{TA_0} = Dg$ is weaker than local extendability.

The proof of \Cref{thm:main_symplectic} relies on the following (non-symplectic) analogue of the tubular neighbourhood theorem for general closed subspaces.

\begin{mainTheorem}\label{thm:main_diffeomorphism}
    For $i=0,1$, let $M_i$ be manifolds and $A_i \subset M_i$ closed subspaces of them. Let $g:A_0 \to A_1$ be a diffeomorphism and let
    \begin{equation*}
        G:TM_0\Big\lvert_{A_0} \to TM_1\Big\lvert_{A_1}
    \end{equation*}
    be a tangent bundle isomorphism over $g$, such that the pair $(g, G)$ is locally extendable.

    Then there exist neighbourhoods $U_i\subset M_i$ of $A_i$ for $i=0,1$ and a diffeomorphism
    \begin{equation*}
        \G:U_0 \to U_1
    \end{equation*}
    which satisfies
    \begin{itemize}
        \item $\G\big\lvert_{A_0}$ = $g$ and
        \item $D\G$ agrees with $G$ along $A_0$, i.e., on $TM_0\Big\lvert_{A_0}$.
    \end{itemize}
\end{mainTheorem}

There is an interesting difference between the results of \Cref{thm:main_symplectic,thm:main_diffeomorphism}. In \Cref{thm:main_diffeomorphism} we obtain $D\G\big\lvert_{A_0} = G$, whereas in \Cref{thm:main_symplectic} this is
not guaranteed. The loss of this property happens because our version of Moser's trick, \Cref{thm:moser_trick}, uses a weak deformation retraction instead of a strong one. In this weak version, the differential of the resulting symplectomorphism might fail to restrict to the identity along the stratified subspace. See \Cref{example:weak_moser_trick_differential} for this failure and \Cref{appendix:strong_deformation_retraction_moser} for a proof that in the strong case such failure does not occur.

\begin{mainTheorem}\label{thm:moser_trick}
    Let $V$ be a manifold, $A\subset V$ a closed subset, and $R:[0,1]\times V \to V$ a smooth weak deformation retraction from $V$ to $A$. Let $\w_0, \w_1$ be symplectic forms on $V$ that agree on $TM\Big\lvert_A$.
    Then there exist neighbourhoods $U',U'' \subset V$ of $A$ and a diffeomorphism
    \begin{equation*}
        \G:U'\to U''
    \end{equation*}
    which fixes $A$ and such that $\G^*\w_1 = \w_0$.
\end{mainTheorem}

We refer to the neighbourhoods constructed in  \Cref{thm:main_coisotropic,thm:main_symplectic} as Weinstein neighbourhoods, by analogy with the case of Lagrangian submanifolds. It would be interesting to know whether they admit the structure of a Weinstein, or even Liouville, domain; see \Cref{question:Liouville_domain}.
\Cref{thm:exactness_of_symplectic_form} provides the first step toward adressing this question: that if the stratified subspace is \emph{isotropic}, i.e., its strata are isotropic, it has a neighbourhood in which the symplectic form is exact.

\Cref{lemma:lagrangian_stratified_equivalence} shows that a stratified subspace $A$ is Lagrangian (i.e., isotropic and \linebreak coisotropic; see \Cref{definition:isotropic_and_coisotropic}) if and only if the union of its Lagrangian strata is dense in $A$.
If, in addition, $A$ is strongly coisotropic, then a neighbourhood of $A$ is exact and is determined, up to a symplectomorphism preserving $A$, by the intrinsic differential structure on $A$ and the restriction of the symplectic form to the Zariski tangent over the singular part (\Cref{thm:main_coisotropic,thm:exactness_of_symplectic_form}).
This is analogous to Weinstein's Lagrangian neighbourhood theorem \cite{WEINSTEIN1971329} for neighbourhoods of Lagrangian submanifolds.

\subsection{Applications}

\subsubsection{Existing normal form theorems}\label{subsubsection:existing_normal_forms}

In the literature, there are symplectic normal form theorems around the following singular subspaces:
\begin{enumerate}
    \item\label{example:nodel_fiber} A Lagrangian 2-sphere with one positive self-intersection,
    used in \cite{MR1852770} to construct generalized symplectic rational blowdowns.
    \item\label{item:2_mutation_conf} A \emph{Mutation configuration}, which is a Lagrangian $2$-torus with an attached Lagrangian disc
    \cite[Lemma~4.11]{PASCALEFF2020106850}, used in~\cite[Section~4]{PASCALEFF2020106850} to analyze wall-crossings of potential functions of monotone Lagrangian $2$-tori.
    \item\label{example:lagrangian_pinwheel} \emph{Good Lagrangian pinwheels}~\cite[Lemma~3.4]{symplectic_rational_blowup}, used in~\cite{brendel2024semi} to construct exotic tori in non-monotone $4$-manifolds.
    \item\label{example:solid_mutation_configuration} \emph{Solid mutation configurations}~\cite[Lemma~2.4]{solid_mutation_conf}, which are generalizations of \cref{item:2_mutation_conf} to higher dimensions, used in~\cite{solid_mutation_conf} to construct infinitely many exotic montone Lagrangian tori in projective spaces.
    \item\label{example:conical_mutation_configurations} \emph{Conical mutation configurations}~\cite[Lemma~6.11]{chanda2024bohrsommerfeld}, which are another generalizations of \cref{item:2_mutation_conf}, used in \cite[Section~6.2]{chanda2024bohrsommerfeld} to develop a generalization of Lagrangian disc surgeries to higher dimensions.
    \item\label{example:Ac_buildings} Ac-buildings \cite[Section~3.3]{alvarezgavela2022arborealization}, which are certain singular subspaces with \emph{arboreal singularities}.
\end{enumerate}
The above examples are all Lagrangian, see below. There are also symplectic normal form theorems around some non-Lagrangian singular subspace, for example, chains of transversely intersecting symplectic spheres in a symplectic $4$-manifold \cite{symplectic_crossing_divisors}.

All the singular subspaces described above admit stratifications that are smoothly locally trivial with conical fibers. For example, a Lagrangian pinwheel can be decomposed into its core circle and an open 2-disc, and a mutation configuration can be decomposed into the intersection of the torus and the disk and the two complements. We verify in \Cref{proposition:L_dpq_stratified_diffeo,proposition:solid_mutation_stratification} that these stratifications are smoothly locally trivial with conical fibers.

The stratifications in Examples \cref{example:nodel_fiber}--\cref{example:Ac_buildings} are isotropic. Those in Examples \cref{example:nodel_fiber}--\cref{example:conical_mutation_configurations} are strongly coistropic, by \Cref{remark:depth_1_strongly_coisotropic}. It would be interesting to check if Example~\cref{example:Ac_buildings} is strongly coisotropic too.

To obtain a normal form for Examples \cref{example:nodel_fiber}--\cref{example:conical_mutation_configurations} using \Cref{thm:main_coisotropic}, it is enough to construct a stratified diffeomorphism $g:A_0\to A_1$ such that, near the singular part, we have $g^*\w_1 = \w_0$ as Zariski forms on $A_0$ (the Lagrangian assumption implies both forms vanish on the regular part; see \Cref{remark:g_pullback_matters_on_singular_part}). In \Cref{section:explicit_examples}, we illustrate the construction for Examples~\cref{item:2_mutation_conf}, \cref{example:lagrangian_pinwheel}, and \cref{example:solid_mutation_configuration}.

A remark about Example~\cref{example:conical_mutation_configurations}: consider a conical mutation configuration in $(M^{2n}, \w)$, consisting of a smooth Lagrangian and a cone over a Legendrian torus $\T^{n-1}\subset S^{2n-1}$, which intersect cleanly along $\T^{n-1}$. Pascaleff and Tonkonog \cite[Remark~5.7]{PASCALEFF2020106850} discuss a symplectic invariant of a Weinstein neighbourhood of this configuration --- the Legendrian isotopy class of the link $\T^{n-1}\subset S^{2n-1}$ of the conical sigularity. 
Our condition $g^*\w_1 = \w_0$ as Zariski forms on $A_0$ detects this invariant: the differential $D_{p_0}g:T^Z_{p_0}A_0 \to T^Z_{g(p_0)}A_1$ of $g$ at the cone point $p_0$ extends to a symplectic linear map $T_{p_0}M_0 \to T_{g(p_0)}M_1$ carrying one Legendrian link to the other.

\subsubsection{New normal form theorems}

Brendel, Hauber, and Schmitz~\cite{brendel2024semi} define certain model spaces $B_{dpq}$, and use them to construct exotic tori for symplectic manifolds $M$ that admit a symplectic embedding $B_{dpq} \hookrightarrow M$. The Lagrangian skeleton of $B_{dpq}$, denoted by $L_{dpq}$, is a chain starting with a Lagrangian $(p,q)$-pinwheel and followed by $d-1$ Lagrangian spheres, such that every member of the chain intersects its neighbours transversely in one point. 
By \Cref{cor:L_dpq_determines_neigbourhood} of \Cref{thm:main_coisotropic}, if $L_{dpq}$ embeds in $M$, then $B_{dpq}$ symplectically embeds in $M$.
Therefore, \cite[Theorem~B]{brendel2024semi} holds under the (a priori weaker) assumption $L_{dpq} \subset M$.

More generally, almost toric fibrations of symplectic $4$-manifolds~\cite{almost_toric} provide families of singular Lagrangian subspaces, often referred to as ``\emph{visible Lagrangians}'' \cite[Chapter~5]{Evans_book_2023}, and model neighbourhoods of them. Examples include
\begin{enumerate}
    \item Lagrangian $2$-tori (regular fibers);
    \item Lagrangian 2-spheres with one positive self-intersection (nodal fibers);
    \item Lagrangian pinwheels (``hitting an edge'' \cite[Section~5.3]{Evans_book_2023});
    \item mutation configurations and Lagrangian pinwheels with attached spheres can be described as unions of ``visible Lagrangians''.
\end{enumerate}

In a general symplectic manifold, given a singular Lagrangian subspace that is intrinsically diffeomorphic to a ``visible Lagrangian'' $L$ in some almost toric symplectic manifold, we can use the symplectic normal form theorem to identify its neighbourhood with a model neighbourhood of $L$. Within the model neighbourhood, ``almost toric manipulations'' such as node sliding can be performed, then pulled back to the original manifold. This generalizes the construction in \cite{brendel2024semi}, where $L=L_{dpq}$, and \cite[Section~4]{PASCALEFF2020106850}, where $L$ is a mutation configuration.

Pascaleff and Tonkonog \cite{PASCALEFF2020106850}
give a wall-crossing formula relating the disc potentials of a monotone Lagrangian torus $L_1$ and its mutation $L_2$. Their main tool is \cite[Theorem~1.1]{PASCALEFF2020106850}, which reduces the question of matching their disc potentials in $M$ to the question of a non-vanishing Lagrangian Floer cohomology within a Liouville domain $U\subset M$ containing both Lagrangians, such that $L_1,L_2\subset U$ are exact.
Contrariwise, Chanda~\cite{chanda2024bohrsommerfeld} proves a wall-crossing formula using neck stretching techniques. In \Cref{question:Liouville_domain} we ask whether our construction, around Lagrangian stratified subspaces, admits a structure of a Liouville domain. In \Cref{question:Lagrangian_skeleton} we furthermore ask whether it admits a structure of a Liouville domain whose Lagrangian skeleton is the given Lagrangian subspace. Positive answers to these questions would render the technique of \cite[Theorem~1.1]{PASCALEFF2020106850} applicable to our neighbourhoods. For conical mutation configurations, this might enable an alternative proof for the wall-crossing formula of Chanda. Moreover, the technique might give rise to wall-crossing type statements for the ``almost toric manipulations'' discussed above.

\subsection{Outline}
In \Cref{section:stratified_differential_spaces}, we give detailed background on the geometry of stratified differential spaces.
In \Cref{section:extending_bundle_iso_to_diffeo}, we prove \Cref{thm:main_diffeomorphism}.
In \Cref{section:moser_trick_proof}, we prove \Cref{thm:moser_trick,thm:main_symplectic}.
In \Cref{section:coisotropic_theorem}, we prove \Cref{thm:main_coisotropic}. The proof uses a technical Lemma on Euler-like vector fields, proved in \Cref{appendix:euler_likes}.
In \Cref{section:exactness_for_isotropic}, we prove that the symplectic form is exact near isotropic stratified subspaces.
In \Cref{section:explicit_examples}, we prove that the hypotheses of \Cref{thm:main_coisotropic} hold for two examples of singular Lagrangian subspaces.
In \Cref{appendix:strong_deformation_retraction_moser}, we address Moser's trick with a strong deformation retraction, see \Cref{remark:mosers_trick_discussion}.

\subsection{Acknowledgements}
We thank Rei Henigman for useful comments on the introduction, Jo\'e Brendel for useful conversations, and Soham Chanda for pointing out useful references to us. We also thank the anonymous referee for detailed and useful comments.
Y.K. is partly supported by the United States -- Israel Binational Science Foundation Grant 2021730.
S. T. and Y. Z. are partly supported by the Israel Science Foundation Grant 2793/21.

\section{Stratified differential spaces}\label{section:stratified_differential_spaces}

\subsection{Differential spaces}

We recall some definitions; for more details see \cite{vector_fields_subcartesian, lusala_sniatycki_2011, sniatycki_book_2013}.

\begin{definition}
    Let $A$ be a topological space. A \textbf{differential structure} on $A$ is a non-empty set $C^\infty(A)$ of continuous functions $A\to \R$ satisfying
    \begin{enumerate}
        \item\label{condition:differential_subbasis} $C^\infty(A)$ induces the given topology on $A$;
        \item for all $F:\R^n\to \R$ smooth and $f_1,\ldots f_n\in C^\infty(A)$, the composition $F(f_1,\ldots,f_n)$ is in $C^\infty(A)$;
        \item if $f:A \to \R$ is a continuous function such that for all $a\in A$ there exist a neighbourhood $U_a$ and $g_a\in C^\infty(A)$ with $f\big\lvert_{U_a} = g_a\big\lvert_{U_a}$, then $f\in C^\infty(A)$.
\end{enumerate}

    A topological space with a differential structure is a \textbf{differential space}.
\end{definition}

\begin{definition}
    Let $A$ be a topological space and let $\mathcal{F}$ be a non-empty family of continuous functions $A\to \R$ that induces the given topology on $A$.
    The \textbf{differential structure generated by $\mathbf{\mathcal{F}}$} consists of those continuous functions $g:A\to \R$ such that for every $p\in A$, there exist a neighbourhood $U$ of $p$, a smooth function $F\in C^\infty\left(\R^n\right)$, and $f_1,\ldots,f_n \in \mathcal{F}$ such that
    \begin{equation*}
        g\Big\lvert_U = {F(f_1, \ldots, f_n)}\Big\lvert_U.
    \end{equation*}
\end{definition}

\begin{definition}
    Let $A,B$ be differential spaces. A continuous function $\varphi:A\to B$ is \textbf{smooth} if for all $f\in C^\infty(B)$, the pullback $\varphi^*f:A\to \R$ satisfies $\varphi^*f\in C^\infty(A)$.
    It is a \textbf{diffeomorphism} if it is smooth, bijective, and its inverse is smooth.

    We denote by $C^\infty(A,B)$ the set of all smooth functions from $A$ to $B$.
\end{definition}

\begin{definition}
    Let $B$ be a differential space and $A\subset B$ a subset. The \textbf{subspace differential structure} on $A$ (with its subspace topology) is the differential structure generated by restrictions to $A$ of functions in $C^\infty(B)$. We say that $A$ is a \textbf{differential subspace of~$\mathbf{B}$}.
\end{definition}

\begin{definition}
    A map $A\to B$ is a \textbf{smooth embedding} if it is a diffeomorphism of $A$ with a differential subspace of $B$.
\end{definition}

Throughout this paper, we implicitly consider manifolds as differential spaces, and consider their subsets as differential subspaces.

\begin{definition}
    Let $A$ be a differential space and let $p\in A$. A \textbf{derivation at $\mathbf{p}$} is a linear map $v:C^\infty(A)\to \R$ satisfying
    \begin{equation*}
        v(fg) = f(p)v(g) + g(p)v(f).
    \end{equation*}
\end{definition}

\begin{definition}
    The set of all derivations at $p$ is called the \textbf{Zariski tangent of $A$ at $p$} and denoted by $T^Z_pA$. It is a vector space over $\R$.
\end{definition}

Write
\begin{equation*}
    T^ZA \coloneqq \bigcup_{p\in A} T^Z_pA.
\end{equation*}
We endow $T^ZA$ with a differential structure as follows. Let $\tau_A:T^ZA\to A$ be the projection
\begin{equation*}
    \tau_A(v) = p, \qquad \forall v\in T^Z_pA
\end{equation*}
and for all $f\in C^\infty(A)$ define $df:T^ZA\to \R$ by
\begin{equation*}
    df(v) = v(f).
\end{equation*}

\begin{definition}[{\cite[Definition~3.3.2]{sniatycki_book_2013}}]\label{definition:Zariski_tangent}
    The \textbf{Zariski tangent $T^ZA$ of $A$} is the differential space given by the set $T^ZA$ with the topology and the differential structure generated by
    \begin{equation*}
        \{f\circ \tau_A:\: f\in C^\infty(A)\}\ \cup\ \{df:\: f\in C^\infty(A)\}.
    \end{equation*}
\end{definition}

\begin{definition}
    Let $A,B$ be differential spaces, $\varphi:A\to B$ a smooth map, and $p\in A$.
    The \textbf{differential} $D^Z_p\varphi:T^Z_pA\to T^Z_{\varphi(p)}B$ of $\varphi$ at $p$ is defined by
    \begin{equation*}
        {D^Z_p\varphi}(v)(f) = v(f\circ \varphi), \qquad \forall v\in T^Z_pA,\, f\in C^\infty(B).
    \end{equation*}
    It is a linear map.
\end{definition}

\begin{proposition}[{\cite[Proposition~3.3.7]{sniatycki_book_2013}}]
    The differential $D^Z\varphi:T^ZA \to T^ZB$ of $\varphi$ is a smooth map, and the following diagram commutes
    \[\begin{tikzcd}[ampersand replacement=\&]
        {T^ZA} \& {T^ZB} \\
        A \& B
        \arrow["D^Z\varphi"', from=1-1, to=1-2]
        \arrow["{\tau_A}", from=1-1, to=2-1]
        \arrow["{\tau_B}"', from=1-2, to=2-2]
        \arrow["\varphi", from=2-1, to=2-2]
    \end{tikzcd}\]
\end{proposition}

\begin{proposition}\label{proposition:tangent_space_embedding}
    Let $i:A\to B$ be an embedding of differential spaces. Then the smooth map $D^Zi:T^ZA \to T^ZB$ is also an embedding of differential spaces.
\end{proposition}
\begin{proof}
    We need to show that $D^Zi$ is injective and that the family ${(D^Zi)}^*\left(C^\infty(T^ZB)\right)$ generates $C^\infty(T^ZA)$.

    For injectivity, it is enough to show that $D^Z_pi$ has a trivial kernel at every $p\in A$. Let $v\in T^Z_pA$ with $D^Z_pi(v)=0\in T^Z_{i(p)}B$. Then for all $f\in C^\infty(B)$ we have
    \begin{equation*}
        0 = D^Z_pi(v)(f) = v(f\circ i).
    \end{equation*}
    Since $C^\infty(A)$ is generated by functions of the form $\tilde{f} = f\circ i$ for $f\in C^\infty(B)$, we get $v=0$.

    Recall that $C^\infty(T^ZA)$ is generated by the functions $\tilde{f}\circ \tau_A$ and $d\tilde{f}$ for $\tilde{f}\in C^\infty(A)$. Since $C^\infty(A)$ is generated by the functions $\tilde{f} = f\circ i$ for $f\in C^\infty(B)$, it is enough to show that the image of ${\left(D^Zi\right)}^*$ contains $f\circ i \circ \tau_A$ and $d(f\circ i)$ for all $f\in C^\infty(B)$. Indeed,
    \begin{equation*}
        f\circ i \circ \tau_A = f\circ \tau_B \circ D^Zi = {(D^Zi)}^*(f\circ \tau_B)
    \end{equation*}
    and
    \begin{align*}
        d(f\circ i)(v)
        =& v(f\circ i) \\
        =& D^Zi(v)(f) \\
        =& df\left((D^Zi)(v)\right) \\
        =& {(D^Zi)}^*(df)(v).
    \end{align*}
\end{proof}

\begin{remark}
    Let $A_i\subset M_i,\,i=0,1$ with $M_i$ manifolds and $A_i$ differential subspaces. By \Cref{proposition:tangent_space_embedding}, we can view $T^ZA_i\subset TM_i\Big\lvert_{A_i}\subset TM_i$ as differential subspaces of the tangent bundle. For a smooth map $g:A_0\to A_1$, its differential $D^Zg:T^ZA_0\to T^ZA_1$ is then a fiberwise-linear smooth map of subspaces of $TM_i$.
\end{remark}

\begin{definition}[{\cite[Definition~5.2.1]{sniatycki_book_2013}}]\label{definition:zariski_forms}
    Let $A$ be a differential space. A \textbf{Zariski $\mathbf{k}$-form on $A$} is an alternating $k$-linear smooth map from $T^ZA$ to $\R$.
\end{definition}

Zariski forms admit pullbacks by smooth maps, defined in the same way as pullbacks of forms for manifolds. See \cite[Equation~(5.6)]{sniatycki_book_2013}. For a manifold $M$ and a differential subspace $A\subset M$  that is closed in $M$, the Zariski forms on $A$ are exactly the pullbacks of forms on $M$ by the inclusion \cite[Proposition~5.2.2]{sniatycki_book_2013}.

\subsection{Stratified differential spaces}

For a subset $A\subset B$ of a topological space, we denote its closure by $\overline{A}$.

\begin{definition}\label{definition:stratified_space}
    Let $A$ be a differential space. A \textbf{stratification} of $A$ is a decomposition
    \begin{equation*}
        A = \bigsqcup_{X\in \S}X,
    \end{equation*}
    whose elements are called \textbf{strata},
    with the following properties.
    \begin{enumerate}
        \item Each $X\in \S$, with its subspace differential structure, is a smooth manifold.
        \item The decomposition is locally finite.
        \item \emph{Condition of the frontier} --- for $X,Y\in\S$, if $X\cap \overline{Y} \neq \emptyset$ then $X\subset \overline{Y}$. Equivalently:
        \begin{itemize}
            \item[(3')] The closure of each stratum is a union of strata.
        \end{itemize}
    \end{enumerate}
    A \textbf{stratified differential space} is the data $\left(A, \S\right)$.
\end{definition}

\begin{definition}
    Let $X,Y\in\S$ and define the order relation by setting $X<Y$ when $X\subset \overline{Y}$ and $X\neq Y$. We say $X,Y$ are \textbf{comparable} when $X<Y$ or $Y<X$.
\end{definition}

\begin{proposition}\label{proposition:induced_decomposition_is_stratification}
    Let $\left(A, \S\right)$ be a stratified differential space and let $U\subset A$ be an open subset. Let
    \begin{equation*}
        S_U = \{X\cap U: \forall X\in \S \text{ with }X\cap U\neq \emptyset\}.
    \end{equation*}
    Then $\left(U, S_U\right)$ is a stratified differential space.
\end{proposition}
\begin{proof}
    We prove that the condition of the frontier holds, the other conditions are immediate. Let $Y\in \S$ with $Y\cap U\neq \emptyset$. Let $X\in \S$, and assume that $X\cap U$ intersects the closure of $Y\cap U$ in $U$. By the condition of the frontier in $A$, we get $X\subset \overline{Y}$, where the closure is taken in $A$. For any $p\in X\cap U$ and any neighbourhood $W$ of $p$ in $U$, since $U$ is open in $A$, the subset $W$ is also open in $A$ and therefore intersects $Y$. Since $W\subset U$, it intersects $Y\cap U$. It follows that $p$ is in the closure of $Y\cap U$ in $U$.
\end{proof}

\begin{definition}\label{definition:induced_stratification}
    With the notation of \Cref{proposition:induced_decomposition_is_stratification}, the stratification $S_U$ of $U$ is \textbf{the induced stratification of $\mathbf{U\subset A}$}.
\end{definition}

\begin{definition}
    Let $\left(A_i, \S_i\right),\,i=0,1$ be stratified differential spaces. A \textbf{stratified diffeomorphism} between them consists of
    \begin{enumerate}
        \item a bijection $s:\S_0\to \S_1$ and
        \item a diffeomorphism $g:A_0\to A_1$,
    \end{enumerate}
    such that $g$ restricted to every $X_0\in\S_0$ is a diffeomorphism onto $s(X_0)$.
\end{definition}

\begin{definition}
    Let $M$ be a manifold. A \textbf{stratified subspace} of $M$ is a closed subspace of $A\subset M$ and a stratification $\S$ of it.
\end{definition}

One often adds additional regularity assumptions to the definition of stratified subspaces of manifolds, such as Whitney~(B) regularity (see \cite{mather2012notes}). A desirable property given by such regularity assumptions is that $X<Y$ implies $\dim{X}<\dim{Y}$. We use a stronger regularity condition, given in \Cref{definition:smoothly_locally_trivial}, which implies Whitney~(B) and therefore also the monotonicity of dimension.

\begin{notation}\label{notation:manifold_dimension}
    If $N$ is a manifold, the notation $N^k$ means that $\dim{N} = k$.
\end{notation}

\begin{definition}
    A subset of a vector space is \textbf{conical} if it is invariant under multiplication by $t\in (0,1]$. A subset of a vector bundle is conical if its intersection with each fiber is conical.
\end{definition}

\begin{definition}[{\cite[Definition~3.6]{zimhony2024commutative}}]\label{definition:smoothly_locally_trivial}
    Let $M^n$ be a manifold.
    A stratified subspace $(A, \S)$ of $M^n$ is \textbf{smoothly locally trivial with conical fibers} if the following holds for each $X^k\in \S$ and $p\in X$.

    There exist a neighbourhood $U_p\subset M$ of $p$, a neighbourhood $W_p\subset \R^k\times \R^{n-k}$ of the origin, and a diffeomorphism
    \begin{equation*}
        \theta_p:U_p\to W_p
    \end{equation*}
    such that
    \begin{enumerate}
        \item $\theta_p(X\cap U_p) = \left(\R^k\times \{0\}\right)\cap W_p$, and
        \item for each $Y\in\S$ with $X<Y$ there exists a conical subset $Y_p\subset \R^{n-k}\setminus\{0\}$ such that
        \begin{equation*}
            \theta_p(Y\cap U_p) = \left(\R^k \times Y_p\right) \cap W_p.
        \end{equation*}
    \end{enumerate}
    A \textbf{conical chart} of $p$ is such data $(U_p, W_p, \theta_p)$.
\end{definition}

\begin{remark}
    Let $(A, \S)$ be a stratified subspace of $M$, let $p\in A$, and let $(U_p, W_p, \theta_p)$ be a conical chart of $p\in A$. Then, identifying $T_p\left(\R^k \times \R^{n-k}\right)$ with $\R^k \times \R^{n-k}$, we have
    \begin{equation*}
        T^Z_pA = D\theta_p^{-1}\left(\R^k \times \vectorspan\left(\bigcup_{Y\in \S}Y_p\right)\right).
    \end{equation*}
\end{remark}

\begin{proposition}\label{proposition:also_whitney_b}
    If a stratified subspace $(A,\S)$ of $M$ is smoothly locally trivial with conical fibers, then it is Whitney~(B) regular \cite[Definition~2.3]{mather2012notes}.
\end{proposition}
\begin{proof}
    Let $X,Y\in \S$, let $p\in X$, and let $\theta_p:U_p\to W_p \subset \R^k\times \R^{n-k}$ be a conical chart of $p$. 
    Let
    $x_i \in \left(\R^k\times \{0\}\right)\cap W_p$
    and
    $y_i \in \left(\R^k \times Y_p\right) \cap W_p$
    be sequences of points converging to $0$, and assume that $T_{y_i}\left(\R^k \times Y_p\right)$ converge to a plane $F\subset T_0\R^n$ and the tangent spaces of the lines through $y_i$ and $x_i$ converge to a line $l\subset T_0\R^n$. We now prove that $l\subset F$.

    Denote by $\pi:\R^k \times \R^{n-k} \to \R^k$ the projection to $\R^k$. Since $Y_p$ is conical, the line through $y_i$ and $\pi(y_i)$ is contained in $\R^k \times Y_p$. Since the line through $\pi(y_i)$ and $x_i$ is contained in $\left(\R^k \times \{0\}\right)$, the line through $y_i$ and $x_i$ is contained in $T_{y_i}\left(\R^k \times Y_p\right)$. It follows that $l\subset F$, and therefore \cite[Definition~2.1]{mather2012notes} is satisfied in the chart $\theta_p:U_p\to W_p$.
\end{proof}

\begin{remark}
    A stratified subspace that is Whitney~(B) regular is \emph{topologically locally trivial} \cite[Corollary~10.6]{mather2012notes}, but not necessarily smoothly locally trivial. It is also Whitney~(A) regular \cite[Proposition~2.4]{mather2012notes}.
\end{remark}

\begin{notation}\label{notation:strata_skeleton}
    Let $(A,\S)$ be a stratified subspace of a manifold $M$. We write
    \begin{align*}
    \S^{<d} &\coloneq \{X\in \S: \dim{X}<d\}, \\
    A^{<d} &\coloneq \bigsqcup_{X\in\S^{<d}}X \\
    \end{align*}
    and similarly for $\S^d, \S^{\geq d}$ and $A^{\geq d}$. Also,
    \begin{equation*}
        M^{\geq d} = M \setminus A^{<d}.
    \end{equation*}
\end{notation}

Recall that we require a stratified subspace $A$ to be closed in $M$.

\begin{remark}\label{remark:closedness_of_strata}
    If $X<Y$ implies $\dim{X}<\dim{Y}$, as in the case of a Whitney~(B) regular stratified subspace, then:
    \begin{enumerate}
        \item $A^{\leq d}$ are closed in $A$;
        \item $A^{\geq d}$ are open in $A$;
        \item $A^{=d}$ and $A^{\geq d}$ are closed in $M^{\geq d}$ and in $A^{\geq d}$;
        \item $A^{\leq d}$ are closed in $M$.
    \end{enumerate}
\end{remark}

The following Lemmas will be used in \Cref{section:coisotropic_theorem}.

\begin{lemma}\label{lemma:stratified_spaces_locally_diffeo}
    For $i=0,1$, let $M_i$ be manifolds of dimension $n$ and $\left(A_i, \S_i\right)$ stratified subspaces of them, smoothly locally trivial with conical fibers. Let
    \begin{equation*}
        g:\left(A_0, \S_0\right) \to \left(A_1, \S_1\right)
    \end{equation*}
    be a stratified diffeomorphism, $p_0\in A_0$, and $p_1 \coloneq g\left(p_0\right)$.

    Then there exist an open set $W\subset \R^n$ and conical charts (\Cref{definition:smoothly_locally_trivial})
    \begin{equation*}
        \theta_{p_0}: U_{p_0} \to W, \quad \theta_{p_1}: U_{p_1} \to W,
    \end{equation*}
    of $p_0, p_1$, such that $\theta_{p_1}^{-1} \circ \theta_{p_0}$ restricts to $g$ on $A_0 \cap U_{p_0}$.
\end{lemma}
\begin{proof}
    Let $\theta_{p_i}':U_{p_i} \to W_{p_i}$ be conical charts and $A_{p_i} \coloneq \theta_{p_i}'\left( A_i\cap U_{p_i}\right)$. Take a local extension
    \begin{equation*}
        g':\vectorspan \{ A_{p_0}\} \cap W_{p_0} \to \vectorspan \{ A_{p_1}\} \cap W_{p_1}
    \end{equation*}
    of $g$. Since $A_{p_i}$ is conical,  $\dim{T^Z_{p_i}A_i} = \dim{\vectorspan \{ A_{p_i}\}}$. Using the fact that $g$ is a diffeomorphism and the inverse function theorem, there exists a small neighbourhood of $0$ in $\vectorspan \{ A_{p_0}\} \cap W_{p_0}$ such that $g'$ restricted to this neighbourhood is a diffeomorphism. We can extend $g'$ to a diffeomorphism \begin{equation*}
        \widetilde{g}:W_{p_0} \to W_{p_1}
    \end{equation*}
    by adding linearly independent linear terms, vanishing on $\vectorspan \{ A_{p_0}\}$, and with target \linebreak
    ${\left(\vectorspan \{ A_{p_1}\}\right)^\perp}$, and possibly shrinking the neighbourhoods. After possibly shrinking the neighbourhoods again, the data ${W = W_{p_1}}$, ${\theta_{p_1} = \theta_{p_1}'}$ and $\theta_{p_0} = \widetilde{g} \circ \theta_{p_0}'$ satisfy the required properties.
\end{proof}

\begin{lemma}\label{lemma:tangential_vector_field_stratified_diffeo}
    Let $(A,\S)$ be a stratified subspace of a manifold $M$, let $\E_t\in \mathfrak{X}(M)$ be a time-dependent vector field, and let $t_0 > 0$. Assume that $A$ is contained in the domain and the range of the time-$t_0$ flow of $\E_t$. Assume that ${X<Y}$ implies ${\dim{X}<\dim{Y}}$, and that for all $X\in \S$ and $p\in X$ we have $\E_t(p)\in T_pX$.
    Then the time-$t_0$ flow $\Phi^{t_0}_{(\E_t)}$ restricts to a stratified diffeomorphism $A \to A$.
\end{lemma}
\begin{proof}
    We prove this by ascending induction on the dimension of strata. For dimension $0$, the subspace $A^0$ is a submanifold of $M^{\geq 0} = M$ that is closed in it. By assumption, for all $X^0\in \S^0$ and $p\in X^0$ we have $\E(p)\in T_pX^0=\{0\}$. It follows that $\Phi^{t_0}_{(\E_t)}$ fixes each $X^0\in \S^0$, and therefore fixes $A^0$.

    For dimension $d$, assume that $\Phi^{t_0}_{(\E_t)}$ restricts to a diffeomorphism $A^{\leq (d-1)} \to A^{\leq (d-1)}$, and therefore (since it bijective) sends $M \setminus A^{\leq (d-1)} = M^{\geq d}$ to itself. The subspace $A^d$ is a submanifold of $M^{\geq d}$ that is closed in it, and by assumption for all $X^d\in \S^d$ and $p\in X^d$ we have $\E(p)\in T_pX^d$. It follows that $\Phi^{t_0}_{(\E_t)}$ preserves each $X^d$, therefore preserves $A^d$, and thus preserves $A^{\leq d}$.
\end{proof}

\subsection{Isotropic and coisotropic stratified subspaces}
\begin{definition}\label{definition:isotropic_and_coisotropic}
    A stratified subspace $(A, \S)$ of a symplectic manifold $(M, \w)$ is \textbf{isotropic} if every $X\in \S$ is isotropic.
    It is \textbf{coisotropic} if for each $p\in A$, the Zariski tangent space $T^Z_pA \subset T_pM$ is coisotropic.
    It is \textbf{Lagrangian} if it is isotropic and coisotropic.
\end{definition}

\begin{remark}
    Note that, for a stratfied subspace $(A,\S)$ of $(M,\w)$, the isotropic condition is imposed stratum-by-stratum, while the coisotropic condition is imposed using the Zariski tangent $T_p^ZA$ of the whole stratified subspace, which sees nearby strata. There are two reasons for this asymmetry: important examples and the implication of each condition.

    Regarding important examples, consider Example \eqref{example:nodel_fiber} from \Cref{subsubsection:existing_normal_forms} --- a Lagrangian 2-sphere with one positive (transversal) self-intersection. 
    A natural stratification of this subspace is the partition to the self-intersecting point and the rest.
    With this stratification, each stratum is isotropic, but the stratum consisting of the self-intersecting point is not Lagrangian, because it is a single point. Furthermore, each Zariski tangent is coisotropic, but the Zariski tangent at the self-intersecting point is not Lagrangian, because it is all of $T_pM$.
    The same issue arises in clean intersections of Lagrangian submanifolds.

    Regarding the implications, analogously to the case where $A$ is a submanifold, we want the geometry of a coisotropic subspace $A\subset M$ to determine $TM\Big\lvert_A$ as a symplectic vector bundle, and we want the symplectic form to be exact in a neighbourhood of an isotropic subspace.
    For the coisotropic case, this happens exactly when $T^Z_pA \subset T_pM$ is coisotropic for each $p\in A$.
    For the isotropic case, the isotropic strata condition implies that for every smooth $r:N \to M$ with $r(N) \subset A$, we have $r^*\w = 0$. Given a neighbourhood $U$ and a smooth weak deformation retraction $r:U \to A$, the desired property follows as in the proof of \Cref{thm:exactness_of_symplectic_form}.
\end{remark}

\begin{lemma}\label{lemma:lagrangian_stratified_equivalence}
    Let $(M, \w)$ be a symplectic manifold and $(A, \S)$ a stratified subspace of it, smoothly locally trivial with conical fibers. Then $(A, \S)$ is Lagrangian if and only if the union of the Lagrangian strata is dense in $A$.
\end{lemma}
\begin{proof}
    Assume the union of the Lagrangian strata is dense in $A$. Whitney (A) regularity implies that every stratum is isotropic. For every $p\in A$, by the local model from \Cref{definition:smoothly_locally_trivial}, the Zariski tangent $T_p^ZA$ contains a Lagrangian plane and is therefore coisotropic.

    Assume $(A, \S)$ is isotropic and coisotropic. If the union of the Lagrangian strata is not dense in $A$, there exists a maximal $k$ and $X\in \S$ of dimension $k$ such that $X$ is not contained in the closure of any Lagrangian stratum. By the assumption that the strata are isotropic, maximality of $k$, and Whitney (B) regularity, there are no $Y\in \S$ with $X<Y$, and therefore for every $p\in X$ we have $T_p^ZA=T_pX$. Define $n\coloneq \frac{1}{2}\dim{M}$. Since $X$ is isotropic, ${k=\dim{X} \leq n}$. Since $X$ is not Lagrangian, $k<n$. We get a contradiction to the assumption that $T_p^ZA= T_pX$ is coisotropic.
\end{proof}

\begin{definition}\label{definition:strongly_coisotropic}
    A stratified subspace $(A, \S)$ of a symplectic manifold $(M, \w)$ is \textbf{strongly coisotropic} if it is coisotropic and for each non-closed stratum $X\in \S$ and $p\in X$ we have
    \begin{equation}\label{condition:strongly_coisotropic}
        {\left(T^Z_pA\right)}^{\w} \subset T_pX.
    \end{equation}
\end{definition}

\begin{example}
    Let $L_1, L_2$ be Lagrangian submanifolds of $(M, \w)$, intersecting cleanly along $N=L_1\cap L_2$. Let $A = L_1\cup L_2$ and $\S= \{N,\, L_1 \setminus N,\, L_2 \setminus N \}$. Then $(A, \S)$ satisfies Condition~\eqref{condition:strongly_coisotropic} along $N$.

    Indeed, for $p\in N$, the clean intersection assumption implies $T^Z_pA = T_pL_1 + T_pL_2$ and $T_pN = T_pL_1 \cap T_pL_2$. Therefore,
    \begin{equation*}
        {\left(T^Z_pA\right)}^{\w} = {\left(T_pL_1 + T_pL_2\right)}^{\w}
        = {\left(T_pL_1\right)}^{\w} \cap {\left(T_pL_2\right)}^{\w}
        = T_pL_1 \cap T_pL_2
        = T_pN.
    \end{equation*}
\end{example}

\begin{example}
    Let $A = \{x_1, x_2 \geq 0, y_1,y_2 = 0 \} \subset \left(\R^4, dx_1 \wedge dy_1 + dx_2\wedge dy_2\right)$ and
    \begin{equation*}
        \S = \Bigl\{\{0\}, \{x_1 > 0, x_2 =0, y_1,y_2 = 0\}, \{x_1 = 0, x_2 > 0, y_1,y_2 = 0\}, \{x_1, x_2 > 0, y_1,y_2 = 0\}\Bigr\}.
    \end{equation*}
    Then $(A, \S)$ is coisotropic but not strongly coisotropic --- the non-closed strata
    \begin{equation*}
        \{x_1 > 0, x_2 =0, y_1,y_2 = 0\}, \{x_1 = 0, x_2 > 0, y_1,y_2 = 0\}
    \end{equation*}
    do not satisfy the Condition~\eqref{condition:strongly_coisotropic}.
\end{example}

\begin{remark}\label{remark:depth_1_strongly_coisotropic}
    Let $(M,\w)$ be a symplectic manifold and let $(A, \S)$ be a coisotropic stratified subspace of it. Assume that every non-closed stratum is open in $A$. Then $(A, \S)$ is strongly coisotropic.

    Indeed, let $X$ be a non-closed stratum and let $p\in X$. Then $X$ is open in $A$ and therefore $T^Z_pA = T_pX$. The coisotropic assumption implies
    \begin{equation*}
        {\left(T^Z_pA\right)}^{\w} \subset T^Z_pA = T_pX.
    \end{equation*}
\end{remark}

\begin{remark}\label{remark:g_pullback_matters_on_singular_part}
    For $i=0,1$, let $(A_i,\S_i) \subset (M_i,\w_i)$ be stratified subspaces of symplectic manifolds, smoothly locally trivial with conical fibers. Assume that $A_i$ are Lagrangian, and let $g:A_0 \to A_1$ be a stratified diffeomorphism.
    Since $g^*\w_1 = \w_0 = 0$ on each Lagrangian stratum, the condition $g^*\w_1 = \w_0$ as Zariski forms on $A_0$ needs to be checked only at singular points of $A_0$.
\end{remark}

\section{Extending tangent bundle isomorphisms to diffeomorphisms}\label{section:extending_bundle_iso_to_diffeo}
\subsection{Extendable and locally extendable pairs}
For $i=0,1$, let $A_i \subset M_i$ be subspaces of manifolds, with the induced subspace differential structure. Let $g:A_0 \to A_1$ be a smooth map.
\begin{definition}
    A \textbf{tangent bundle morphism over $\mathbf{g}$} is a map
    \begin{equation*}
        G:TM_0\Big\lvert_{A_0} \to TM_1\Big\lvert_{A_1}
    \end{equation*}
    such that:
    \begin{enumerate}
        \item it is smooth as a map of differential subspaces $TM_i\Big\lvert_{A_i} \subset TM_i, \: i=0,1$;
        \item $G$ sends $T_{a_0}M_0$ linearly to $T_{g(a_0)}M_1$ for all $a_0 \in A_0$.
    \end{enumerate}
\end{definition}

\begin{definition}\label{definition:bundle_isomorphism_over_g}
    Let $G:TM_0\Big\lvert_{A_0} \to TM_1\Big\lvert_{A_1}$ be  a tangent bundle morphism over a diffeomorphism $g:A_0 \to A_1$. Then $G$ is a tangent bundle \textbf{isomorphism} over $g$ if the linear maps $G\Big\lvert_{T_{a_0}M_0}: T_{a_0}M_0 \to T_{g(a_0)}M_1$ are invertible for all $a_0 \in A_0$.

    Let $\w_0, \w_1$ be symplectic forms on $M_0, M_1$ respectively. Then $G$ is a \textbf{symplectic tangent bundle isomorphism} over $g$ if is a tangent bundle isomorphism over $g$ and $G^*\w_1  = \w_0$.
\end{definition}

\begin{remark}
    If $G:TM_0\Big\lvert_{A_0} \to TM_1\Big\lvert_{A_1}$ is a tangent bundle isomorphism over a diffeomorphism $g:A_0 \to A_1$, then it is a diffeomorphism $TM_0\Big\lvert_{A_0} \to TM_1\Big\lvert_{A_1}$.
\end{remark}

\begin{definition}\label{definition:extendable_pairs}
    Let $G$ be a tangent bundle morphism over $g$. An \textbf{extension of the pair $\mathbf{(g,G)}$} consists of a neighbourhood $U_0$ of $A_0$ and a smooth map
    \begin{equation*}
        \G:U_0 \to M_1
    \end{equation*}
    such that $\G\Big\lvert_{A_0} = g$ and $D\G = G$ along $A_0$, i.e., on $TM_0\Big\lvert_{A_0}$.

    The pair $(g,G)$ is \textbf{extendable} if it has an extension.

    The pair is \textbf{locally extendable} if around every $a_0 \in A_0$ there exists a neighbourhood $U_{a_0}$ of $a_0$ in $M_0$ such that the pair $(g\Big\lvert_{A_0 \cap U_{a_0}}, G\Big\lvert_{A_0 \cap U_{a_0}})$ is extendable.
\end{definition}
\begin{remark}\label{remark:locally_extendable_analog}
    Let $(g,G)$ be a locally extendable pair. Then necessarily
    \begin{equation}\label{equation:zariski_differential_agrees}
        G\Big\lvert_{T^ZA_0} = D^Zg.
    \end{equation}
    When $A_i$ are submanifolds, Condition~\eqref{equation:zariski_differential_agrees} is equivalent to local extendability of $(g,G)$. For general subspaces $A_i \subset M_i$, Condition~\eqref{equation:zariski_differential_agrees} is necessary but not sufficient for $(g,G)$ to be locally extendable, as the following example demonstrates.
\end{remark}
\begin{example}
    Let $M_0 = M_1 = \R^2$ and $A_0 = A_1$ be the union of the axes. Let $g:A_0 \to A_1$ be the identity and
    \begin{equation*}
        G(x,y) = \begin{pmatrix}
            1 & x \\
            0  & 1
        \end{pmatrix}.
    \end{equation*}
    Then $G$ is a tangent bundle isomorphism over $g$ with $G\Big\lvert_{T^ZA_0} = D^Zg$, but the pair $(g,G)$ is not (locally) extendable. Indeed, if $\G:\R^2 \to \R^2$ is such an extension, then the smooth function $\G_x:\R^2 \to \R$ satisfies
    \begin{equation*}
        \left(\frac{d\G_x}{dx}, \frac{d\G_x}{dy}\right)(x,y) = \left(1, x\right),\quad (x,y) \in A_0.
    \end{equation*}
    Taking derivative of the smooth function $\frac{d\G_x}{dy}$ with respect to the path $t \mapsto (t, 0)$ (contained in $A_0$) we get
    \begin{equation*}
        \frac{d}{dx}\frac{d\G_x}{dy} = 1,
    \end{equation*}
    but taking derivative of the smooth function $\frac{d\G_x}{dx}$ with respect to the path $t\mapsto (0,t)$ we get
    \begin{equation*}
        \frac{d}{dy}\frac{d\G_x}{dx} = 0.
    \end{equation*}
    This contradicts the smoothness of $\G_x$.
\end{example}

\subsection{Local extendability implies extendability}

In this subsection we show that local extendability of a pair $(g,G)$ is equivalent to extendability. To this end, we prove a stronger result, \Cref{lemma:stronger_local_extendability_implies_extendability}, which we use in the proof of \Cref{thm:main_coisotropic} in \Cref{section:coisotropic_theorem}.

\begin{lemma}\label{lemma:stronger_local_extendability_implies_extendability}
    For $i=0,1$, let $M_i$ be manifolds and $B_i \subset A_i \subset M_i$ closed subspaces of them. Let $g:A_0 \to A_1$ be a smooth map and
    \begin{equation*}
        G:TM_0\Big\lvert_{B_0} \to TM_1\Big\lvert_{B_1}
    \end{equation*}
    a tangent bundle morphism over $g\Big\lvert_{B_0}$.
    Assume that for each $b_0\in B_0$ there exist a neighbourhood $U_{b_0}$ of $b_0$ in $M_0$ and a smooth map $\G^{b_0}:U_{b_0} \to M_1$ with
    \begin{enumerate}
        \item $\G^{b_0}\Big\lvert_{A_0\cap U_{b_0}} = g$, and
        \item $D\G^{b_0} = G$ along $B_0\cap U_{b_0}$.
    \end{enumerate}
    Then there exist a neighbourhood $U_0$ of $B_0$ in $M_0$ and a smooth map $\G:U_0 \to M_1$ with
    \begin{enumerate}
        \item $\G\Big\lvert_{A_0\cap U_0} = g$, and
        \item $D\G = G$ along $B_0$.
    \end{enumerate}

    Moreover, let $P_0, O_0$ be open subsets of $M_0$, with $\overline{P_0} \subset O_0$, and let
    \begin{equation*}
        \G^{O_0}:O_0 \to M_1
    \end{equation*}
    be a smooth map with
    \begin{enumerate}
        \item $\G^{O_0}\Big\lvert_{A_0\cap O_0} = g$ and
        \item $D\G^{O_0} = G$ along $B_0 \cap O_0$.
    \end{enumerate}
    Then $\G:U_0 \to M_1$ can be taken so that it agrees with $\G^{O_0}$ on $\overline{P_0} \cap U_0$.
\end{lemma}
\begin{proof}
    For each $b_0\in B_0 \setminus \overline{P_0}$, let $U_{b_0}$ be a neighbourhood of $b_0$  disjoint from  $\overline{P_0}$, and let $\G^{b_0}:U_{b_0} \to M_1$ be a smooth map with $\G^{b_0}\Big\lvert_{A_0\cap U_{b_0}} = g$ and $D\G^{b_0} = G$ along $B_0\cap U_{b_0}$.

    Let
    \begin{equation*}
        U_0' \coloneq O_0 \cup \bigcup_{B_0 \setminus \overline{P_0}}U_{b_0}.
    \end{equation*}
    Using paracompactness, there exist a discrete subset $I \subset B_0 \setminus \overline{P_0}$ and a smooth partition of unity
    \begin{equation*}
        \{\varphi_{O_0}\} \cup {\{\varphi_{b_0}\}}_{b_0 \in I}:U_0' \to [0,1],
    \end{equation*}
    such that $\varphi_{O_0}$ is supported in $O_0$ and each $\varphi_{b_0}$ is supported in $U_{b_0}$. In particular, $\varphi_{O_0}$ restricts to $1$ on $\overline{P_0}$.

    Identify $M_1$ with a submanifold of some $\R^N$, and let $W \subset \R^N$ be a tubular neighbourhood of $M_1$ in $\R^N$ with projection $\pi:W \to M_1$.

    Let $\G_{\R^N}:U_0' \to \R^N$ be the smooth function defined by
    \begin{equation*}
        \G_{\R^N}(p_0) = \varphi_{O_0}(p_0)\cdot \G^{O_0}(p_0) + \sum_{b_0\in I}\varphi_{b_0}(p_0)\cdot \G^{b_0}(p_0)
    \end{equation*}
    and let $U_0 =\G_{\R^N}^{-1}(W)$, which is a neighbourhood of $B_0$ in $M_0$. Define $\G:U_0 \to M_1$ by
    \begin{equation*}
        \G = \pi \circ \G_{\R^N}.
    \end{equation*}
    We claim that $\G$ satisfies the required properties.

    Let $p_0\in A_0 \cap U_0$. Since $\G^{O_0}\Big\lvert_{A_0\cap O_0} = g$, for all $b_0\in I$ we have $\G^{b_0}\Big\lvert_{A_0 \cap U_{b_0}} = g$, and $\pi$ fixes $M_1$, we get
    \begin{align*}
        \G(p_0) &= \pi\left(\varphi_{O_0}(p_0)\cdot \G^{O_0}(p_0) + \sum_{b_0\in I}\varphi_{b_0}(p_0)\cdot \G^{b_0}(p_0)\right) \\
        &= \pi\left(\underbrace{\left(\varphi_{O_0}(p_0) + \sum_{b_0\in I}\varphi_{b_0}(p_0)\right)}_{1}\cdot g(p_0)\right) \\
        &= \pi\left(g(p_0)\right) \\
        &= g(p_0).
    \end{align*}
    Similarly, for $p_0\in B_0$ and $v\in T_{p_0}M_0$, since $D\pi$ fixes $TM_1 \subset T\R^N\Big\lvert_{M_1}$, we get
    \begin{align*}
        D\G(v) =& \left(D\pi\circ D\G_{\R^N}\right)(v) \\
        =& D\pi\Bigl(D\varphi_{O_0}(v)\cdot \G^{O_0}(p_0) + \varphi_{O_0}(p_0)\cdot D\G^{O_0}(v)\Bigr) + \\
        & D\pi\left(\sum_{b_0\in I}D\varphi_{b_0}(v)\cdot \G^{b_0}(p_0) + \sum_{b_0\in I}\varphi_{b_0}(p_0)\cdot D\G^{b_0}(v)\right) \\
        =& D\pi\left(\underbrace{\left(D\varphi_{O_0}(v) + \sum_{b_0\in I}D\varphi_{b_0}(v)\right)}_0\cdot g(p_0) + \underbrace{\left(\varphi_{O_0}(p_0) + \sum_{b_0\in I}\varphi_{b_0}(p_0)\right)}_1\cdot G(v)\right) \\
        =& \left(D\pi \circ G\right)(v) \\
        =& G(v).
    \end{align*}

    Finally, since $\varphi_{O_0}$ restricts to $1$ on $\overline{P_0}$, the diffeomorphism $\G:U_0 \to M_1$ agrees with $\G^{O_0}$ on $\overline{P_0}$.
\end{proof}

\begin{thm}\label{thm:local_extendability_implies_extendability}
    For $i=0,1$, let $A_i \subset M_i$ be closed subspaces of manifolds, with the induced subspace differential structure. Let $g:A_0 \to A_1$ be a smooth map and $G:TM_0\Big\lvert_{A_0} \to TM_1\Big\lvert_{A_1}$ a tangent bundle morphism over $g$.
    Then the pair $(g,G)$ is locally extendable if and only if it is extendable.
\end{thm}
\begin{proof}
    Apply \Cref{lemma:stronger_local_extendability_implies_extendability} with $B_i = A_i$.
\end{proof}

\subsection{Proof of Theorem~\ref{thm:main_diffeomorphism}}
\begin{lemma}\cite[Exercise~1.8.14]{Guillemin_pollack}\label{lemma:inverse_function_revisited}
    For $i=0,1$, let $M_i$ be manifolds, $A_i\subset M_i$ closed subsets, and $f:M_0 \to M_1$ a smooth function. Assume that for all $a_0\in A_0$, the differential ${D_{a_0}f:T_{a_0}M_0 \to T_{f(a_0)}M_1}$ is an isomorphism and that $f$ restricts on $A_0$ to a diffeomorphism of $A_0$ onto $A_1$. Then there exists a neighbourhood $U_0$ of $A_0$ in $M_0$ such that $f\Big\lvert_{U_0}:U_0 \to f(U_0)$ is a diffeomorphism.
\end{lemma}
\begin{proof}
    Since near $A_0$ the map $f$ is a local diffeomorphism, it is enough to find a neighbourhood $U_0$ of $A_0$ such that $f$ is injective on $U_0$.

    For every $a_1\in A_1$ and $a_0 \in A_0$ with $f(a_0)=a_1$, by the inverse function theorem and the assumption that
    \begin{equation*}
        D_{a_0}f:T_{a_0}M_0 \to T_{f(a_0)}M_1
    \end{equation*}
    is an isomorphism, there exists a neighbourhood $\widetilde{U}_{a_1}$ of $a_1$ and a local inverse $g_{a_1}:\widetilde{U}_{a_1} \to M_0$ of $f$.
    Using paracompactness and the fact that $A_1$ is closed, there exist a discrete subset $I\subset A_1$ and a locally finite collection
    ${\{U_{a_1}\}}_{a_1 \in I}$
    of open sets of $M_1$ that covers $A_1$ and such that $\overline{U_{a_1}} \subset \widetilde{U}_{a_1}$ for all $a_1 \in I$. Write
    \begin{equation*}
        U_{A_1} \coloneq \bigcup_{a_1 \in I} U_{a_1}.
    \end{equation*}

    Let
    \begin{equation*}
        W = \{q_1\in U_{A_1}: \: g_{a_1}(q_1) = g_{a_1'}(q_1)\text{ whenever }a_1, a_1'\in I\text { and } q_1\in U_{a_1}\cap U_{a_1'}\}.
    \end{equation*}
    We prove that $W$ contains an open neighbourhood of $A_1$. Let $p_1\in A_1$, let $p_0\in A_0$ be the unique point with $f(p_0) = p_1$, and let
    \begin{equation*}
        J_{p_1} = \{a_1 \in I:\: p_1 \in \overline{U_{a_1}}\}.
    \end{equation*}
    By the local finiteness of ${\{U_{a_1}\}}_{a_1 \in I}$, the subset $J_{p_1}\subset I$ is finite, and there exists an open neighbourhood $V_{p_1}$ of $p_1$ in $U_{A_1}$ that intersects each $U_{a_1}$ if and only if $a_1 \in J_{p_1}$.
    Write
    \begin{equation*}
        W_{p_1} \coloneq V_{p_1} \cap f\left(\bigcap_{a_1\in J_{p_1}}g_{a_1}\left(\widetilde{U}_{a_1}\right)\right).
    \end{equation*}
    Since the subset $J_{p_1}\subset I$ is finite and for all $a_1 \in J_{p_1}$ we have $g_{a_1}(p_1) = p_0$ and $p_1 \in \widetilde{U}_{a_1}$, the subset $W_{p_1}$ is an open neighbourhood of $p_1$ in $U_{A_1}$.

    Let ${q_1 \in W_{p_1}}$ and assume that there exist $a_1,a_1'\in I$ with $q_1 \in U_{a_1}\cap U_{a_1'}$.
    Then $a_1, a_1' \in J_{p_1}$, and therefore ${q_1\in f\left(g_{a_1}\left(\widetilde{U}_{a_1}\right) \cap g_{a_1'}\left(\widetilde{U}_{a_1'}\right)\right)}$.
    Now,
    \begin{equation*}
        q_1 = \left( f\circ g_{a_1} \right)(q_1) = \left( f\circ g_{a_1'} \right)(q_1)
    \end{equation*}
    and $f$ is injective on $g_{a_1}\left(\widetilde{U}_{a_1}\right) \cap g_{a_1'}\left(\widetilde{U}_{a_1'}\right)$, which implies ${g_{a_1}(q_1) = g_{a_1'}(q_1)}$. It follows that $W_{p_1} \subset W$.

    Therefore, $W$ contains an open neighbourhood of $A_1$. Denote such a neighbourhood by $W_{A_1}$. On $W_{A_1}$, one can patch together the locally defined $g_{a_1},\, a_1 \in I$ to a smooth inverse $g$ of $f$. In particular, $U_0 \coloneq g\left(W_{A_1}\right)$ is an open neighbourhood of $A_0$ on which $f$ is injective.
\end{proof}

We are now ready to prove \Cref{thm:main_diffeomorphism}.
\begin{proof}[Proof of \Cref{thm:main_diffeomorphism}]
    By \Cref{thm:local_extendability_implies_extendability}, the pair $(g,G)$ is extendable, i.e., there exist a neighbourhood $\widetilde{U}_0$ of $A_0$ in $M_0$ and an extension
    \begin{equation*}
        \G:\widetilde{U}_0 \to M_1
    \end{equation*}
    of $(g, G)$. By \Cref{lemma:inverse_function_revisited}, and since $G$ is an isomorphism, $\G$ restricts to a diffeomorphism from some open neighbourhood $U_0$ of $A_0$ to some open neighbourhood $U_1$ of $A_1$.
\end{proof}

\section{Relative Poincar\'e Lemma and Moser’s trick}\label{section:moser_trick_proof}
In this section we prove \Cref{thm:moser_trick}, and then use it together with \Cref{thm:main_diffeomorphism} to prove \Cref{thm:main_symplectic}.

\begin{definition}\label{definition:smooth_weak_deformation_retraction}
    Let $M$ be a manifold, $A\subset M$ a subset, and $V$ a neighbourhood of $A$ in $M$. A \textbf{smooth weak deformation retraction} of $V$ to $A$ is a smooth map
    \begin{equation*}
        R:[0,1]\times V \to V
    \end{equation*}
    so that the family of functions $r^t = R(t, \cdot)$ satisfies
    \begin{enumerate}
        \item $r^1 = \Id_V$,
        \item $r^0(V) \subset A$, and
        \item\label{condition:remains_in_A} $A$ is an invariant subset of $R$, i.e., $r^t(A) \subset A$, for all $t\in [0,1]$.
    \end{enumerate}
\end{definition}

\begin{remark}\label{remark:strong_vs_weak_deformation_retraction}
    Here the word ``weak'' refers to condition \cref{condition:remains_in_A}, which allows $A$ to move within itself along the deformation. A \textbf{strong} deformation retraction is a weak deformation retraction with $r^t\Big\lvert_A = \Id_A$ for all $t\in[0,1]$.
\end{remark}

\subsection{Relative Poincar\'e Lemma}

\begin{notation}
    Let $M$ be a manifold. We denote the smooth $k$-forms on $M$ by $\Omega^k\left(M\right)$, the exterior derivative by $d$, the contraction with a vector field $X$ by $\i_X$, and the Lie derivative in the direction of $X$ by $\L_X$.
\end{notation}

In this subsection we prove a version of the relative Poincar\'e Lemma, \Cref{lemma:relative_poincare}, which works for a general subset of a manifold, provided it admits a smooth weak deformation retraction from a neighbourhood. Cf. \cite{relative_poincare}.

\begin{lemma}[Homotopy property of the fiber integration operator {\cite[Lemma~17.9]{lee_smooth_manifolds}}]\label{lemma:homotopy_property_of_fiber_integration}
    Let $V$ be a manifold. Let $i_s:V\to [0,1] \times V$ be the inclusion at the $s$ level, let $\pi_V:[0,1] \times V \to V$ be the projection, and let
    \begin{equation*}
        {\left(\pi_V\right)}_*: \Omega^k\left([0,1] \times V\right) \to \Omega^{k-1}\left(V\right)
    \end{equation*}
    be the fiber integration operator defined by
    \begin{equation*}
        {\left(\pi_V\right)}_*\gamma = \int_{0}^{1}i_t^* \i_{\frac{\partial}{\partial t}}\gamma \,dt
    \end{equation*}
    for all $\gamma \in \Omega^k\left([0,1] \times V\right)$.
    Then for all $\gamma \in \Omega^k\left([0,1] \times V\right)$ we have
    \begin{equation*}
        i_1^*\gamma - i_0^*\gamma = {\left(\pi_V\right)}_*d\gamma + d{\left(\pi_V\right)}_*\gamma.
    \end{equation*}
\end{lemma}

\begin{lemma}[Relative Poincar\'e Lemma]\label{lemma:relative_poincare}
    Let $V$ be a manifold, $A\subset V$ a subset, and ${R:[0,1]\times V \to V}$ a smooth weak deformation retraction from $V$ to $A$. Let $\alpha \in \Omega^k\left(V\right)$ be a closed $k$-form on $V$ that vanishes along $A$.
    Then the $(k-1)$-form
    \begin{equation*}
        \beta \coloneq {\left(\pi_V\right)}_*R^*\alpha
    \end{equation*}
    vanishes along $A$ and satisfies $d\beta=\alpha$.

    Moreover, given an $R$-invariant subset $B\subset V$ such that $\alpha$ vanishes on $T^ZB$, the form $\beta$ vanishes on $T^ZB$.
\end{lemma}
\begin{proof}
    By assumption, $R\circ i_1 = \Id_V$. Let $r \coloneq R\circ i_0$, then
    \begin{equation*}
        i_1^*R^*\alpha = \alpha, \quad i_0^*R^*\alpha=r^*\alpha.
    \end{equation*}

    Applying \Cref{lemma:homotopy_property_of_fiber_integration} to $\gamma = R^*\alpha$ we obtain
    \begin{equation*}
        \alpha - r^*\alpha = {\left(\pi_V\right)}_*dR^*\alpha + d{\left(\pi_V\right)}_*R^*\alpha.
    \end{equation*}
    Since $\alpha$ is closed, $dR^*\alpha = R^*d\alpha = 0$. Since $\alpha$ vanishes along $A$ and $r(V) \subset A$, we get $r^*\alpha = 0$. Therefore,
    \begin{equation*}
        \alpha = d{\left(\pi_V\right)}_*R^*\alpha
    \end{equation*}
    and we can take $\beta = {\left(\pi_V\right)}_*R^*\alpha$.

    The fact that $\beta$ vanishes along $A$ follows from its definition, the fact that $\alpha$ vanishes along $A$, and $R\left([0,1]\times A\right) \subset A$.

    Given an $R$-invariant subset $B\subset V$, the Zariski tangent $T^ZB$ is invariant under $Dr^t$ and $\frac{d}{dt}r^t(p) \in T^Z_{R(t, p)}B$ for $p\in B$. For $v\in T^Z_pB$, we have
    \begin{align*}
        \beta_p(v)
        &= \int_{0}^{1}i_t^* \i_{\frac{\partial}{\partial t}}R^*\alpha \,dt \\
        &= \int_{0}^{1}\alpha_{R(t, p)}\left(\underbrace{\frac{d}{ds}\bigg\lvert_{t}r^s(p), Dr^t\cdot v}_{\in T^ZB}\right) \,dt,
    \end{align*}
    so if $\alpha$ vanishes on $T^ZB$ then $\beta$ also vanishes on $T^ZB$.
\end{proof}

\subsection{Moser's trick}

\begin{lemma}\label{lemma:flow_domain_vanish_neighbourhood}
    Let $M$ be a manifold, $A\subset M$ a subset, and ${\left(X_t\right)}_{t\in [0,1]}$ a time-dependent vector field with local flow
    \begin{equation*}
        \Phi^{t_0}_{(X_t)}:\mathcal{O} \to M,
    \end{equation*}
    where $\mathcal{O} \subset M \times \R$ is the maximal flow domain. Assume that $X_t$ vanishes along $A$ for all $t$.
    Then there exists a neighbourhood $U$ of $A$ such that
    \begin{equation*}
        U \times [0,1] \subset \mathcal{O}.
    \end{equation*}
\end{lemma}
\begin{proof}
    For each $p\in A$, because $X_t$ vanishes along $A$, we have
    \begin{equation*}
        \{p\} \times [0, 1] \subset \mathcal{O}.
    \end{equation*}
    Since $\mathcal{O} \subset M \times [0,1]$ is open and $[0,1]$ is compact, the tube lemma yields a neighbourhood $U_p$ of $p$ in $M$ such that
    \begin{equation*}
        U_p \times [0, 1] \subset \mathcal{O}.
    \end{equation*}
    Finally, take $U = \bigcup_{p\in A}U_p$.
\end{proof}

\begin{lemma}\label{lemma:using_beta_for_moser}
    Let $V$ be a manifold, $A \subset V$ a subset, and $\w_0, \w_1$ symplectic forms on $V$. Suppose that there exists a $1$-form $\beta$ with $d\beta = \w_0 - \w_1$. Assume that along $A$ the $1$-form $\beta$ vanishes and $\w_1 = \w_0$.
    Then there exist neighbourhoods $U',U'' \subset V$ of $A$ and a diffeomorphism
    \begin{equation*}
        \G:U'\to U''
    \end{equation*}
    such that
    \begin{enumerate}
        \item\label{property:moser_diff_symplectic_pullback} $\G^*\w_1 = \w_0$ and
        \item\label{property:moser_vanishes_on_vanishing_beta} $\G$ fixes $A$.
    \end{enumerate}
\end{lemma}
\begin{proof}
    Define
    \begin{equation*}
        \w_t = (1-t)\w_0 + t\w_1.
    \end{equation*}
    Then $\w_t$ is a smooth family of $2$-forms, non-degenerate along $A$, and hence non-degenerate on a neighbourhood of $A$. It satisfies $\frac{d}{dt}\w_t = \w_1 - \w_0 = -d\beta$.

    Let ${\left(X_t\right)}_{t\in [0,1]}$ be the time-dependent vector field satisfying
    \begin{equation*}
        \i_{X_t}\w_t = \beta
    \end{equation*}
    and let $\Phi^{t_0}_{(X_t)}$ be its flow. By \Cref{lemma:flow_domain_vanish_neighbourhood}, there exists a smaller neighbourhood $U' \subset V$ of $A$ on which the map $\Phi^{1}_{(X_t)}$ is well defined. We claim that
    \begin{equation*}
        \G \coloneq \Phi^{1}_{(X_t)}:U' \to U'' \coloneq \G(U')
    \end{equation*}
    satisfies the required properties.

    For property \cref{property:moser_diff_symplectic_pullback}, we differentiate
    \begin{align*}
        \frac{d}{ds}\bigg\lvert_{t_0}{\left(\Phi^{s}_{(X_t)}\right)}^*\w_s
        &= {\left(\Phi^{t_0}_{(X_t)}\right)}^*\left(\frac{d}{ds}\bigg\lvert_{t_0}\w_s + \L_{X_{t_0}}\w_{t_0} \right) \\
        &= {\left(\Phi^{t_0}_{(X_t)}\right)}^*\left(-d\beta + d\underbrace{\i_{X_{t_0}}\w_{t_0}}_{\beta} + \i_{X_{t_0}}\underbrace{d\w_{t_0}}_0\right) \\
        &= 0,
    \end{align*}
    and get that ${\left(\Phi^{s}_{(X_t)}\right)}^*\w_s = \w_0$ for all $s\in [0,1]$. In particular, ${\left(\Phi^1_{(X_t)}\right)}^*\w_1 = \w_0$.

    Property \cref{property:moser_vanishes_on_vanishing_beta} follows from the fact that $X_t$ vanishes on $\{ \beta = 0\} \cap U'$.
\end{proof}

\subsection{Proof of Theorem \ref{thm:moser_trick}}

\begin{proof}[Proof of \Cref{thm:moser_trick}]
    The form $\alpha \coloneq \w_0 - \w_1$ is closed and vanishes along $A$, so by \Cref{lemma:relative_poincare} there exists a $1$-form $\beta$ which vanishes along $A$ and such that $d\beta = \w_0 - \w_1$. Applying \Cref{lemma:using_beta_for_moser} yields the required diffeomorphism.
\end{proof}

\begin{remark}\label{remark:mosers_trick_discussion}
    The following \Cref{example:weak_moser_trick_differential} shows that the differential of the diffeomorphism constructed in the proof of \Cref{thm:moser_trick} does not necessarily restrict to the identity along $A$.

    However, if we further assume that the smooth deformation retraction is strong,  then ($A$ is necessarily a submanifold and) we prove in \Cref{appendix:strong_deformation_retraction_moser} that the differential of the diffeomorphism constructed in the proof of \Cref{thm:moser_trick} restricts to the identity along $A$.
\end{remark}

\begin{example}\label{example:weak_moser_trick_differential}
    Let $A=\R_x = \{y=0\} \subset \R^2_{x,y}$, and take
    \begin{equation*}
        \w_0 = dx\wedge dy,\quad \w_1 = (1 + y)dx\wedge dy
    \end{equation*}
    so that $\alpha = \w_0 - \w_1 = -ydx\wedge dy$. Take the smooth weak deformation retraction
    \begin{equation*}
        r^t(x,y) = (tx, ty).
    \end{equation*}
    Then the fiber integration construction yields
    \begin{align*}
        \beta(x,y)(v) &= \int_{0}^{1}i_t^*(x,y)\left(\i_{\frac{\partial}{\partial t}}R^*\alpha\right) \,dt \\
        &= \int_{0}^{1} \alpha\left(r^t(x,y)\right)\left(\frac{d}{ds}\bigg\lvert_t r^s(x,y), D_{(x,y)}r^t(v)\right) dt \\
        &= \int_{0}^{1} \left(-ty dx \wedge dy\right)\left(x \partial x + y \partial y, tv\right)dt  \\
        &= \left(\int_{0}^{1}-t^2 dt \right)\left(xy dy - y^2dx\right)(v) \\
        &= \frac{1}{3}\left(y^2 dx - yx dy\right)(v)
    \end{align*}
    with $d\beta = \alpha$ and
    \begin{equation*}
        X_t = \frac{-y}{3(1+ty)} \left(x \partial x  + y \partial y \right).
    \end{equation*}

    We use the identification $T\R^2 \simeq \R^2 \times \R^2$. Define $f(t, y) = \frac{-y}{3(1+ty)}$, so that
    \begin{equation*}
        X_t(x,y) = \begin{pmatrix}
            xf(t,y) \\
            yf(t,y)
        \end{pmatrix}.
    \end{equation*}
    Since $f(t, 0) =0$ and $\frac{\partial f}{\partial y}(t,0)= \frac{-1}{3}$, we obtain
    \begin{equation*}
        D_{(x,0)}X_t
        = \begin{pmatrix}
            f(t,0) & x\frac{\partial f}{\partial y}(t,0) \\
            0 & 0\cdot\frac{\partial f}{\partial y}(t,0) + f(t,0)
        \end{pmatrix}
        = \begin{pmatrix}
            0 & \frac{-x}{3} \\
            0 & 0
        \end{pmatrix}.
    \end{equation*}

    Using the identification $T\R^2 \simeq \R^2 \times \R^2$, the differential $D_{(x,0)}\Phi_{(X_t)}^{t_0}$ at $y=0$ needs to satisfy the differential equation
    \begin{equation*}
        \frac{d}{ds}\bigg\lvert_{t_0}\left(D_{(x,0)}\Phi_{(X_t)}^{s}\right) = \begin{pmatrix}
            0 & \frac{-x}{3} \\
            0 & 0
        \end{pmatrix}D_{(x,0)}\Phi_{(X_t)}^{t_0}
    \end{equation*}
    with $D_{(x,0)}\Phi_{(X_t)}^{0} = \Id$. One can check that
    \begin{equation*}
        D_{(x,0)}\Phi_{(X_t)}^{t_0} =
        \begin{pmatrix}
            1 & \frac{-xt_0}{3} \\
            0 & 1
        \end{pmatrix}
    \end{equation*}
    is the unique solution, and in particular $ D_{(x,0)}\Phi_{(X_t)}^1 \neq \Id$ for $x\neq 0$.
\end{example}

\subsection{Proof of Theorem~\ref{thm:main_symplectic}}
\begin{proof}[Proof of \Cref{thm:main_symplectic}]
    Apply \Cref{thm:main_diffeomorphism} with $(g,G)$ to obtain neighbourhoods $U_0,U_1$ of $A_0,A_1$, respectively, and a diffeomorphism
    \begin{equation*}
        \widetilde{\G}:U_0 \to U_1
    \end{equation*}
    that restricts to $g$ on $A_0$ and whose differential restricts to $G$ on $TM_0\Big\lvert_{A_0}$. Since $G$ is a symplectic tangent bundle isomorphism, the symplectic forms $\w_0$ and $\widetilde{\G}^*\w_1$ agree along $A_0$.

    By assumption, $A_0$ is smoothly locally trivial with conical fibers, so by the results of \cite{zimhony2024commutative}, there exists a neighbourhood $V_0$ of $A_0$ in $U_0$ and a smooth weak deformation retraction of $V_0$ to $A_0$. Applying \Cref{thm:moser_trick}, we obtain neighbourhoods $U', U''\subset V_0$ of $A_0$ and a diffeomorphism $\G^\w:U'\to U''$ with
    \begin{equation*}
        {\left(\G^\w\right)}^*{\widetilde{\G}}^*\w_1 = \w_0
    \end{equation*}
    and ${\left(\G^\w\right)}\Big\lvert_{A_0} = \Id$. It follows that
    \begin{equation*}
        \G \coloneq \widetilde{\G} \circ \G^\w:U' \to \widetilde{\G}(U'')
    \end{equation*}
    is a symplectomorphism with $\G\big\lvert_{A_0} = g$.
\end{proof}

\section{Symplectic normal form for coisotropic stratified subspaces}\label{section:coisotropic_theorem}
In this section we prove \Cref{thm:main_coisotropic}.

\subsection{Proof of Theorem~\ref{thm:main_coisotropic}}
For $i=0,1$, let $\left(M_i, \w_i\right)$ be symplectic manifolds and $\left(A_i, \S_i\right)$ stratified subspaces of them, smoothly locally trivial with conical fibers. Assume that $A_i$ are strongly coisotropic and that there exists a stratified diffeomorphism $g:\left(A_0, \S_0\right) \to \left(A_1, \S_1\right)$ with $g^*\w_1 = \w_0$ as Zariski forms on $A_0$.

We construct the required neighbourhoods and symplectomorphism using ascending recursion on the dimension of strata.

Recall \Cref{notation:strata_skeleton}.

\subsubsection{Assumption for construction step d}\label{subsubsection:construction_assumption}
There exist
\begin{enumerate}
    \item for $i=0,1$, neighbourhoods $U^{\leq (d-1)}_i$ of $A^{\leq (d-1)}_i$ in $M_i$;
    \item a symplectomorphism
    \begin{equation*}
        \G^{\leq (d-1)}: \left(U^{\leq (d-1)}_0, \w_0\right) \to \left(U^{\leq (d-1)}_1, \w_1\right)
    \end{equation*}
    \item a stratified diffeomorphism $g^{\leq (d-1)}:A_0 \to A_1$;
\end{enumerate}
such that the following hold.
\begin{enumerate}
    \item $\G^{\leq (d-1)}$ and $g^{\leq (d-1)}$ agree on $U^{\leq (d-1)}_0 \cap A_0$.
    \item The stratified diffeomorphisms $g, g^{\leq (d-1)}:A_0 \to A_1$ are isotopic through a family of stratified diffeomorphisms ${\{g_t\}}_{t\in [0,1]}$, with $g_t^*\w_1 = \w_0$ as Zariski forms on $A_0$ for all $t\in [0,1]$.
\end{enumerate}

Note that $U^{\leq (d-1)}_i$ may intersect strata of dimension higher than $d-1$.

\subsubsection{Construction step d}\label{subsubsection:constrction_step}
For $d=0$, take $A^{\leq -1}_i=\emptyset$ and $U^{\leq -1}_i=\emptyset$.

For $d>0$, let $V^{\leq (d-1)}_0$ and $W^{\leq (d-1)}_0$ be neighbourhoods of $A^{\leq (d-1)}_0$ in $M_0$ with
\begin{equation*}
    A^{\leq (d-1)}_0 \subset W^{\leq (d-1)}_0 \subset \overline{W^{\leq (d-1)}_0} \subset V^{\leq (d-1)}_0 \subset \overline{V^{\leq (d-1)}_0} \subset U^{\leq (d-1)}_0,
\end{equation*}
and such that
\begin{equation*}
    V^{\leq (d-1)}_1 \coloneq \G^{\leq (d-1)}\left(V^{\leq (d-1)}_0\right), \quad W^{\leq (d-1)}_1 \coloneq \G^{\leq (d-1)}\left(W^{\leq (d-1)}_0\right)
\end{equation*}
satisfy
\begin{equation*}
    A^{\leq (d-1)}_1 \subset W^{\leq (d-1)}_1 \subset \overline{W^{\leq (d-1)}_1} \subset V^{\leq (d-1)}_1 \subset \overline{V^{\leq (d-1)}_1} \subset U^{\leq (d-1)}_1,
\end{equation*}
where the closures are taken in $M_i$.

By local finiteness of $\S_i$, normality of $M_i$ as topological spaces, and \Cref{remark:closedness_of_strata}, there exist neighbourhoods $U^{X^d_i}_i$ of $X^d_i \in \S^d_i$ in $M^{\geq d}_i = M_i \setminus A^{\leq (d-1)}_i$ such that
\begin{equation*}
    \forall X^d_i,Y^d_i\in \S^d_i:\: X^d_i \neq Y^d_i \implies U^{X^d_i}_i \cap U^{Y^d_i}_i = \emptyset.
\end{equation*}
See \cite[Lemma~4.6]{zimhony2024commutative} for more details.

\begin{definition}\label{definition:extension_for_stratum}
    Let $X^d_0 \in \S_0^d$ and $X_1^d \coloneq g^{\leq (d-1)}\left(X^d_0\right) \in \S_1^d$.

    \textbf{An extension for $\mathbf{\left(g^{\leq (d-1)}, \G^{\leq (d-1)} \right)}$ to $\mathbf{(X^d_0, X^d_1)}$} consists of
    \begin{enumerate}
        \item possibly shrinking $U^{X_i}_i$;
        \item neighbourhoods $V^{X^d_i}_i, W^{X^d_i}_i$ of $X^d_i$ in $U_i^{X^d_i}$, satisfying
        \begin{equation*}
            X^d_i \subset W^{X^d_i}_i \subset \overline{W^{X^d_i}_i} \subset V^{X^d_i}_i \subset \overline{V^{X^d_i}_i} \subset U^{X^d_i}_i,
        \end{equation*}
        where the closures are taken in $M^{\geq d}_i$;
        \item a diffeomorphism
        \begin{equation*}
            \G^{X^d_0}: U_0^{X^d_0} \to  U_1^{X^d_1}
        \end{equation*}
        with $\G^{X^d_0}\left(V^{X^d_0}_0\right) = V^{X^d_1}_1$, and similarly for $\overline{V^{X^d_i}_i}, W^{X^d_i}_i$, and $\overline{W^{X^d_i}_i}$;
        \item a stratified diffeomorphism $g^{X^d_0}:U_0^{X^d_0} \cap A_0 \to U_1^{X^d_1} \cap A_1$ defined by
        \begin{equation*}
            g^{X^d_0} \coloneq \G^{X^d_0}\Big\lvert_{U_0^{X^d_0} \cap A_0},
        \end{equation*}
        where we consider $U_i^{X^d_i} \cap A_i$ with their induced stratifications (\Cref{definition:induced_stratification});
    \end{enumerate}
    such that the following conditions hold.
    \begin{enumerate}
        \item $\G^{X^d_0}$ restricts on $W^{X^d_0}_0$ to a symplectomorphism $W^{X^d_0}_0 \to W^{X^d_1}_1$;
        \item $\G^{X^d_0}$ agrees with $\G^{\leq (d-1)}$ on $W^{\leq (d-1)}_0 \cap U^{X^d_0}_0$;
        \item $g^{X^d_0}$ agrees with $g^{\leq (d-1)}$ on $W^{\leq (d-1)}_0 \cap U^{X^d_0}_0 \cap A_0$ and on $\left(U^{X^d_0}_0 \setminus \overline{V^{X^d_0}_0}\right) \cap A_0$;
        \item The stratified diffeomorphisms
        \begin{equation*}
            g^{X^d_0},g^{\leq (d-1)}\Big\lvert_{U^{X^d_0}_0 \cap A_0}:U^{X^d_0}_0 \cap A_0 \to U^{X^d_1}_1 \cap A_1
        \end{equation*}
        are isotopic through a family of stratified diffeomorphisms
        \begin{equation*}
            g_t:U^{X^d_0}_0 \cap A_0 \to U^{X^d_1}_1 \cap A_1,
        \end{equation*}
        satisfying, for all $t\in [0,1]$, the following properties.
        \begin{enumerate}
            \item $g_t^*\w_1 = \w_0$ as Zariski forms on $A_0$.
            \item $g_t$ agrees with $g^{\leq (d-1)}$ on $W^{\leq (d-1)}_0 \cap U^{X^d_0}_0 \cap A_0$ and on $\left(U^{X^d_0}_0 \setminus \overline{V^{X^d_0}_0}\right) \cap A_0$.
        \end{enumerate}
    \end{enumerate}
\end{definition}

See \Cref{figure:schematics_neighbourhood} for a schematic drawing of neighbourhoods.

\begin{figure}
    \centering
    \begin{tikzpicture}[scale=0.9]
        \fill[opacity=0.2,color=magenta] (0,0) -- (7,3) node [anchor=west,opacity=0.4]{$U^{X^d}$} -- (7,-3);
        \fill[opacity=0.3,color=magenta] (0,0) -- (7,2) node [anchor=west,opacity=0.5]{$V^{X^d}$} -- (7,-2);
        \fill[opacity=0.4,color=magenta] (0,0) -- (7,1) node [anchor=west,opacity=0.6]{$W^{X^d}$} -- (7,-1);

        \fill[opacity=0.6,color=blue!50] (0,0) circle (1.5);
        \draw[color=blue] (0,1.5) node[below]{$W^{\leq (d-1)}$};
        \fill[opacity=0.4,color=blue!50] (0,0) circle (2.5);
        \draw[color=blue] (0,2.5) node[below]{$V^{\leq (d-1)}$};
        \fill[opacity=0.2,color=blue!50] (0,0) circle (3.5);
        \draw[color=blue] (0,3.5) node[below]{$U^{\leq (d-1)}$};

        \draw[color=red,line width=1.5pt] (0,0) -- (7,0) node[anchor=west]{$X^d$};
        \draw[line width=1.5pt] (0,0) -- (-3, 2.625) node[align=center,anchor=north east]{Other strata of \\ dimension d} -- (-4,3.5);
        \draw[line width=1.5pt] (0,0) -- (-4,-3.5);
        \filldraw[color=blue] (0,0) circle (4pt) node[below=8pt]{$A^{\leq (d-1)}$} ;
    \end{tikzpicture}
    \caption{A schematic drawing of neighbourhoods: the stratified subspace $A^{\leq d}$ is a union of the three lines, which represent strata of dimension $d$, and the vertex in the middle, which represents $A^{\leq (d-1)}$.}
    \label{figure:schematics_neighbourhood}
\end{figure}

\begin{proposition}\label{proposition:extensions_nbhd_partition}
    Assume that there exist extensions for $\left(g^{\leq (d-1)}, \G^{\leq (d-1)} \right)$ to $(X^d_0, X^d_1)$ for each $X^d_0 \in \S_0^d$ and $X_1^d \coloneq g^{\leq (d-1)}\left(X^d_0\right) \in \S_1^d$.
    Let $i=0,1$ and $p_i \in M_i$. Then either
    \begin{enumerate}
        \item there exists some $X^d_i\in \S^d_i$ with $p_i\in U^{X^d_i}_i$,
        \item $p_i \in A_i^{\leq (d-1)}$, or
        \item $p_i \in M^{\geq d}_i \setminus \bigsqcup_{X^d_i\in \S^d_i} \overline{V^{X^d_0}_i}$.
    \end{enumerate}
\end{proposition}
\begin{proof}
    By construction.
\end{proof}

\begin{proposition}\label{proposition:from_extending_to_single_stratum_to_gleqd}
    Assume that there exist extensions for $\left(g^{\leq (d-1)}, \G^{\leq (d-1)} \right)$ to $(X^d_0, X^d_1)$ for each $X^d_0 \in \S_0^d$ and $X_1^d \coloneq g^{\leq (d-1)}\left(X^d_0\right) \in \S_1^d$.
    Define $g^{\leq d}:A_0 \to A_1$ by
    \begin{equation*}
        g^{\leq d}(p_0) = \begin{cases}
            g^{X^d_0}(p_0), & p_0 \in U^{X^d_0}_0 \cap A_0, \\
            g^{\leq (d-1)}(p_0), & \text{otherwise}.
        \end{cases}
    \end{equation*}
    Then it is a stratified diffeomorphism $A_0 \to A_1$.
\end{proposition}
\begin{proof}
    We prove that $g^{\leq d}$ is smooth, surjective, injective, has a smooth inverse, and is a stratified map.

    For smoothness, let $p_0 \in A_0$. Using \Cref{proposition:extensions_nbhd_partition}, we divide into the following cases.

    If $p_0\in U^{X^d_0}_0$ for some $X^d_0\in \S^d_0$, then there exists a neighbourhood of $p_0$ in $A_0$ contained in $U^{X^d_0}_0 \cap A_0$. In this neighbourhood, $g^{\leq d}$ is smooth since it agrees with $g^{X^d_0}$.

    If $p_0 \in A_0^{\leq (d-1)}$, then there exists a neighbourhood of $p_0$ in $A_0$ contained in $W^{\leq (d-1)}_0 \cap A_0$. In this neighbourhood, $g^{\leq d}$ is smooth since it agrees with $g^{\leq (d-1)}$, as every $g^{X^d_0}$ agrees with $g^{\leq (d-1)}$ on $W^{\leq (d-1)}_0 \cap U^{X^d_0}_0 \cap A_0$.

    If $p_0 \in A^{\geq d}_0 \setminus \bigsqcup_{X^d_0\in \S^d_0} \overline{V^{X^d_0}}$, then there exists a neighbourhood of $p_0$ in $A^{\geq d}_0$ which does not intersect any $\overline{V^{X^d_0}_0}$. In this neighbourhood, $g^{\leq d}$ is smooth since it agrees with $g^{\leq (d-1)}$, as every $g^{X^d_0}$ agrees with $g^{\leq (d-1)}$ on $\left(U^{X^d_0}_0 \setminus \overline{V^{X^d_0}_0}\right) \cap A_0$.

    Injectivity and surjectivity follow from
    \begin{equation*}
        X^d_i \neq Y^d_i \implies U^{X^d_i}_i \cap U^{Y^d_0}_i = \emptyset
    \end{equation*}
    and from the fact that the functions
    \begin{equation*}
        g^{X^d_0}:U^{X^d_0}_0\cap A_0 \to U^{X^d_1}_1 \cap A_1, \quad X^d_0\in \S^d_0
    \end{equation*}
    and
    \begin{equation*}
        g^{\leq (d-1)}\Big\lvert_{A_0 \setminus \bigsqcup_{X^d_0\in \S^d_0} U^{X^d_0}_0}: A_0 \setminus \bigsqcup_{X^d_0\in \S^d_0} U^{X^d_0}_0 \to A_1 \setminus \bigsqcup_{X^d_1\in \S^d_1} U^{X^d_1}_1
    \end{equation*}
    are bijective.

    Let ${\left(g^{\leq d}\right)}^{-1}:A_1 \to A_0$ be the inverse of $g^{\leq d}$. The same argument used to prove that $g^{\leq d}$ is smooth, replacing $U^{X^d_0}_0$ with $U^{X^d_1}_1$ and $g^{\leq (d-1)}, g^{X^d_0}$ with their inverses, proves that ${\left(g^{\leq d}\right)}^{-1}$ is smooth.

    Lastly, $g^{\leq d}:A_0 \to A_1$ is a stratified diffeomorphism since $g^{X^d_0}, X^d_0\in \S^d_0$ and $g^{\leq (d-1)}$ send $X_0\in \S_0$ to $g^{\leq (d-1)}(X_0) \in \S_1$.
\end{proof}

\begin{proposition}\label{proposition:gleqd_isotopic_to_g}
    Let $g^{\leq d}:A_0 \to A_1$ be the stratified diffeomorphism constructed in \Cref{proposition:from_extending_to_single_stratum_to_gleqd}. Then it is isotopic to $g:A_0 \to A_1$ through a family of stratified diffeomorphisms
    \begin{equation*}
        g_t:A_0 \to A_1 \quad t\in [0,1],
    \end{equation*}
    with $g_t^*\w_1 = \w_0$ as Zariski forms on $A_0$ for all $t\in [0,1]$.
\end{proposition}
\begin{proof}
    By concatenating isotopies, it is enough to prove that $g^{\leq (d-1)}$ and $g^{\leq d}$ are isotopic through a family of stratified diffeomorphisms ${\{g_t\}}_{t\in [0,1]}$, with $g_t^*\w_1 = \w_0$ as Zariski forms on $A_0$ for all $t\in [0,1]$.

    Recall that the stratified diffeomorphisms
    \begin{equation*}
        g^{X^d_0},g^{\leq (d-1)}\Big\lvert_{U^{X^d_0}_0 \cap A_0}:U^{X^d_0}_0 \cap A_0 \to U^{X^d_1}_1 \cap A_1
    \end{equation*}
    are isotopic through families of stratified diffeomorphisms ${\{g^{X^d_0}_t\}}_{t\in [0,1]}$, satisfying, for all ${t\in [0,1]}$,
    \begin{enumerate}
        \item ${\left(g^{X^d_0}_t\right)}^*\w_1 = \w_0$ and
        \item $g^{X^d_0}_t$ agrees with $g^{\leq (d-1)}$ on $W^{\leq (d-1)}_0 \cap U^{X^d_0}_0 \cap A_0$ and on $\left(U^{X^d_0}_0 \setminus \overline{V^{X^d_0}_0}\right) \cap A_0$.
    \end{enumerate}

    Define $g_t:A_0 \to A_1,\,t\in[0,1]$ by
    \begin{equation*}
        g_t(p_0) = \begin{cases}
            g^{X^d_0}_t(p_0), & p_0 \in U^{X^d_0}_0, \\
            g^{\leq (d-1)}(p_0), & \text{otherwise}.
        \end{cases}
    \end{equation*}
    Using \Cref{proposition:extensions_nbhd_partition} and the properties of ${\{g^{X^d_0}_t\}}_{t\in [0,1]}$, we get that ${\{g_t\}}_{t\in [0,1]}$ is a family satisfying the required properties.
\end{proof}

\begin{proposition}
    Assume that there exist extensions for $\left(g^{\leq (d-1)}, \G^{\leq (d-1)} \right)$ to $(X^d_0, X^d_1)$ for each $X^d_0 \in \S_0^d$ and $X_1^d \coloneq g^{\leq (d-1)}\left(X^d_0\right) \in \S_1^d$. Then there exist $U^{\leq d}_i, \G^{\leq d}$, and $g^{\leq d}$ satisfying the recursion assumption \ref{subsubsection:construction_assumption} for construction step $d+1$.
\end{proposition}
\begin{proof}
    Let $g^{\leq d}:A_0 \to A_1$ be the stratified diffeomorphism constructed in \Cref{proposition:from_extending_to_single_stratum_to_gleqd}. By \Cref{proposition:gleqd_isotopic_to_g}, it is isotopic to $g:A_0 \to A_1$ through a family of stratified diffeomorphisms ${\{g_t\}}_{t\in [0,1]}$, with $g_t^*\w_1 = \w_0$ as Zariski forms on $A_0$ for all $t\in [0,1]$.

    Define
    \begin{equation*}
        U^{\leq d}_i \coloneq W^{\leq (d-1)}_i \cup \bigcup_{X^d_i \in \S^d_i}W^{X^d_i}_i
    \end{equation*}
    and
    \begin{equation*}
        \G^{\leq d}:U^{\leq d}_0 \to U^{\leq d}_1
    \end{equation*}
    by
    \begin{equation*}
        \G^{\leq d}(p_0) = \begin{cases}
            \G^{X^d_0}(p_0), & p_0 \in W^{X^d_0}_0, \\
            \G^{\leq (d-1)}(p_0), & p_0 \in W^{\leq (d-1)}_0.
        \end{cases}
    \end{equation*}
    Using \Cref{proposition:extensions_nbhd_partition} and the properties of ${\{\G^{X^d_0}\}}_{X^d_0\in \S^d_0}$, we get that $\G^{\leq d}$ is smooth, bijective, and has a smooth inverse.
    Since $\G^{\leq (d-1)}$ and ${\{\G^{X^d_0}\}}_{X^d_0\in \S^d_0}$ restrict to symplectomorphisms on $W^{\leq (d-1)}_i$ and ${\{W^{X^d_i}_i\}}_{X^d_0\in \S^d_0}$ respectively, $\G^{\leq d}$ is a symplectomorphism.
\end{proof}

Therefore, to prove \Cref{thm:main_coisotropic}, it remains to prove:
\begin{thm}\label{thm:extending_to_stratum}
    Let $X^d_0\in \S^d_0$. There exists an extension for $\left(g^{\leq (d-1)}, \G^{\leq (d-1)} \right)$ to $(X^d_0, X^d_1)$.
\end{thm}

We prove \Cref{thm:extending_to_stratum} in the following \Cref{subsection:extending_to_higher_stratum}.

\subsection{Extending to a stratum}\label{subsection:extending_to_higher_stratum}

Let $X_0 \coloneq X^d_0\in \S^d_0$ and $X_1 \coloneq g^{\leq (d-1)}(X_0) \in \S^d_1$. Our strategy for proving \Cref{thm:extending_to_stratum} is the following.
\begin{enumerate}
    \item Apply \Cref{lemma:stratified_spaces_locally_diffeo,lemma:stronger_local_extendability_implies_extendability} to construct a diffeomorphism $\G^{X_0}_{\text{diff}}$ from a neighbourhood of $X_0$ in $M_0$ to a neighbourhood of $X_1$ in $M_1$, that agrees with $\G^{\leq (d-1)}$ on a neighbourhood of $\overline{W^{\leq (d-1)}_0}$ and whose differential along $X_0$ is a symplectic tangent bundle isomorphism.
    \item Construct a tubular neighbourhood of $X_0$ in $M_0$ whose multiplication by $t\in[0,1]$ preserves $A_0$.
    \item Use this multiplication by scalars to apply a strong version of Moser's trick for $\w_0$ and ${\left(\G^{X_0}_{\text{diff}}\right)}^*\w_1$, preserving $A_0$.
\end{enumerate}

\subsubsection{Diffeomorphism around a stratum}
Denote by
\begin{equation*}
    G^{\leq (d-1)}:TM_0\Big\lvert_{U^{\leq (d-1)}_0 \cap A_0} \to TM_1\Big\lvert_{U^{\leq (d-1)}_1 \cap A_1}
\end{equation*}
the symplectic tangent bundle isomorphism over $g^{\leq (d-1)}\Big\lvert_{U^{\leq (d-1)}_0 \cap A_0}$ defined by
\begin{equation*}
    G^{\leq (d-1)} \coloneq D\G^{\leq (d-1)}\Big\lvert_{U^{\leq (d-1)}_0 \cap A_0}.
\end{equation*}

\begin{proposition}\label{proposition:symplectic_bundle_iso_over_X0}
    There exists a symplectic tangent bundle isomorphism
    \begin{equation*}
        G^{X_0}:\left(TM_0\Big\lvert_{X_0}, \w_0 \right) \to \left(TM_1\Big\lvert_{X_1}, \w_1 \right)
    \end{equation*}
    over $g^{\leq (d-1)}\Big\lvert_{X_0}$ which
    \begin{enumerate}
        \item agrees with $G^{\leq (d-1)}$ along $X_0 \cap \overline{V^{\leq (d-1)}_0}$ and
        \item agrees with $D^Zg^{\leq (d-1)}$ on $T^ZA_0\Big\lvert_{X_0}$.
    \end{enumerate}
\end{proposition}
\begin{proof}
    For $i=0,1$, the subbundles $T^ZA_i\Big\lvert_{X_i}$ are coisotropic in $TM_i\Big\lvert_{X_i}$. Given compatible almost complex structures $J_i$ on $TM_i$, we obtain decompositions
    \begin{equation*}
        TM_i\Big\lvert_{X_i} = T^ZA_i\Big\lvert_{X_i} \oplus J_i\left( {\left( T^ZA_i\Big\lvert_{X_i}\right)}^{\w_i} \right) \simeq T^ZA_i\Big\lvert_{X_i} \oplus {\left( {\left( T^ZA_i\Big\lvert_{X_i}\right)}^{\w_i} \right)}^*
    \end{equation*}
    and a symplectic tangent bundle isomorphism
    \begin{equation*}
        T^ZA_0\Big\lvert_{X_0} \oplus {\left( {\left( T^ZA_0\Big\lvert_{X_0}\right)}^{\w_0} \right)}^* \xrightarrow{\sim} T^ZA_1\Big\lvert_{X_1} \oplus {\left( {\left( T^ZA_1\Big\lvert_{X_1}\right)}^{\w_1} \right)}^*
    \end{equation*}
    given by
    \begin{equation*}
        D^Zg^{\leq (d-1)} \oplus {\left(D^Zg^{\leq (d-1)}\right)}^{-1,*}.
    \end{equation*}
    In particular, such a symplectic bundle isomorphism agrees with $D^Zg^{\leq (d-1)}$ on $T^ZA_0\Big\lvert_{X_0}$.

    It remains to find compatible almost complex structures $J_i$ such that the induced isomorphism agrees with $G^{\leq (d-1)}$ along $X_0 \cap \overline{V^{\leq (d-1)}_0}$. Let $J_0$ be an $\w_0$-compatible almost complex structure on $TM_0$. It is enough to find a compatible almost complex structure $J_1$ on $TM_1$ such that
    \begin{equation*}
        J_1 = {\left(G^{\leq (d-1)}\right)}_*J_0
    \end{equation*}
    along $X_1 \cap \overline{V^{\leq (d-1)}_1}$.

    Let $h_0$ be the $\w_0$-compatible Riemannian metric on $TM_0$ induced by $J_0$. Using partitions of unity and $\overline{V^{\leq (d-1)}_1} \subset U^{\leq (d-1)}_1$, use $\G^{\leq (d-1)}$ to construct a Riemannian metric $h_1$ on $TM_1$ such that
    \begin{equation*}
        h_1\Big\lvert_{\overline{V^{\leq (d-1)}_1} \cap X_1} = {\left(G^{\leq (d-1)}\right)}_*h_0.
    \end{equation*}
    Using the polar decomposition as in \cite[Lecture~2]{weinstein_lectures}, let $J_1$ be the $\w_1$-compatible almost complex structure induced by $h_1$. Since $h_1$ is $\w_1$-compatible on $\overline{V^{\leq (d-1)}_1} \cap X_1$, the $\w_1$-compatible Riemannian metric induced by $J_1$ agrees with $h_1$ on that subset. Since $\w_1, h_1$ are \linebreak
    $G^{\leq (d-1)}$-related to $\w_0, h_0$ on $\overline{V^{\leq (d-1)}_1} \cap X_1$, it follows that $J_1 = \G^{\leq (d-1)}_*J_0$ on that subset.
\end{proof}

Let
\begin{equation*}
    G^{X_0}:\left(TM_0\Big\lvert_{X_0}, \w_0 \right) \to \left(TM_1\Big\lvert_{X_1}, \w_1 \right)
\end{equation*}
be as in \Cref{proposition:symplectic_bundle_iso_over_X0}.

\begin{proposition}\label{proposition:coisotropic_local_extendability}
    Let $p_0\in X_0$ and $p_1=g^{\leq (d-1)}(p_0)\in X_1$. Then there exist
    \begin{enumerate}
        \item neighbourhoods $U^{p_i}_i$ of $p_i$ in $M_i$ and
        \item a diffeomorphism $\G^{p_0}:U^{p_0}_0 \to U^{p_1}_1$ which agrees with $g^{\leq (d-1)}$ on $A_0\cap U^{p_0}_0$ and whose differential $D\G^{p_0}$ agrees with $G^{X_0}$ along $X_0$.
    \end{enumerate}
\end{proposition}
\begin{proof}
    Since $A_i$ are smoothly locally trivial with conical fibers, \Cref{lemma:stratified_spaces_locally_diffeo} yields conical charts $\theta_{p_0}:U^{p_0}_0 \to W$ and $\theta_{p_1}:U^{p_1}_1 \to W$ of $p_0, p_1$ respectively such that $\theta_{p_1}^{-1} \circ \theta_{p_0}$ restricts to $g^{\leq (d-1)}$ on $A_0 \cap U^{p_0}_0$. Write
    \begin{equation*}
        X_W \coloneq \theta_{p_0}\left(X_0 \cap U^{p_0}_0 \right) = \theta_{p_1}\left(X_1 \cap U^{p_1}_1 \right),
    \end{equation*}
    which by definition satisfies
    \begin{equation*}
        X_W = \left(\R^k\times \{0\}\right) \cap W \subset \R^k \times \R^{n-k}.
    \end{equation*}

    Let $(x,y)$ be coordinates on $W\subset \R_x^k\times \R_y^{n-k}$ and $(x,y,v_x,v_y)$ be coordinates on $ TW$. Define $\tau:W \to TW\Big\lvert_{X_W}$ by
    \begin{equation*}
        \tau(x,y) = (x, 0, x, y)
    \end{equation*}
    and $\sigma:TW\Big\lvert_{X_W} \to \R^k \times \R^{n-k}$ by
    \begin{equation*}
        \sigma(x, 0, v_x, v_y) = (v_x, v_y).
    \end{equation*}
    Let $\G_W:W \to \R^k \times \R^{n-k}$ be the map given by the composition
    \[\begin{tikzcd}
        W & {TW\Big\lvert_{X_W}} & {TM_0\Big\lvert_{X_0}} & {TM_1\Big\lvert_{X_1}} & {TW\Big\lvert_{X_W}} & {\R^k \times \R^{n-k}}
        \arrow["\tau", from=1-1, to=1-2]
        \arrow["{D\theta_{p_0}^{-1}}", from=1-2, to=1-3]
        \arrow["{G_{X_0}}", from=1-3, to=1-4]
        \arrow["{D\theta_{p_1}}", from=1-4, to=1-5]
        \arrow["\sigma", from=1-5, to=1-6].
    \end{tikzcd}\]
    We verify that $\G_W$ has the following properties:
    \begin{enumerate}
        \item\label{property:G_W_differential} $D\G_W = D\theta_{p_1} \circ G_{X_0} \circ D\theta_{p_0}^{-1}$ along $X_W$.
        \item\label{property:G_W_diffeo} After possibly shrinking $W$, the map $\G_W:W \to \G_W(W)$ is a diffeomorphism.
        \item\label{property:G_W_preserves_A_W} $\G_W$ fixes $A_W \coloneq \theta_{p_0}\left(A_0 \cap U_{p_0} \right) = \theta_{p_1}\left(A_1 \cap U_{p_1} \right)$.
    \end{enumerate}

    Property \cref{property:G_W_differential} follows by calculation in local coordinates. Property \cref{property:G_W_diffeo} follows from the inverse function theorem.
    To prove property \cref{property:G_W_preserves_A_W}, note that by the definition of a conical chart (\Cref{definition:smoothly_locally_trivial}), the subset $A_W$ is conical and therefore $\tau(A_W) \subset T^ZA_W\Big\lvert_{X_W}$. Since both $G_{X_0}$ and $D\theta_{p_1}^{-1} \circ D\theta_{p_0}$ restrict to $D^Zg^{\leq (d-1)}$ on $T^ZA_0$, the bundle isomorphism
    \begin{equation*}
        \left(D\theta_{p_1} \circ G_{X_0} \circ D\theta_{p_0}^{-1}\right):TW\Big\lvert_{X_W} \to TW\Big\lvert_{X_W}
    \end{equation*}
    restricts on $T^ZA_W\Big\lvert_{X_W}$ to the identity, and it follows that for $(x,y)\in A_W$ we have
    \begin{align*}
        \G_W(x,y)
        =& \left(\sigma \circ D\theta_{p_1} \circ G_{X_0} \circ D\theta_{p_0}^{-1}\right)\underbrace{(x, 0, x, y)}_{\in T^ZA_W\Big\lvert_{X_W}} \\
        =& \sigma(x,0,x,y) \\
        =& (x,y).
    \end{align*}
    Finally, we shrink $W$ such that $\G_W(W)$ is contained in the domain of $\theta_{p_1}^{-1}$ and define ${\G^{p_0} = \theta_{p_1}^{-1} \circ \G_W \circ \theta_{p_0}}$. One can check that $\G^{p_0}$ satisfies the required Properties.
\end{proof}

\begin{proposition}\label{proposition:existence_of_G_X_0}
    Possibly shrinking $U^{X_i}_i$, there exists a diffeomorphism
    \begin{equation*}
        \G^{X_0}_{\text{diff}}:U^{X_0}_0 \to U^{X_1}_1
    \end{equation*}
    which
    \begin{enumerate}
        \item agrees with $g^{\leq (d-1)}$ on $U^{X_0}_0 \cap A_0$,
        \item\label{property:G_X_0_agrees_with_G_leq_on_V_leq} agrees with $\G^{\leq (d-1)}$ on $V^{\leq (d-1)}_0 \cap U^{X_0}_0$, and
        \item whose differential $D\G^{X_0}_{\text{diff}}$ agrees with $G^{X_0}$ along $X_0$.
    \end{enumerate}
\end{proposition}
\begin{proof}
    By \Cref{proposition:coisotropic_local_extendability}, for every $p_0\in X_0$, there exist a neighbourhood $U^{p_0}_0$ of $p_0$ in $U^{X_0}_0$ and a smooth map $\G^{p_0}:U^{p_0}_0 \to M_1$ that
    \begin{enumerate}
        \item agrees with $g^{\leq (d-1)}$ on $A_0\cap U^{p_0}_0$ and
        \item whose differential $D\G^{p_0}$ agrees with $G^{X_0}$ along $X_0\cap U^{p_0}_0$.
    \end{enumerate}

    Applying \Cref{lemma:stronger_local_extendability_implies_extendability} with the sets $M_i = U_i^{X_i},\ {B_i=X_i},\ {A_i= U^{X_i}_i \cap A_i}$, the smooth maps ${g=g^{\leq (d-1)}\Big\lvert_{U^{X_0}_0 \cap A_0}},\ {G=G^{X_0}}$, the additional data ${O_0 = U^{\leq (d-1)}_0},{P_0 = V^{\leq (d-1)}_0}$, and ${\G^{O_0}=\G^{\leq (d-1)}}$, after possibly shrinking $U^{X_i}_i$, we obtain a smooth map
    \begin{equation*}
        \G^{X_0}_{\text{diff}}:U^{X_0}_0 \to U^{X_1}_1
    \end{equation*}
    which
    \begin{enumerate}
        \item agrees with $g^{\leq (d-1)}$ on $U^{X_0}_0 \cap A_0$,
        \item agrees with $\G^{\leq (d-1)}$ on $V^{\leq (d-1)}_0 \cap U^{X_0}_0$, and
        \item whose differential $D\G^{X_0}_{\text{diff}}$ agrees with $G^{X_0}$ along $X_0$.
    \end{enumerate}

    Finally, apply \Cref{lemma:inverse_function_revisited} and possibly shrink $U^{X_i}_i$ so that $\G^{X_0}_{\text{diff}}$ is a diffeomorphism.
\end{proof}

Let
\begin{equation*}
    \G^{X_0}_{\text{diff}}:U^{X_0}_0 \to U^{X_1}_1
\end{equation*}
be as in \Cref{proposition:existence_of_G_X_0}.

\subsubsection{Tangential tubular neighbourhood}
In this subsection we work with Euler-like vector fields and their induced tubular neighbourhoods. See \Cref{appendix:euler_likes} for a quick introduction to Euler-like vector fields.

The following lemma is discussed in greater generality in~\cite[Section~4]{zimhony2024commutative}. We give a brief proof here.

\begin{lemma}\label{lemma:tangent_euler_like}
    Let $M$ be a manifold and let $(A,\S)\subset M$ be a stratified subspace, smoothly locally trivial with conical fibers. Let $X^d\in \S$ be a stratum.
    Then there exists an Euler-like vector field $\E$ along $X$ that is tangent to higher strata, i.e., $\E(y) \in T_yY$ for $Y>X$ and $y\in Y$ in the domain of $\E$.
\end{lemma}
\begin{proof}
    For each $p\in X$ let $\theta_p:U_p \to W_p$ be a conical chart of $p$ as in \Cref{definition:smoothly_locally_trivial}. Let $(x_1,\ldots, x_d, y_1, \ldots, y_{n-d})$ be coordinates on $W_p$ such that $\theta_p(X\cap U_p) = \{y=0\}\cap W_p$. Write
    \begin{equation*}
        \E^p \coloneq {\left(\theta_p^{-1}\right)}_*\left(\sum_{i=1}^{n-d}y_i\partial y_i\right),
    \end{equation*}
    which is Euler-like along $X\cap U_p$ and tangent to higher strata by the definition of a conical chart.

    Using paracompactness, there exist a discrete subset $I\subset X$ and a partition of unity ${\{\varphi_{p}\}}_{p\in I}$ in
    \begin{equation*}
        U\coloneq \bigcup_{p\in X}U_p,
    \end{equation*}
    with $\supp{\varphi_{p}} \subset U_{p}$. The vector field
    \begin{equation*}
        \E\coloneq \sum_{p\in I} \varphi_{p} \E^{p}
    \end{equation*}
    is defined on $U$, Euler-like along $N$ (by \Cref{lemma:euler_like_convex}), and tangent to higher strata by construction.
\end{proof}

\begin{proposition}\label{proposition:good_euler_like}
    There exists an Euler-like vector field $\E_0$ along $X_0$ in $M_0$, inducing a tubular neighbourhood embedding $\Psi:\nu(M_0, X_0) \to U^{X_0}_0$ and a multiplication by scalars (\Cref{definition:induced_multiplication_by_scalar})
    \begin{equation*}
        m^t_{\E_0}:\Psi\left(\nu(M_0, X_0)\right) \to \Psi\left(\nu(M_0, X_0)\right), \quad t\in[0,1],
    \end{equation*}
    with the following properties.
    \begin{enumerate}
        \item $\E_0$ is tangent to higher strata.
        \item\label{property:open_closed_euler_like_thing} If $\overline{W^{\leq(d-1)}_0}$ intersects a fiber of $m^0_{\E_0}$, then the whole fiber is contained in $V^{\leq (d-1)}_0$.
    \end{enumerate}
\end{proposition}
\begin{proof}
    Apply \Cref{lemma:tangent_euler_like}, \Cref{thm:euler_like_tubular}, and \Cref{lemma:closed_open_sets_euler}.
\end{proof}

Let $\E^{X_0}$ be an Euler-like vector field along $X_0$ in $U^{X_0}_0$ as in \Cref{proposition:good_euler_like}. Possibly shrinking $U^{X_i}_i$, we assume that $U^{X_0}_0$ is the image of the induced tubular neighbourhood embedding
\begin{equation*}
    \Psi:\nu(M_0, X_0) \xrightarrow{\sim} U^{X_0}_0
\end{equation*}
and $U^{X_1}_1 = \G^{X_0}_{\text{diff}}\left(U^{X_0}_0\right)$.
Let
\begin{equation*}
    m^t_{\E_0}:U^{X_0}_0 \to U^{X_0}_0, \quad t\in [0,1]
\end{equation*}
be the induced multiplication by scalars. By the assumption that $\E^{X_0}$ is tangent to higher strata, \Cref{lemma:tangential_vector_field_stratified_diffeo}, and
\begin{equation*}
    \frac{d}{dt}\bigg\lvert_{t=0}m^{\exp(t)}_{\E_0}(p) = \E^{X_0}(p),
\end{equation*}
the map $m^t_{\E_0}$ preserves strata for $t\in (0,1]$. Since $U^{X_0}_0 \cap A_0$ is closed in $U^{X_0}_0$, it follows that $U^{X_0}_0 \cap A_0$ is an invariant subset of $m^t_{\E_0}$ for $t\in [0,1]$.

\subsubsection{Moser's trick revisited}

We follow the proof of \Cref{lemma:using_beta_for_moser} with some modifications.

\begin{proposition}
    Consider $U^{X_0}_0$ with the two symplectic forms $\w_0, {\left(\G^{X_0}_{\text{diff}}\right)}^*\w_1$. Then $\w_0$ and ${\left(\G^{X_0}_{\text{diff}}\right)}^*\w_1$ agree
    \begin{enumerate}
        \item on $T^ZA_0\Big\lvert_{U^{X_0}_0 \cap A_0}$,
        \item along $V^{\leq (d-1)}_0 \cap U^{X_0}_0$, and
        \item along $X_0$.
    \end{enumerate}
\end{proposition}
\begin{proof}
    Using the properties of $\G^{X_0}_{\text{diff}}$ from \Cref{proposition:existence_of_G_X_0} and ${\left(g^{\leq (d-1)}\right)}^*\w_1 = \w_0$ as Zariski forms on $T^ZA_0$.
\end{proof}

Let $R:[0,1]\times U^{X_0}_0 \to U^{X_0}_0$ be the smooth deformation retraction $U^{X_0}_0 \to X^0$ defined by $R(t, p) = m^t_{\E_0}(p)$.

Let $\alpha \coloneq \w_0 - {\left(\G^{X_0}_{\text{diff}}\right)}^*\w_1$, which is a closed $2$-form that vanishes along $X_0$, and write
\begin{equation*}
    \w_t = (1-t)\w_0 + t{\left(\G^{X_0}_{\text{diff}}\right)}^*\w_1
\end{equation*}
and
\begin{equation*}
    \beta \coloneq {\left(\pi_V\right)}_*R^*\alpha = \int_{0}^{1}i_t^* \i_{\frac{\partial}{\partial t}}R^*\alpha \,dt.
\end{equation*}

\begin{proposition}\label{proposition:beta_properties}
    The $1$-form $\beta$ satisfies
    \begin{enumerate}
        \item $d\beta = \alpha$,
        \item $\beta$ vanishes along $X_0$, and
        \item\label{property:beta_vanishes_on_zariski_tangent} $\beta$ vanishes on $T^ZA_0\Big\lvert_{U^{X_0}_0 \cap A_0}$.
    \end{enumerate}
\end{proposition}
\begin{proof}
    Using the fact that $U^{X_0}_0 \cap A_0$ is an invariant subset of $R$ and applying \Cref{lemma:relative_poincare}.
\end{proof}

Define the time-dependent vector field ${\left(X_t\right)}_{t\in [0,1]}$ on $U^{X_0}_0$ by
\begin{equation*}
    \i_{X_t}\w_t = \beta
\end{equation*}
and let $\Phi^{t_0}_{(X_t)}$ be its flow. Since $\beta$ vanishes along $X_0$, so does ${\left(X_t\right)}_{t\in [0,1]}$.

\begin{proposition}\label{proposition:creating_V_X_0}
    There exists a neighbourhood $V^{X_0}_0$ of $X_0$ in $U^{X_0}_0$ satisfying the following properties.
    \begin{enumerate}
        \item The time-1 flow $\Phi^{1}_{(X_t)}$ is well defined on $V^{X_0}_0$, i.e., $\Phi^{1}_{(X_t)}\left(V^{X_0}_0\right) \subset U^{X_0}_0$.
        \item $\overline{V^{X_0}_0} \subset U^{X_0}_0$, where the closure is taken in $M^{\geq d}_0$.
        \item\label{property:image_of_V_0_closed} The image of $\overline{V^{X_0}_0}$ under $\G^{X_0}_{\text{diff}}$ is closed in $M^{\geq d}_1$.
    \end{enumerate}
\end{proposition}
\begin{proof}
    Applying \Cref{lemma:flow_domain_vanish_neighbourhood} and possibly shrinking. For the last property, use the fact that $M_i$ are normal as topological spaces.
\end{proof}

Let $V^{X_0}_0 \subset U^{X_0}_0$ as in \Cref{proposition:creating_V_X_0}.

Let $\varphi:U^{X_0}_0 \to [0,1]$ be a smooth bump function, supported in $V^{X_0}_0$, which is identically $1$ on a neighbourhood $W'$ of $X_0$. Consider the time-dependent vector field ${\left(\varphi X_t\right)}_{t\in [0,1]}$, whose time-1 flow $\Phi^{t_0}_{(\varphi X_t)}:U^{X_0}_0 \to U^{X_0}_0$ is well defined by construction.

Let
\begin{equation*}
    W^{X_0}_0 \coloneq \bigcap_{t_0\in [0,1]}{\left(\Phi^{t_0}_{(\varphi X_t)}\right)}^{-1}(W').
\end{equation*}

\begin{proposition}
    $W^{X_0}_0$ is an open neighbourhood of $X_0$, with $\overline{W^{X_0}_0} \subset V^{X_0}_0$.
\end{proposition}
\begin{proof}
    The map $[0,1]\times U^{X_0}_0 \to U^{X_0}_0$ defined by $(t_0, p_0) \to {\left(\Phi^{t_0}_{(\varphi X_t)}\right)}^{-1}(p_0)$ is proper, and therefore it is closed. It follows that the subset
    \begin{equation*}
        U^{X_0}_0 \setminus W^{X_0}_0 = \bigcup_{t_0\in [0,1]}{\left(\Phi^{t_0}_{(\varphi X_t)}\right)}^{-1}\left(U^{X_0}_0 \setminus W'\right),
    \end{equation*}
    which is the image of $[0,1]\times\left(U^{X_0}_0 \setminus W'\right)$ under this map, is closed.
    Since ${\left(\Phi^{t_0}_{(\varphi X_t)}\right)}^{-1}$ fixes $X_0$ for all $t_0\in [0,1]$, the subset $W^{X_0}_0$ contains $X_0$. Since
    \begin{equation*}
        \overline{W^{X_0}_0}\subset \overline{W'} \subset \supp(\varphi) \subset V^{X_0}_0,
    \end{equation*}
    we have $\overline{W^{X_0}_0} \subset V^{X_0}_0$.
\end{proof}

Write $\G^\w \coloneq \Phi^{1}_{(\varphi X_t)}:U^{X_0}_0 \to U^{X_0}_0$.

\begin{proposition}\label{proposition:restricts_to_symplec_on_W_X_0}
    ${\left(\G^\w\right)}^*{\left(\G^{X_0}_{\text{diff}}\right)}^*\w_1 = \w_0$ on $W^{X_0}_0$.
\end{proposition}
\begin{proof}
    Let $p_0 \in W^{X_0}_0$. The definition of $W^{X_0}_0$ implies that for all $t_0\in [0,1]$ we have
    \begin{equation*}
        \Phi^{t_0}_{(\varphi X_t)}(p_0) = \Phi^{t_0}_{(X_t)}(p_0).
    \end{equation*}
    Using the same arguments as in the proof of \Cref{lemma:using_beta_for_moser}, it follows that
    \begin{equation*}
        \frac{d}{ds}\bigg\lvert_{t_0}{\Bigl({\left(\Phi^{s}_{(X_t)}\right)}^*\left(\w_s\right)\Bigr)} = 0,
    \end{equation*}
    and therefore ${\left(\Phi^{1}_{(X_t)}\right)}^*{\left(\G^{X_0}_{\text{diff}}\right)}^*\w_1 = \w_0$.

    Let $U^{p_0}$ be a neighbourhood of $p_0$ contained in $W^{X_0}_0$. Then the restriction of $\Phi^{1}_{(\varphi X_t)}$ to $U^{p_0}$ agrees with the restriction of $\Phi^{1}_{(X_t)}$ to ${U^{p_0}}$, and in particular their differential agree at $p_0$. Thus,
    \begin{align*}
        \left({\left(\G^\w\right)}^*{\left(\G^{X_0}_{\text{diff}}\right)}^*\w_1\right)(p_0) &= \left({\left(\Phi^{1}_{(\varphi X_t)}\right)}^*{\left(\G^{X_0}_{\text{diff}}\right)}^*\w_1\right)(p_0) \\
        &= \left({\left(\Phi^{1}_{(X_t)}\right)}^*{\left(\G^{X_0}_{\text{diff}}\right)}^*\w_1\right)(p_0) \\
        &= \w_0(p_0).
    \end{align*}
\end{proof}

\begin{proposition}\label{proposition:G_w_fixes_W_X_0}
    The map $\G^\w$ fixes $\overline{W^{\leq (d-1)}_0} \cap U^{X_0}_0$.
\end{proposition}
\begin{proof}
    Let $p_0\in \overline{W^{\leq (d-1)}_0} \cap U^{X_0}_0$. By \Cref{proposition:good_euler_like}\eqref{property:open_closed_euler_like_thing}, $m^t_{\E_0}(p_0) \in V^{\leq (d-1)}_0 \cap U^{X_0}_0$ for all $t\in [0,1]$. Since $\alpha$ vanishes along $V^{\leq (d-1)}_0 \cap U^{X_0}_0$, for all $v_0 \in T_{p_0}M_0$ we have
    \begin{align*}
        \beta_{p_0}(v_0) &= \int_{0}^{1}{\left(i_t^* \i_{\frac{\partial}{\partial t}}R^*\alpha\right)}_{p_0}(v_0) \,dt \\
        &= \int_{0}^{1}\underbrace{{\alpha}_{m^t_{\E_0}(p_0)}}_0\left(\frac{d}{ds}\bigg\lvert_{t}m^s_{\E_0}(p_0), Dm^t_{\E_0}\cdot v_0\right) \\
        &= 0.
    \end{align*}
    Therefore, $X_t(p_0) = 0$ for all $t\in [0,1]$ and $\G^\w(p_0) = \Phi^{1}_{(\varphi X_t)}(p_0) = p_0$.
\end{proof}

\begin{proposition}
    $U^{X_0}_0\cap A_0$ is an invariant subset of $\G^\w$. Moreover, considering $U^{X_0}_0\cap A_0$ with its induced stratification, $\G^\w$ restricts to a stratified diffeomorphism.
\end{proposition}
\begin{proof}
    By \Cref{proposition:beta_properties}\eqref{property:beta_vanishes_on_zariski_tangent}, $\beta_{p_0}$ vanishes on $T^Z_{p_0}A_0$, and therefore $X_t(p_0) \in {\left(T^Z_{p_0}A_0\right)}^{\w_t}$ for $t\in [0,1]$.
    For all $t\in [0,1]$, the restriction of $\w_t$ to $T^Z_{p_0}A_0$ agrees with $\w_0$.
    Therefore, for all $t\in [0,1]$, we have ${\left(T^Z_{p_0}A_0\right)}^{\w_t} = {\left(T^Z_{p_0}A_0\right)}^{\w_0}$ and the subspace $T^Z_{p_0}A_0 \subset \left(T_{p_0}M_0, \w_t(p_0)\right)$ is coisotropic.
    It follows that $X_t(p_0) \in {\left(T^Z_{p_0}A_0\right)}^{\w_0}$.

    Let $Y\in \S$ be the stratum containing $p_0$. By the strongly coisotropic assumption on $\left(A_0, \S_0\right)$, we get
    \begin{equation*}
        X_t(p_0) \in {\left(T^Z_{p_0}A_0\right)}^{\w_0} \subset T_{p_0}Y,
    \end{equation*}
    i.e., the vector field $X_t$ is tangent to strata for $t\in [0,1]$. Applying \Cref{lemma:tangential_vector_field_stratified_diffeo} finishes the proof.
\end{proof}

Let
\begin{equation*}
    g^\w\coloneq \G^\w\Big\lvert_{U^{X_0}_0 \cap A_0}:U^{X_0}_0 \cap A_0 \to U^{X_0}_0 \cap A_0.
\end{equation*}

\begin{proposition}\label{proposition:g_w_isotopic_to_id}
    The stratified diffeomorphism
    \begin{equation*}
        g^\w:U^{X_0}_0 \cap A_0 \to U^{X_0}_0 \cap A_0
    \end{equation*}
    is isotopic to $\Id:U^{X_0}_0 \cap A_0 \to U^{X_0}_0 \cap A_0$ through a family of stratified diffeomorphisms
    \begin{equation*}
        g_t:U^{X_0}_0 \cap A_0 \to U^{X_0}_0 \cap A_0, \quad t\in [0,1],
    \end{equation*}
    such that, for each $t\in [0,1]$, the stratified diffeomorphism $g_t$ satisfies the following properties.
    \begin{enumerate}
        \item $g_t^*\w_t = \w_0$ as Zariski forms on $U^{X_0}_0 \cap A_0$.
        \item $g_t$ fixes $W^{X_0}_0 \cap U^{X_0}_0 \cap A_0$ and $\left(U^{X_0}_0 \setminus \overline{V^{X^d_0}_0}\right) \cap A_0$.
    \end{enumerate}

    In particular, ${\left(g^\w\right)}^*{\left(\G^{X_0}_{\text{diff}}\right)}^*\w_1 = \w_0$ as Zariski forms on $A_0$, and $g^\w$ fixes $W^{X_0}_0 \cap U^{X_0}_0 \cap A_0$ and  $\left(U^{X_0}_0 \setminus \overline{V^{X^d_0}_0}\right) \cap A_0$.
\end{proposition}
\begin{proof}
    For all $t_0 \in [0,1]$, write
    \begin{equation*}
        g_{t_0} \coloneq \Phi^{t_0}_{(\varphi X_t)}\Big\lvert_{U^{X_0}_0 \cap A_0}.
    \end{equation*}
    Then ${\{g_t\}}_{t\in [0,1]}$ is a family of stratified diffeomorphisms, with $g_0 = \Id$ and $g_1 = g^\w$.

    Differentiating ${\left(\Phi^{s}_{(\varphi X_t)}\right)}^*\w_s$ we obtain
    \begin{align*}
        \frac{d}{ds}\bigg\lvert_{t_0}{\left(\Phi^{s}_{(\varphi X_t)}\right)}^*\w_s
        &= {\left(\Phi^{t_0}_{(\varphi X_t)}\right)}^*\left(\frac{d}{ds}\bigg\lvert_{t_0}\w_s + \L_{(\varphi X_{t_0})}\w_{t_0} \right) \\
        &= {\left(\Phi^{t_0}_{(\varphi X_t)}\right)}^*\left(-\alpha + d\underbrace{\i_{(\varphi X_{t_0})}\w_{t_0}}_{\varphi \beta} + \i_{(\varphi X_{t_0})}\underbrace{d\w_{t_0}}_0\right) \\
        &= {\left(\Phi^{t_0}_{(\varphi X_t)}\right)}^*\left(-\alpha + d\varphi \wedge \beta + \varphi d\beta \right) \\
        &= {\left(\Phi^{t_0}_{(\varphi X_t)}\right)}^*\left((-1 + \varphi)\alpha + d\varphi \wedge \beta\right).
    \end{align*}
    Recall that $\alpha$ vanishes on $T^ZA_0$ and, by \Cref{proposition:beta_properties}\eqref{property:beta_vanishes_on_zariski_tangent}, so does $\beta$.
    It follows that $\left((-1 + \varphi)\alpha + d\varphi \wedge \beta\right)$ vanishes on $T^ZA_0$.
    Since $D\Phi^{t_0}_{(\varphi X_t)}$ sends $T^ZA_0$ to $T^ZA_0$, the $2$-form
    \begin{equation*}
        \frac{d}{ds}\bigg\lvert_{t_0}{\left(\Phi^{s}_{(\varphi X_t)}\right)}^*\w_s = {\left(\Phi^{t_0}_{(\varphi X_t)}\right)}^*\left((-1 + \varphi)\alpha + d\varphi \wedge \beta\right)
    \end{equation*}
    vanishes on $T^ZA_0$.

    Let $p_0\in U^{X_0}_0 \cap A_0$ and $u_0,v_0 \in T^Z_{p_0}A_0$. Then
    \begin{equation*}
        \frac{d}{ds}\bigg\lvert_{t_0}\i_{u_0}\i_{v_0}{\left(\Phi^{s}_{(\varphi X_t)}\right)}^*\w_s = \i_{u_0}\i_{v_0}\frac{d}{ds}\bigg\lvert_{t_0}{\left(\Phi^{s}_{(\varphi X_t)}\right)}^*\w_s = 0,
    \end{equation*}
    and therefore $g_t^*\w_t = \w_0$ as Zariski forms on $A_0$.

    The second property follows from \Cref{proposition:G_w_fixes_W_X_0} and $\supp\varphi \subset V^{X_0}_0$.
\end{proof}
\begin{remark}
    $\w_t = \w_0$ as Zariski forms on $U^{X_0}_0 \cap A_0$ for all $t\in [0,1]$. So, equivalently, the stratified diffeomorphisms $g_t$ preserve $\w_0$ as a Zariski form on $U^{X_0}_0 \cap A_0$ for all $t\in [0,1]$.
\end{remark}

We are now ready to prove \Cref{thm:extending_to_stratum}.

\begin{proof}[Proof of \Cref{thm:extending_to_stratum}]
    Let
    \begin{equation*}
        \G^{X_0} \coloneq \G^{X_0}_{\text{diff}} \circ \G^\w: U^{X_0}_0 \to U^{X_1}_1
    \end{equation*}
    and
    \begin{equation*}
        g^{X_0} \coloneq \G^{X_0}_{\text{diff}}\Big\lvert_{U^{X_0}_0 \cap A_0} \circ g^\w:U^{X_0}_0 \cap A_0 \to U^{X_1}_1 \cap A_1.
    \end{equation*}
    Write
    \begin{equation*}
        V^{X_1}_1 \coloneq \G^{X_0}\left(V^{X_0}_0\right), \quad W^{X_1}_1 \coloneq \G^{X_0}\left(W^{X_0}_0\right).
    \end{equation*}

    We claim that the data $\left(U^{X_i}, V^{X_i}_i, W^{X_i}_i, \G^{X_0}, g^{X_0}\right)$ is an extension for $\left(g^{\leq (d-1)}, \G^{\leq (d-1)} \right)$ to $(X_0, X_1)$, see \Cref{definition:extension_for_stratum}.

    By \Cref{proposition:creating_V_X_0}\eqref{property:image_of_V_0_closed}, $\G^{X_0}_{\text{diff}}\left(\overline{V^{X_0}_0}\right)$ is closed in $M^{\geq d}_1$. Since $\supp\varphi \subset V^{X_0}_0$, the diffeomorphism $\G^\w = \Phi^{1}_{(\varphi X_t)}$ preserves $V^{X_0}_0$, and therefore
    \begin{equation*}
        \G^{X_0}\left(\overline{V^{X_0}_0}\right)= \G^{X_0}_{\text{diff}}\left(\overline{V^{X_0}_0}\right)
    \end{equation*}
    is closed in $M^{\geq d}_1$. It follows that
    \begin{equation*}
        X_i \subset W^{X_i}_i \subset \overline{W^{X_i}_i} \subset V^{X_i}_i \subset \overline{V^{X_i}_i} \subset U^{X_i}_i
    \end{equation*}
    for $i=0,1$, where the closures are taken in $M^{\geq d}_i$.

    By \Cref{proposition:restricts_to_symplec_on_W_X_0}, $\G^{X_0}$ restricts to a symplectomorphism on $W^{X_0}_0$.

    By \Cref{proposition:G_w_fixes_W_X_0}, $\G^\w$ fixes $\overline{W^{\leq (d-1)}_0} \cap U^{X_0}_0$.
    By \Cref{proposition:existence_of_G_X_0}\eqref{property:G_X_0_agrees_with_G_leq_on_V_leq}, $\G^{X_0}_{\text{diff}}$ agrees with $\G^{\leq (d-1)}$ on $V^{\leq (d-1)}_0 \cap U^{X_0}_0$, and in particular on $W^{\leq (d-1)}_0 \cap U^{X_0}_0$.
    It follows that $\G^{X_0}$ agrees with $\G^{\leq (d-1)}$ on $W^{\leq (d-1)}_0 \cap U^{X_0}_0$.

    Composing the isotopies $g_t$ from \Cref{proposition:g_w_isotopic_to_id} with $\G^{X_0}_{\text{diff}}\Big\lvert_{U^{X_0}_0 \cap A_0}$, we obtain a family of stratified diffeomorphisms
    \begin{equation*}
        g_t:U^{X_0}_0 \cap A_0 \to U^{X_1}_1 \cap A_1
    \end{equation*}
    with $g_0 = g^{\leq (d-1)}\Big\lvert_{U^{X_0}_0 \cap A_0}$ and $g_1 = g^{X_0}$. For each $t\in [0,1]$, the stratified diffeomorphism $g_t$ satisfies the following properties.
    \begin{enumerate}
        \item $g_t^*\w_1 = \w_0$ as Zariski forms on $A_0$.
        \item $g_t$ agrees with $g^{\leq (d-1)}$ on $W^{\leq (d-1)}_0 \cap U^{X^d_0}_0 \cap A_0$ and on $\left(U^{X^d_0}_0 \setminus \overline{V^{X^d_0}_0}\right) \cap A_0$.
    \end{enumerate}
    In particular, ${\left(g^{X_0}\right)}^*\w_1 = \w_0$ as Zariski forms on $A_0$ and $g^{X_0}$ agrees with $g^{\leq (d-1)}$ on $W^{\leq (d-1)}_0 \cap U^{X^d_0}_0 \cap A_0$ and on $\left(U^{X^d_0}_0 \setminus \overline{V^{X^d_0}_0}\right) \cap A_0$.
\end{proof}

\section{Exactness for isotropic stratified subspaces}\label{section:exactness_for_isotropic}

Recall that for a symplectic manifold $(M, \w)$ and an isotropic submanifold $A \subset M$, the symplectic form is exact on a tubular neighbourhood of $A$. \Cref{thm:exactness_of_symplectic_form} is an analogue of this assertion for our case, where $A$ is allowed to be singular. In this section, we relax the assumption on $A$ to just Whitney (B) regularity, rather than requiring it to be smoothly locally trivial with conical fibers.

\begin{thm}\label{thm:exactness_of_symplectic_form}
    Let $(V, \w)$ be a symplectic manifold and $(A, \S)$ a Whitney (B) regular isotropic stratified subspace of it. Let $R:[0,1]\times V \to V$ be a smooth weak deformation retraction of $V$ to $A$ (\Cref{definition:smooth_weak_deformation_retraction}).
    Then $\w$ is exact, with a primitive given by
    \begin{equation*}
        {\left(\pi_V\right)}_*R^*\w = \int_{0}^{1}i_t^* \i_{\frac{\partial}{\partial t}}R^*\w \,dt.
    \end{equation*}
\end{thm}

We prove \Cref{thm:exactness_of_symplectic_form} further below. In the proof, we use the following topological lemmas.

For a topological space $V$ and a subset $A$, we denote the interior of $A$ in $V$ by $\interior_V(A)$.

\begin{lemma}\label{lemma:closed_set_interior_complement_dense}
    Let $V$ be a topological space and $A$ a closed subset. Then
    \begin{equation*}
        \left(V\setminus A\right) \cup \interior_V(A)
    \end{equation*}
    is dense in $V$.
\end{lemma}
\begin{proof}
    Immediate from definitions.
\end{proof}

\begin{lemma}\label{lemma:closed_filtration_interiors_dense}
    Let $V$ be a topological space, $N\in \N$ and
    \begin{equation*}
        V = V^{\leq N} \supset V^{\leq (N-1)} \supset \ldots \supset V^{\leq 1} \supset V^{\leq 0} \supset V^{\leq (-1)} = \emptyset
    \end{equation*}
    a descending chain of closed subsets.
    Then
    \begin{equation*}
        \bigcup_{d=0}^N \interior_V\left(V^{\leq d} \setminus V^{\leq (d-1)}\right)
    \end{equation*}
    is dense in $V$.
\end{lemma}
\begin{proof}
    We prove by induction on $N$.

    The base $N=1$ is \Cref{lemma:closed_set_interior_complement_dense}.

    For the induction step $N$, define $V' \coloneq \interior_{V}\left(V^{\leq (N-1)} \right)$ and consider the descending chain
    \begin{equation*}
        V' = \left(V'\cap V^{\leq (N-1)}\right) \supset \ldots \supset \left(V' \cap V^{\leq 1}\right) \supset \left(V' \cap V^{\leq 0}\right) \supset \left(V' \cap V^{\leq (-1)}\right) = \emptyset
    \end{equation*}
    of closed subsets in $V'$. By the induction assumption for $N-1$, we get that
    \begin{equation*}
        \bigcup_{d=0}^{N-1} \interior_{V'}\left(\left(V' \cap V^{\leq d}\right) \setminus \left(V' \cap V^{\leq (d-1)}\right)\right)
        = \bigcup_{d=0}^{N-1} \interior_{V'}\left( \left(V^{\leq d} \setminus V^{\leq (d-1)}\right) \cap V'\right)
    \end{equation*}
    is dense in $V'$. Note that since $V'$ is open in $V$ we have
    \begin{equation*}
        \interior_{V'}\left( \left(V^{\leq d} \setminus V^{\leq (d-1)}\right) \cap V'\right) = \interior_{V}\left(V^{\leq d} \setminus V^{\leq (d-1)}\right) \cap V'.
    \end{equation*}
    Therefore,
    \begin{equation*}
        \bigcup_{d=0}^{N-1} \interior_{V}\left(V^{\leq d} \setminus V^{\leq (d-1)}\right) \cap V'
    \end{equation*}
    is dense in $V'$.

    By \Cref{lemma:closed_set_interior_complement_dense} applied to the pair $V=V^{\leq N}, V^{\leq (N-1)}$, we get that
    \begin{equation*}
        \left( V^{\leq N} \setminus V^{\leq (N-1)}\right) \cup \interior_{V}\left(V^{\leq (N-1)} \right) = \left( V^{\leq N} \setminus V^{\leq (N-1)}\right) \cup V'
    \end{equation*}
    is dense in $V$. It follows that
    \begin{equation*}
        \left( V^{\leq N} \setminus V^{\leq (N-1)}\right) \cup \bigcup_{d=0}^{N-1} \interior_{V}\left(V^{\leq d} \setminus V^{\leq (d-1)}\right) \cap V'
    \end{equation*}
    is dense in $V$, which implies the result.
\end{proof}

We are now ready to prove \Cref{thm:exactness_of_symplectic_form}.

\begin{proof}[Proof of \Cref{thm:exactness_of_symplectic_form}]
    As in \Cref{definition:smooth_weak_deformation_retraction}, we write $r^t:=R(t,\cdot)$. Let $r\coloneq r^0$.
    Apply \Cref{lemma:homotopy_property_of_fiber_integration} to $\gamma=R^*\w$. Since $d\w=0$ and $r^1=\Id_V$, we get
    \begin{equation*}
        \w -r^*\w = d\left( (\pi_V)_*R^*\w \right).
    \end{equation*}
    Thus, it is enough to prove that $r^*\w =0$.

    By Whitney (B) regularity, the subsets $A^{\leq d}$ are closed in $V$ and, by continuity, the subsets $V^{\leq d} \coloneq r^{-1}(A^{\leq d})$ are closed in $V$. By \Cref{lemma:closed_filtration_interiors_dense}, the union
    \begin{equation*}
        \bigcup_{d=0}^N \interior_V\left(V^{\leq d} \setminus V^{\leq (d-1)}\right)
    \end{equation*}
    is dense in $V$. It follows that it is enough to prove that $r^*\w = 0$ on the open subsets $\interior_V\left(V^{\leq d} \setminus V^{\leq (d-1)}\right)$ for every $d$.

    Let $q\in \interior_V\left(V^{\leq d} \setminus V^{\leq (d-1)}\right)$, then $r(q) \in X$ for some $X\in \S^d$. For all $v\in T_qV$ we have
    \begin{equation*}
        D_qr(v) \in T_{r(q)}X
    \end{equation*}
    and the result follows from the assumption that $X$ is isotropic.
\end{proof}

Recall \cite[Remark~11.11]{cieliebak2012stein} that for a closed manifold $L$, there exists a Weinstein structure on the cotangent bundle $T^*L$, such that the zero section is its Lagrangian skeleton. \Cref{thm:exactness_of_symplectic_form} indicates that an analogous phenomenon might happen in our singular setting.

\begin{question}\label{question:Liouville_domain}
    Let $(V, \w)$ be a symplectic manifold and $(A, \S)$ a Lagrangian stratified subspace of it. Assume that $V$ admits a smooth weak deformation retraction to $A$.
    Under what conditions are there a neighbourhood $W$ of $A$ in $V$ with a smooth boundary and a vector field $Z$ such that $\left(W, \w, Z \right)$ is a Liouville domain, i.e., $d\i_Z\w = \w$ and $Z$ is outward pointing along $\partial W$? Furthermore, under what conditions is there a Weinstein structure on $W$?
\end{question}

\begin{question}\label{question:Lagrangian_skeleton}
    Under what conditions is there a Liouville domain $\left(W, \w, Z \right)$ such that $A$ is the Lagrangian skeleton of $W$?
\end{question}

In a future work we intend to study the relation between our construction, Liouville domains, and their Lagrangian skeleta.

\section{Explicit examples}\label{section:explicit_examples}

In this section, we construct stratified diffeomorphisms satisfying the assumption of \Cref{thm:main_coisotropic} for two examples of singular Lagrangian subspaces: Lagrangian $(p,q)$-pinwheels with $d-1$ attached spheres, and solid mutation configurations. The construction is possible in these examples for the following reasons: 
\begin{enumerate}
    \item they are Lagrangian, so the symplectic forms vanish on the regular part;
    \item there exist symplectic normal forms near their singular parts, detailed below;
    \item there are no topological obstructions to globally extending diffeomorphisms that are defined near the singular parts.
\end{enumerate}

In \Cref{subsection:lagrangian_pinwheels}, we use the following standard normal form result.

\begin{lemma}\label{lemma:normal_form_around_transverse_intersection}
    Let $(M,\w)$ be a symplectic manifold and let $L, L'\subset M$ be Lagrangian submanifolds. Assume that $L$ and $L'$ intersect transversely at a point $p$. Then there exists a Darboux chart $(x,y):U\to W \subset \R^{2n}$ around $p$ that identifies $L\cap U$ with $\{y=0\} \cap W$ and $L'\cap U$ with $\{x=0\} \cap W$.
\end{lemma}

In \Cref{subsection:solid_mutation_conf}, we use the following generalization, due to Po\'zniak.

\begin{lemma}[{\cite[Proposition~3.4.1]{pozniak_lagrangian_clean_intersection}}]\label{lemma:normal_form_around_clean_intersection}
    Let $(M,\w)$ be a symplectic manifold with Lagrangian submanifolds $L, L'\subset M$. Assume that $L$ and $L'$ intersect cleanly along a compact submanifold $N$, i.e., $N = L\cap L'$ and for every $p\in N$ we have $T_pN = T_pL\cap T_pL'$.
    Then there exist a vector bundle $E\to N$, a neighbourhood $U$ of $N$ in $M$, a neighbourhood $V$ of $N$ in $T^*E$, and a symplectomorphism $\psi:U \to V$ such that the following holds.
    \begin{enumerate}
        \item $\psi\Big\lvert_N$ is the zero section of $T^*E\Big\lvert_N$.
        \item $\psi\left(L\cap U\right) = E\cap V$.
        \item $\psi\left(L'\cap U\right) = TN^{\text{ann}} \cap V$, where
        \begin{equation*}
            TN^{\text{ann}} = \{\alpha \in T^*E\Big\lvert_N:\: \alpha\Big\lvert_{TN}\equiv 0\}.
        \end{equation*}
    \end{enumerate}
\end{lemma}

\subsection{Lagrangian pinwheels with attached spheres}\label{subsection:lagrangian_pinwheels}

Let $D\subset \C$ be the closed unit $2$-disk and let $(M, \w)$ be a symplectic $4$-manifold.

\begin{definition}[{\cite[Definition~3.1]{symplectic_rational_blowup},\cite[Definition~2.3]{evans2018markov}}]\label{definition:good_pinwheel}
   Let $p,q\in \N$ be coprime integers. A \textbf{good Lagrangian $\mathbf{(p,q)}$-pinwheel embedding} is a Lagrangian immersion $f:D \to M$ such that the following holds.
   \begin{enumerate}
        \item $f\Big\lvert_{D \setminus \partial D}$ is an embedding.
        \item $f(x) = f(y)$ if and only if $x /y$ is a $p$-th root of unity. It follows that $C\coloneq f\left(\partial D\right)$ is an embedded circle in $M$.
        \item There exists a neighbourhood $U$ of $C$ in $M$, a neighbourhood $V$ of $\partial D$ in $D$, and coordinates $\left(\varphi, s, z = x+iy\right):U \to \left(\R / 2\pi \Z\right)\times \R \times \C$ such that        
        \begin{enumerate}
            \item for any $e^{i\theta} \in \partial D$, we have $(\varphi, s,z)\left(f\left(e^{i\theta}\right)\right) = \left(p\theta, 0,0\right)$,
            \item $\w = d\varphi \wedge ds + dx \wedge dy$, and
            \item $\left((\varphi, s,z)\circ f\right)\Big\lvert_V$ is given by
            \begin{equation*}
                re^{i\theta} \mapsto \left(p\theta,\, q{(1-r)}^2/2p,\, (1-r)e^{-iq\theta}\right).
            \end{equation*}
        \end{enumerate}
   \end{enumerate}

   A \textbf{good Lagrangian $\mathbf{(p,q)}$-pinwheel} is a closed subset $L_{pq}\subset M$ that is the image of such an embedding $f:D\to M$. 
   The subset $C=f\left(\partial D\right)$ is the \textbf{core circle} of $L_{pq}$.
\end{definition}

\begin{definition}
    Let $p,q\in \N$ be coprime integers and let $d\in \N$. Let $L_{pq}$ be a good Lagrangian $(p,q)$-pinwheel and, if $d\geq 2$, let ${\{S^2_i\}}_{1\leq i\leq d-1}$ be Lagrangian $2$-spheres. 
    Assume that any two consecutive Lagrangian $2$-spheres intersect transversely at a single point, and any non-consecutive Lagrangian $2$-spheres are disjoint.
    Assume that $L_{pq}$ intersects $S^2_1$ transversely at a single point not on the core circle and is disjoint from $S^2_i$ for $i>1$.

    The closed subset
    \begin{equation*}
        L_{dpq} \coloneq L_{pq}\cup \bigcup_{1\leq i\leq d-1}S^2_i
    \end{equation*}
    is a \textbf{good Lagrangian $\mathbf{(p,q)}$-pinwheel with $\mathbf{d-1}$ attached spheres}. For $d=1$, we have $L_{1pq} = L_{pq}$.
\end{definition}

\begin{proposition}\label{proposition:stratification_of_pinwheel}
    Let $L_{dpq}$ be a good Lagrangian $(p,q)$-pinwheel with $d-1$ attached spheres. Consider the decomposition $\S$ of $L_{dpq}$ into the $d-1$ intersection points, the core circle, and the $d$ connected components of the rest. Then $\S$ is a stratification, and $(A, \S)$ is smoothly locally trivial with conical fibers, Lagrangian, and strongly coisotropic.
\end{proposition}
\begin{proof}
    We verify that transverse intersection points and points in the core circle admit conical charts as in \Cref{definition:smoothly_locally_trivial}. 
    
    For transverse intersection points, the existence of a conical chart follows from \Cref{lemma:normal_form_around_transverse_intersection}.
    
    Let $U$ and $(\varphi, s, z):U \to \left(\R / 2\pi \Z\right) \times \R \times \C$ be as in \Cref{definition:good_pinwheel}. 
    Let $c\in C$ be a point in the core circle. Use the coordinates to identify its neighbourhood in $M$ with a neighbourhood of $c=(\varphi_0,0,0)$ in $(\varphi_0 -\varepsilon, \varphi_0 + \varepsilon)\times \R \times \C$.
    Define the smooth map
    \begin{equation*}
        h:(\varphi_0 -\varepsilon, \varphi_0 + \varepsilon)\times \R \times \C \to (\varphi_0 -\varepsilon, \varphi_0 + \varepsilon)\times \R \times \C
    \end{equation*}
    by
    \begin{equation*}
        h(\varphi, s, z) = \left(\varphi, s - q\norm{z}^2/2p, e^{iq\varphi/p}z\right).
    \end{equation*}
    
    Let $\theta \in \R,\,\varphi_\theta \in (\varphi_0 - \varepsilon, \varphi_0 + \varepsilon)$, and $k_\theta \in \Z$ be such that $p\theta = \varphi_{\theta} + 2\pi k_{\theta}$. Then
    \begin{equation}\label{equation:pinwheel_conical_nbhd}
        \begin{aligned}
            (h\circ f)(re^{i\theta}) 
            &= h\left(\varphi_\theta, q{(1-r)}^2/2p, (1-r)e^{-iq\left(\varphi_{\theta}/p + 2\pi k_{\theta}/p\right)}\right) \\
            &= \left(\varphi_\theta, q{(1-r)}^2/2p - q{(1-r)}^2/2p, e^{iq\varphi_\theta/p}(1-r)e^{-iq\left(\varphi_{\theta}/p + 2\pi k_{\theta}/p\right)}\right) \\
            &= \left(\varphi_\theta, 0, (1-r)e^{-iq2\pi k_{\theta}/p}\right).
        \end{aligned}
    \end{equation}
    In a small enough neighbourhood $W$ of $c$, the map $h$ gives a chart in which, by \Cref{equation:pinwheel_conical_nbhd}, the intersection of the $(p,q)$-pinwheel with $W$ is identified with 
    \begin{equation*}
        \left((\varphi_0 -\varepsilon, \varphi_0 + \varepsilon) \times \{0\} \times \{re^{i\theta}: r\geq 0, e^{ip\theta} = 1\}\right)\cap h(W).
    \end{equation*}
    Since $\{re^{i\theta}: r\geq 0, e^{ip\theta} = 1\} \subset \C$ is conical, this chart is conical.

    The formula for $f$ implies that $(L_{dpq},\S)$ is Lagrangian. By \Cref{remark:depth_1_strongly_coisotropic}, it is strongly coisotropic.
\end{proof}

We now verify that good Lagrangian pinwheels satisfy the hypotheses of \Cref{thm:main_coisotropic}.

\begin{proposition}\label{proposition:L_dpq_stratified_diffeo}
    Let $(M,\w)$ and $(\widetilde{M},\widetilde{\w})$ be symplectic manifolds. Let $p,q\in \N$ be coprime integers and let $d\in \N$. Let $L_{dpq} \subset M$ and $\widetilde{L}_{dpq}\subset \widetilde{M}$ be good Lagrangian $(p,q)$-pinwheels with $d-1$ attached spheres, and let $\S$ and $\widetilde{\S}$ be their stratifications from \Cref{proposition:stratification_of_pinwheel}.
    Then there exists a stratified diffeomorphism $g:\left(L_{dpq}, \S\right) \to \left(\widetilde{L}_{dpq}, \widetilde{\S}\right)$ with $g^*\widetilde{\w} = \w$ as Zariski forms on $L_{dpq}$.
\end{proposition}
\begin{proof}
    Write $L_{dpq} = L_{pq}\cup \bigcup_{i}S^2_i$ and $\widetilde{L}_{dpq} = \widetilde{L}_{pq}\cup \bigcup_{i}\widetilde{S}^2_i$, and let $C\subset L_{pq}$ and $\widetilde{C} \subset \widetilde{L}_{pq}$ be the core circles. Let $a_1 =L_{pq}\cap S^2_1$ and $\widetilde{a}_1 = \widetilde{L}_{pq}\cap \widetilde{S}^2_1$,  and, for $2\leq i \leq d-1$, let $a_i =S^2_i\cap S^2_{i+1}$ and $\widetilde{a}_i = \widetilde{S}^2_i\cap \widetilde{S}^2_{i+1}$.

    Let $U_0\subset M$ and $\widetilde{U}_0\subset \widetilde{M}$ be open neighbourhoods of $C$ and $\widetilde{C}$ respectively, and let
    \begin{equation*}
        \psi:U_0 \to \left(\R / 2\pi \Z\right) \times \R \times \C \qquad\text{and}\qquad \widetilde{\psi}:\widetilde{U}_0 \to \left(\R / 2\pi \Z\right) \times \R \times \C
    \end{equation*}
    be as in \Cref{definition:good_pinwheel}. 
    Potentially shrinking the neighbourhoods $U_0$ and $\widetilde{U}_0$, the composition ${\left(\widetilde{\psi}\right)}^{-1}\circ \psi:U_0 \to \widetilde{U}_0$ is a symplectomorphism which sends $L_{pq} \cap U_0$ to $\widetilde{L}_{pq} \cap \widetilde{U}_0$ and $C$ to $\widetilde{C}$.
    The restriction
    \begin{equation*}
        g_0\coloneq \left({\left(\widetilde{\psi}\right)}^{-1}\circ \psi\right)\Big\lvert_{L_{pq} \cap U_0}:
        L_{pq} \cap U_0 \to \widetilde{L}_{pq} \cap \widetilde{U}_0
    \end{equation*}
    is a stratified diffeomorphism with $g_0^*\widetilde{\w} = \w$ as Zariski forms on $L_{dpq} \cap U_0$.

    Fix orientations on $L_{pq} \setminus C$ and $\widetilde{L}_{pq} \setminus \widetilde{C}$
    such that $g_0\Big\lvert_{\left(L_{pq}\cap U_0\right) \setminus C}$ is orientation preserving. 
    For each $1\leq i \leq d-1$, fix orientation on $S^2_i$ and $\widetilde{S}^2_i$ such that all the transverse intersections are positive. 

    For each $1\leq i \leq d-1$, we now construct diffeomorphism between neighbourhoods of the transverse intersections. Fix $i$. Applying \Cref{lemma:normal_form_around_transverse_intersection} at the transverse intersection points $a_i,\widetilde{a}_i$, we obtain neighbourhoods $U_i, \widetilde{U}_i$ of $a_i,\widetilde{a}_i$ in $M,\widetilde{M}$ respectively, and symplectomorphisms
    \begin{equation*}
        (x_1, x_2, y_1, y_2):U_i \to \R^4 \qquad\text{and}\qquad (\widetilde{x}_1, \widetilde{x}_2, \widetilde{y}_1, \widetilde{y}_2):\widetilde{U}_i \to \R^4
    \end{equation*}
    that identify the Lagrangian components with $\{x=0\},\,\{y=0\},\,\{\widetilde{x}=0\}$, and $\{\widetilde{y}=0\}$.
    Potentially shrinking the neighbourhoods, the composition
    \begin{equation*}
        {(\widetilde{x}_1, \widetilde{x}_2, \widetilde{y}_1, \widetilde{y}_2)}^{-1}\circ (x_1, x_2, y_1, y_2):U_i \to \widetilde{U}_i
    \end{equation*}
    is a symplectomorphism which restricts to a stratified diffeomorphism
    \begin{equation*}
        g_i:L_{dpq} \cap U_i \to \widetilde{L}_{dpq} \cap \widetilde{U}_i
    \end{equation*}
    such that the following holds. 
    \begin{enumerate}
        \item\label{property:g_i_property_1} $g_i^*\widetilde{\w} = \w$ as Zariski forms on $L_{dpq}\cap U_i$.
        \item $g_i\left(S^2_i\cap U_i\right)=\widetilde{S}^2_i \cap \widetilde{U}_i$.
        \item If $i=1$, then $g_1\left(L_{pq}\cap U_1\right) = \widetilde{L}_{pq}\cap \widetilde{U}_1$.
        \item\label{property:g_i_property_4}  If $2\leq i \leq d-1$, then $g_i\left(S^2_{i-1}\cap U_i\right)=\widetilde{S}^2_{i-1} \cap \widetilde{U}_i$.
    \end{enumerate}
    Possibly replacing $x_1,y_1$ with $-x_1, -y_1$, we arrange that the restriction 
    \begin{equation*}
        g_i\Big\lvert_{S^2_i \cap U_i}:S^2_i \cap U_i \to \widetilde{S}^2_i \cap \widetilde{U}_i
    \end{equation*}
    preserves orientation and properties \eqref{property:g_i_property_1}--\eqref{property:g_i_property_4} still hold. 

    If $d=1$, then $L_{pq} \setminus C$ and $\widetilde{L}_{pq}\setminus \widetilde{C}$ are Lagrangian open disks, and $g_0:L_{pq} \cap U_0 \to \widetilde{L}_{pq} \cap \widetilde{U}_0$ extends to a stratified diffeomorphism $g:L_{pq} \to \widetilde{L}_{pq}$ that agrees with $g_0$ in a neighbourhood of $C$.
    If $d\geq 2$, then the punctured pinwheels $L_{pq} \setminus \left(C \cup \{a_1\}\right)$ and $\widetilde{L}_{pq} \setminus \left(\widetilde{C} \cup \{\widetilde{a}_1\}\right)$ and the punctured spheres $S^2_i \setminus \{a_{i},a_{i+1}\}$ and $\widetilde{S}^2_i \setminus \{\widetilde{a}_{i}, \widetilde{a}_{i+1}\}$ for $1\leq i\leq d-2$ are Lagrangian annuli, and the punctured spheres $S^2_{d-1} \setminus \{a_{d-1}\}$ and $\widetilde{S}^2_{d-1} \setminus \{\widetilde{a}_{d-1}\}$ are Lagrangian open disks.
    Since there are no obstructions coming from orientation, we can extend ${\{g_i\}}_i$ to a stratified diffeomorphism $g$ which agrees with $g_0$ in a neighbourhood of $C$ and with ${\{g_i\}}_{1\leq i \leq d-1}$ in neighbourhoods of $a_i$. 
    
    Finally, the condition $g^*\widetilde{\w} = \w$ follows from the corresponding property of ${\{g_i\}}_{0\leq i \leq d-1}$ and since both forms vanish on the regular part.
\end{proof}

\begin{cor}\label{cor:L_dpq_determines_neigbourhood}
    With the assumptions and notation of \Cref{proposition:L_dpq_stratified_diffeo}, there exist neighbourhoods $U,\widetilde{U}$ of $L_{dpq},\widetilde{L}_{dpq}$ in $M,\widetilde{M}$ respectively and a symplectomorphism $\G:U \to \widetilde{U}$ which restricts to a stratified diffeomorphism $\left(L_{dpq}, \S\right) \to \left(\widetilde{L}_{dpq}, \widetilde{\S}\right)$.
\end{cor}
\begin{proof}
    Apply \Cref{proposition:L_dpq_stratified_diffeo,proposition:L_dpq_stratified_diffeo} and \Cref{thm:main_coisotropic}.
\end{proof}

\subsection{Solid mutation configurations}\label{subsection:solid_mutation_conf}

\begin{definition}[{\cite[Definition~2.1]{solid_mutation_conf}}]\label{definition:solid_mutation}
    Let $(M^{2n}, \w)$ be a symplectic manifold. A \textbf{solid mutation configuration} consists of a Lagrangian torus $L \coloneq \T^n \subset M$ and a Lagrangian solid torus $\mathfrak{T}\coloneq D^2\times \T^{n-2} \subset M$ such that
    \begin{enumerate}
        \item $L$ and $\mathfrak{T}$ intersect cleanly along $\partial \mathfrak{T} \simeq \T^{n-1}$, and
        \item the pair $\left(L, \partial \mathfrak{T}\right)$ is diffeomorphic to the standard pair $(\T^n, \T^{n-1})$.
    \end{enumerate}
\end{definition}

\begin{remark}
    If $n=2$, then a solid mutation configuration is a mutation configuration in the sense of \cite[Definition~4.10]{PASCALEFF2020106850}.
\end{remark}

\begin{proposition}\label{proposition:solid_mutation_stratification}
    Let $(L, \mathfrak{T})$ be a solid mutation configuration. Consider the decomposition $\S$ of $L \cup \mathfrak{T}$ into $\partial \mathfrak{T},\, \interior\left(\mathfrak{T}\right)$, and $\left(L \setminus \partial \mathfrak{T}\right)$.
    Then $\S$ is a stratification, and $\left(L \cup \mathfrak{T},\S\right)$ is smoothly locally trivial with conical fibers, Lagrangian, and strongly coisotropic.
\end{proposition}
\begin{proof}
    To verify that points in $\partial \mathfrak{T}$ admit conical charts as in \Cref{definition:smoothly_locally_trivial}, apply \Cref{lemma:normal_form_around_clean_intersection} and take a local trivialization of $E$ over $N$. 
    \Cref{definition:solid_mutation} implies that
    $\left(L \cup \mathfrak{T},\S\right)$ is Lagrangian, and by \Cref{remark:depth_1_strongly_coisotropic} it is strongly coisotropic.
\end{proof}

To construct a stratified diffeomorphism satisfying the hypotheses of \Cref{thm:main_coisotropic} we use the following Lemmas.

\begin{lemma}[{\cite[Theorem~III.3.3]{kosinski2013differential}}]\label{lemma:uniqueness_of_collars}
    Let $M$ be a manifold with boundary and let $\psi:U \to \widetilde{U}$ be a diffeomorphism of collar neighbourhoods of $\partial M$ that restricts to the identity on $\partial M$. 
    Then there exists a diffeomorphism $M\to M$ that restricts to $\psi$ on a smaller neighbourhood of $\partial M$.
\end{lemma}

\begin{lemma}\label{lemma:uniqueness_of_standard_codim_1_subtorus}
    Let $(\T^n, \T^{n-1})$ be the standard torus pair and let $\psi:U \to \widetilde{U}$ be a diffeomorphism of tubular neighbourhoods of $\T^{n-1}$ that restricts to the identity on $\T^{n-1}$.
    Then there exists a diffeomorphism $\T^n \to \T^n$ that restricts to $\psi$ on a smaller neighbourhood of $\T^{n-1}$.
\end{lemma}
\begin{proof}
    Identify $\T^n \simeq \T^{n-1} \times S^1$. 
    Using the vector bundle structures on $U$ and $\widetilde{U}$ and the isomorphism $\text{GL}_1(\R) \simeq \R^*$, we may assume that $U = \widetilde{U} \simeq \T^{n-1} \times (-\varepsilon, \varepsilon)$ and that $\psi:\T^{n-1} \times (-\varepsilon, \varepsilon) \to \T^{n-1} \times (-\varepsilon, \varepsilon)$ is either the identity or $(t, x) \mapsto (t,-x)$. The map $(t, x) \mapsto (t,-x)$ can be extended to $\T^{n-1} \times S^1 \to \T^{n-1} \times S^1$ by
    \begin{equation*}
        \left(t, \theta\right) \mapsto \left(t, -\theta\right),
    \end{equation*}
    and similarly for the identity.
\end{proof}

\begin{remark}
    The analogue of \Cref{lemma:uniqueness_of_standard_codim_1_subtorus} fails for subtori of codimension greater than $1$. 
    Indeed, let $S^1 \times \{0\} \times \{0\} \subset \T^3$ and let $S^1 \times D^2 \subset \T^3$ be a tubular neighbourhood of it. The diffeomorphism
    \begin{equation*}
        (\theta, z) \mapsto (\theta, e^{i\theta}z)
    \end{equation*}
    cannot be extended to a diffeomorphism $\T^3 \to \T^3$, as we now explain.
    
    If such an extension exists, then it induces a diffeomorphism of $\T^3 \setminus \left(S^1 \times D\right)$ to itself, and therefore an automorphism of
    \begin{equation*}
        \pi_1\left(\T^3 \setminus \left(S^1 \times D\right)\right) = \pi_1\left(S^1 \times \left(\T^2 \setminus D\right)\right) \simeq \Z c \oplus F_2(a,b)
    \end{equation*}
    where $c$ is the generator of $\pi_1\left(S^1\right)$ and $F_2(a,b)$ is the free group generated by $a$ and $b$. Writing $\varepsilon$ for the unit in $F_2(a,b)$, this automorphism maps $(c, \varepsilon)$ to $(c, aba^{-1}b^{-1})$. 
    Since $( c,\varepsilon)$ is in the center of $\Z \oplus F_2(a,b)$ and $(c, aba^{-1}b^{-1})$ is not, no such automorphism exists.
\end{remark}

\begin{proposition}\label{proposition:solid_mutation_stratified_diffeo}
    Let $(M, \w)$ and $(\widetilde{M},\widetilde{\w})$ be symplectic manifolds. Let $(L, \mathfrak{T})\subset M$ and $(\widetilde{L}, \widetilde{\mathfrak{T}})\subset \widetilde{M}$ be solid mutation configurations, and let $\S$ and $\widetilde{\S}$ be their stratifications from \Cref{proposition:solid_mutation_stratification}. Then there exists a stratified diffeomorphism $g:\left(L \cup \mathfrak{T},\S\right) \to \left(\widetilde{L} \cup \widetilde{\mathfrak{T}},\widetilde{\S}\right)$ with $g^*\widetilde{\w} = \w$ as Zariski forms on $L \cup \mathfrak{T}$.
\end{proposition}
\begin{proof}
    Fix a diffeomorphism of pairs $g_1:\left(\mathfrak{T}, \partial \mathfrak{T}\right) \to \left(\widetilde{\mathfrak{T}}, \partial \widetilde{\mathfrak{T}}\right)$ and a diffeomorphism of pairs $g_2:\left(L, \partial \mathfrak{T}\right) \to \left(\widetilde{L}, \partial \widetilde{\mathfrak{T}}\right)$ that agrees with $g_1$ on $\partial \mathfrak{T}$.
    Applying \Cref{lemma:normal_form_around_clean_intersection} along $\partial \mathfrak{T}, \partial \widetilde{\mathfrak{T}}$  in $M, \widetilde{M}$ respectively, we obtain neighbourhoods $U,\widetilde{U}$ of $\partial \mathfrak{T}, \partial \widetilde{\mathfrak{T}}$  in $M, \widetilde{M}$ respectively and a symplectomorphism $U \to \widetilde{U}$ that restricts to a stratified diffeomorphism 
    \begin{equation*}
        g_{\partial}:\left(L \cup \mathfrak{T} \right)\cap U \to \left(\widetilde{L} \cup \widetilde{\mathfrak{T}}\right) \cap \widetilde{U}
    \end{equation*}
    such that the following holds.
    \begin{enumerate}
        \item $g_\partial^*\widetilde{\w} = \w$ as Zariski forms on $\left(L \cup \mathfrak{T} \right)\cap U$.
        \item $g_\partial\Big\lvert_{\partial \mathfrak{T}}:\partial \mathfrak{T} \to \partial \widetilde{\mathfrak{T}}$ coincides with $g_1$ and $g_2$.
        \item $g_\partial\left(\interior\left(\mathfrak{T}\right) \cap U\right) = \interior\left(\widetilde{\mathfrak{T}}\right) \cap \widetilde{U}$.
        \item $g_\partial\left(\left(L \setminus \partial \mathfrak{T}\right) \cap U\right) = \left(\widetilde{L} \setminus \partial \widetilde{\mathfrak{T}}\right) \cap \widetilde{U}$.
    \end{enumerate}
    
    Using $g_1,g_2$ to identify $\mathfrak{T},L$ with $\widetilde{\mathfrak{T}},\widetilde{L}$ respectively and applying \Cref{lemma:uniqueness_of_collars,lemma:uniqueness_of_standard_codim_1_subtorus}, we obtain a stratified diffeomorphism $g:L \cup \mathfrak{T} \to \widetilde{L} \cup \widetilde{\mathfrak{T}}$ that agrees with $g_\partial$ on a smaller neighbourhood of $\partial \mathfrak{T}$. It satisfies $g^*\widetilde{\w} = \w$ by the corresponding property of $g_\partial$.
\end{proof}

\begin{cor}[{\cite[Lemma~2.4]{solid_mutation_conf}}]
    With the assumptions and notation of \Cref{proposition:solid_mutation_stratified_diffeo}, there exist neighbourhoods $U,\widetilde{U}$ of $L \cup \mathfrak{T},\widetilde{L} \cup \widetilde{\mathfrak{T}}$ in $M,\widetilde{M}$ respectively and a symplectomorphism $\G:U \to \widetilde{U}$ which restricts to a stratified diffeomorphism 
    \begin{equation*}
        g:\left(L \cup \mathfrak{T},\S\right) \to \left(\widetilde{L} \cup \widetilde{\mathfrak{T}},\widetilde{\S}\right).
    \end{equation*}
\end{cor}
\begin{proof}
    Apply \Cref{proposition:solid_mutation_stratification,proposition:solid_mutation_stratified_diffeo} and \Cref{thm:main_coisotropic}.
\end{proof}

\appendix

\section{Euler-like vector fields}\label{appendix:euler_likes}

In this appendix we give a quick introduction to Euler-like vector fields and prove
\linebreak \Cref{lemma:closed_open_sets_euler}. For more details about Euler-like vector fields see \cite{bursztyn2019splitting,MEINRENKEN2021224}.

For a submanifold $N\subset M$, denote the normal bundle of $N$ in $M$ by $\nu(M,N)$.

\begin{definition}
    A \textbf{tubular neighbourhood} of $N$ in $M$ consists of a conical open neighbourhood $O\subset \nu(M,N)$ of the zero section and an embedding $\Psi:O\to M$ taking the zero section $N\subset \nu(M,N)$ to $N\subset M$, such that the induced linearization
    \begin{equation*}
        \nu(\Psi):\nu(M,N)\simeq \nu(O,N) \to \nu(M,N)=\nu(M,N)
    \end{equation*}
    is the identity.

    A tubular neighbourhood is \textbf{complete} if $O=\nu(M,N)$.
\end{definition}

\begin{definition}
    Let $E\to M$ be a vector bundle and denote the fiberwise multiplication by $t\in \R$ by $m^t:E\to E$. The \textbf{Euler vector field on $E$} is the vector field $\E_E$ with flow
    \begin{equation*}
        \Phi_{\E_E}^t = m^{\exp(t)}:E\to E.
    \end{equation*}
\end{definition}

\begin{definition}[{\cite[Definition~3.1]{MEINRENKEN2021224}}]
    An \textbf{Euler-like vector field along $\mathbf{N}$} is a vector field $\E$, defined on a neighbourhood of $N$ in $M$, tangent to $N$, and such that its linearization is the Euler vector field on $\nu(M,N)$.
\end{definition}

\begin{remark}
    We do not require an Euler-like vector field to be defined on all $M$. We also do not require completeness of $\E$ or completeness of the tubular neighbourhood.
\end{remark}

\begin{thm}[\cite{bursztyn2019splitting,MEINRENKEN2021224}]\label{thm:euler_like_tubular}
    Let $\E_M$ be an Euler-like vector field along $N$. Then there exists a unique maximal tubular neighbourhood embedding $\Psi:O\to M$ such that $\E_M$ and $\E_{ \nu(M,N)}$ are $\Psi$-related.

    Moreover, the open set $\Psi(O)\subset M$ is given by
    \begin{equation*}
        \{p\in M : \lim_{t\to-\infty}\Phi_{\E_M}^t(p) \text{ exists and is in }N\}.
    \end{equation*}

    If $\E_M$ is complete, then the tubular neighbourhood is complete.
\end{thm}

\begin{lemma}[Convexity of Euler-like vector fields]\label{lemma:euler_like_convex}
    Let ${\{\E_i\}}_{i\in I}$ be Euler-like vector fields along $N$ and let ${\{\varphi_i\}}_{i\in I}$ be a partition of unity.
    Then $\sum_{i\in I} \varphi_i\cdot \E_i$ is also Euler-like along $N$.
\end{lemma}
\begin{proof}
    Calculation in local coordinates.
\end{proof}

\begin{definition}\label{definition:induced_multiplication_by_scalar}
    Let $\E$ be an Euler-like vector field along $N$ and let $\Psi:O \to M$ be the induced tubular neighbourhood embedding. For all $t\in[0,1]$ define the \textbf{induced multiplication by scalars} $m_{\E}^t:\Psi(O) \to \Psi(O)$ by
    \begin{equation*}
        m_{\E}^t = \Psi \circ m^t \circ \Psi^{-1}.
    \end{equation*}
\end{definition}

The following Lemma is used in \Cref{section:coisotropic_theorem}.

\begin{lemma}\label{lemma:closed_open_sets_euler}
    Let $N\subset M$ be a submanifold, $U\subset M$ an open subset, and $K\subset M$ a closed subset with $K\subset U$. Let $\E$ be an Euler-like vector field along $N$, inducing a tubular neighbourhood embedding $\Psi:O \to M$.
    Then there exists a smooth function $\varphi:\Psi(O)\to [0,1]$ that is identically $1$ on a neighbourhood of $N$, and such that the vector field $\widetilde{\E}\coloneq \varphi\cdot \E$ is Euler-like and satisfies the following property. If $K$ intersects a fiber of $m_{\widetilde{\E}}^0:\widetilde{\Psi}(\nu(M,N)) \to N$, then the whole fiber is contained in $U$.
\end{lemma}
\begin{figure}
    \centering
    \begin{tikzpicture}
        \foreach \x in {0,...,16}{
            \draw [dashed, ->] (\x * 0.5,-5) -- (\x * 0.5, -0.1);
            \draw [dashed, ->] (\x * 0.5, 5) -- (\x * 0.5,0.1);
        };
        \fill[rounded corners=10mm,color=yellow, opacity=0.5] (0,3) -- (6,3.5) -- (3, 1) -- (5, -4) node[above left = 4mm,opacity=1,color=yellow!20!black]{$U$} -- (0,-3);
        \filldraw[rounded corners=5mm,color=blue!70, opacity=0.5] (0,2) -- (5,3) -- (2,1) -- (4,-3) node[above left = 1mm,opacity=1,color=blue!50!black]{$K$} -- (0, -2);
        \fill[color=magenta, opacity=0.3] (0,1.3) -- (8,1.3) -- (8, -1.3) node[anchor=south east,opacity=1,fill=magenta!40]{$\supp{\varphi}$} -- (0,-1.3) -- cycle;
        \draw[color=red,line width=1pt] (0,0) -- (5.75,0) node[anchor=south,color=red!70!black]{$N$} -- (8,0);
    \end{tikzpicture}
    \caption{A schematic drawing of \Cref{lemma:closed_open_sets_euler}: the dashed arrows represent the orbits of $m^t_{\E}:\Psi(O) \to \Psi(O)$ for $t\in[0,1]$. If $K \cap \supp{\varphi}$ meets a fiber ${\left(m^0_{\E}\right)}^{-1}(n)$ then ${\left(m^0_{\E}\right)}^{-1}(n) \cap \supp{\varphi} \subset U$.}
    \label{figure:closed_open_euler_like}
\end{figure}
\begin{proof}
    Let $\norm{\cdot}$ be any fiber-wise norm on $\nu(M,N)$. For $n\in N$, denote by ${B(n, r) \subset \nu_n(M,N)}$ the ball of radius $r$ around $n$ in the fiber $\nu_n(M,N)$.

    Set $U_{\nu} \coloneq \Psi^{-1}(U \cap \Psi(O))$ and $K_{\nu} \coloneq \Psi^{-1}(K \cap \Psi(O))$.
    Let $f_{N\cap U}:\left(N\cap U_{\nu}\right) \to \R_{>0}$ be a smooth function with
    \begin{equation*}
        B\left(n, f_{N \cap U}(n)\right) \subset O \cap U_{\nu},\qquad \forall n\in\left(N\cap U_{\nu}\right),
    \end{equation*}
    let $f_{N \setminus K}: \left(N \setminus K_{\nu}\right) \to \R_{>0}$ be a smooth function with
    \begin{equation*}
        B\left(n, f_{N \setminus K}(n)\right) \subset O \setminus K_{\nu}, \qquad \forall n\in\left(N \setminus K_{\nu}\right),
    \end{equation*}
    and let $g:N\to [0,1]$ be a smooth bump function with
    \begin{enumerate}[(a)]
        \item $g\lvert_{N\cap K_{\nu}} \equiv 1$ and
        \item $g\lvert_{N \setminus U_{\nu}} \equiv 0$.
    \end{enumerate}

    Define $f:N \to \R_{>0}$ by
    \begin{equation*}
        f = g\cdot f_{N\cap U} + (1-g)\cdot f_{N \setminus K}
    \end{equation*}
    which is a smooth positive function on $N$. We verify in the following that it satisfies
    \begin{enumerate}
        \item\label{property:contained_in_O} $B(n, f(n)) \subset O$;
        \item\label{property:f_ball} for all $n\in N$, if $K_{\nu} \cap B(n, f(n)) \neq \emptyset$ then $B(n, f(n)) \subset U_{\nu}$.
    \end{enumerate}

    Property \cref{property:contained_in_O} holds because $f$ is a convex combination $f_{N\cap U}, f_{N \setminus K}$. For the second property, let $q\in K_{\nu} \cap B(n, f(n))$. Then either $f_{N \setminus K}(n)$ is not defined, or
    \begin{equation*}
        f_{N \setminus K}(n) < \norm{q} < f(n)
    \end{equation*}
    by definition of $f_{N \setminus K}$. It follows that $f_{N\cap U}(n)$ is defined and
    \begin{equation*}
        \norm{q} < f(n) \leq f_{N\cap U}(n),
    \end{equation*}
    which implies that $B(n, f(n)) \subset B\left(n, f_{N \cap U}(n)\right)$. But by definition of $f_{N \cap U}$ we have
    \begin{equation*}
        B\left(n, f_{N \cap U}(n)\right) \subset U_{\nu},
    \end{equation*}
    so $B(n, f(n)) \subset B\left(n, f_{N \cap U}(n)\right) \subset U_{\nu}$.

    Now let $h:\R_{\geq 0} \to [0,1]$ be a smooth function that is identically $1$ on $[0,\nicefrac{1}{3}]$ and identically $0$ on $[\nicefrac{2}{3}, \infty)$, and define
    \begin{equation*}
        \varphi_{\nu}(q)\coloneq h\left(\frac{\norm{q}}{(f\circ m_0)(q)}\right)
    \end{equation*}
    and
    \begin{equation*}
        \widetilde{\E}_{\nu} \coloneq \varphi_{\nu} \cdot \E_{\nu}.
    \end{equation*}
    The vector field $\widetilde{\E}_{\nu}$ is defined on $O$, Euler-like along $N\subset \nu(M,N)$, and induces a complete tubular neighbourhood embedding $\widetilde{\Psi}_{\nu}:\nu(M,N) \to O$ such that for all $n\in N$,
    \begin{equation*}
        \widetilde{\Psi}_{\nu}(\nu_n(M,N)) \subset B(n, f(n)) \subset \nu_n(M,N).
    \end{equation*}
    Therefore, property \cref{property:f_ball} of $f$ implies that if $K_\nu$ intersects a fiber of $m^0_{\tilde{\E}_\nu}$, then the whole fiber is contained in $U_{\nu}$.

    Let $\varphi=\varphi_{\nu}\circ \Psi^{-1}$ and note that $\widetilde{\E}\coloneq \varphi\cdot \E$ is the  pushforward of $\widetilde{\E}_{\nu}$ by $\Psi$. It is defined on $\Psi(O)$, an Euler-like vector field along $N\subset M$, and induces a complete tubular neighbourhood $\widetilde{\Psi}:\nu(M,N)\to M$ defined by
    \begin{equation*}
        \widetilde{\Psi}= \Psi\circ \widetilde{\Psi}_{\nu}.
    \end{equation*}
    The required property follows from the same property of $\widetilde{\E}_{\nu}$ in $\nu(M,N)$ and the fact that $\widetilde{\E}, \widetilde{\E}_{\nu}$ are $\Psi$-related.
\end{proof}

\section{Theorem~\ref{thm:moser_trick} for a strong deformation retraction}\label{appendix:strong_deformation_retraction_moser}

Recall from \Cref{remark:strong_vs_weak_deformation_retraction} the difference between a weak and a strong deformation retraction.

\begin{thm}\label{thm:moser_trick_with_a_strong_deformation_retraction}
    Let $M$ be a manifold, $A\subset M$ a submanifold, and $V$ a neighbourhood of $A$ admitting a smooth \textbf{strong} deformation retraction to $A$. Let $\w_0, \w_1$ be two symplectic forms on $V$ that agree on $TM\Big\lvert_A$.
    Then there exist neighbourhoods $U',U'' \subset V$ of $A$ and a diffeomorphism
    \begin{equation*}
        \G:U'\to U''
    \end{equation*}
    such that
    \begin{enumerate}
        \item $\G\big\lvert_A = \Id_A$,
        \item $\G^*(\w_1) = \w_0$,
        \item\label{property:strong_deformation_id_differential} $\mathbf{D\G\big\lvert_A = \Id}$.
    \end{enumerate}
\end{thm}

Note that in the statement of \Cref{thm:moser_trick_with_a_strong_deformation_retraction} we don't need the assumption that $A$ is a manifold, as a smooth retract of a manifold is a submanifold:

\begin{thm}[{\cite[Theorem~1.15]{michor2008topics}}]
    Let $M$ be a manifold, $A\subset M$ a subset, and $V$ a neighbourhood of $A$ admitting a smooth retraction to $A$.
    Then $A$ is a submanifold of $M$.
\end{thm}

\begin{definition}
    Let $M$ be a manifold and $A$ a submanifold. The \textbf{vanishing ideal of $\mathbf{A}$ in $\mathbf{M}$} is
    \begin{equation*}
        I(M,A) \coloneq \{f\in C^\infty(M):\: f\big\lvert_A \equiv 0 \}.
    \end{equation*}
    Note that this is indeed an ideal of $C^\infty(M)$.

    The \textbf{time-dependent vanishing ideal of $A$ in $M$} is
    \begin{equation*}
        I_{td}(M,A) \coloneq \{f\in C^\infty(\R \times M):\: f\big\lvert_{\R \times A} \equiv 0 \}.
    \end{equation*}
    This is an ideal of $C^\infty(\R \times M)$.
\end{definition}

\begin{definition}[{\cite[Definition~2.2]{diff_geom_of_weightings}}]\label{definition:order_of_vanishing_filtration}
    Abbreviate $I \coloneq I(M,A)$ and $I_{td} \coloneq I_{td}(M,A)$.

    The \textbf{$\mathbf{I}$-adic filtration of $\mathbf{C^\infty(M)}$} is
    \begin{equation*}
        \ldots \subset I_2 \subset I_1 \subset I_0 = C^\infty(M),
    \end{equation*}
    where $I_j = \left<I^j \right>$ is the ideal generated by $I^j$.

    Similarly, the \textbf{$\mathbf{I_{td}}$-adic filtration of $\mathbf{C^\infty(\R \times M)}$} is
    \begin{equation*}
        \ldots \subset I_{td,2} \subset I_{td,1} \subset I_{td,0} = C^\infty(\R \times M),
    \end{equation*}
    where $I_{td,j} = \left<I_{td}^j \right>$.
\end{definition}

\begin{remark}
    \Cref{definition:order_of_vanishing_filtration} gives a sense of order of vanishing of functions without using local coordinates. An equivalent definition is by order of vanishing in some (and thus, in all) local coordinates. The proof of this equivalence uses the existence of a finite atlas, whose domains of charts are not necessarily connected (see, e.g., \cite[Section~1.2]{alma990000534930204146}).
\end{remark}

\begin{notation}
    Let $E\to M$ be a vector bundle. We denote the set of smooth sections of $E$ by $\Gamma(E)$ and the set of smooth time-dependent sections of $E$ by $\Gamma_{td}(E)$.
\end{notation}

\begin{definition}\label{definition:I_adic_filtration_of_sections}
    Let $E\to M$ be a vector bundle. The \textbf{induced $\mathbf{I}$-adic filtration} of the $C^\infty(M)$-module $\Gamma(E)$ is
    \begin{equation*}
        \ldots \subset I^{\Gamma(E)}_2 \subset I^{\Gamma(E)}_1 \subset I^{\Gamma(E)}_0 = {\Gamma(E)},
    \end{equation*}
    where $I^{\Gamma(E)}_j = \left< I_j \cdot \Gamma(E) \right>$ is the submodule generated by $I_j \cdot \Gamma(E)$.

    Similarly, the \textbf{induced $\mathbf{I_{td}}$-adic filtration} of the $C^\infty(\R \times M)$-module $\Gamma_{td}(E)$ is
    \begin{equation*}
        \ldots \subset I^{\Gamma_{td}(E)}_2 \subset I^{\Gamma_{td}(E)}_1 \subset I^{\Gamma_{td}(E)}_0 = {\Gamma_{td}(E)},
    \end{equation*}
    where $I^{\Gamma_{td}(E)}_j =  \left< I_{td,j} \cdot \Gamma_{td}(E) \right>$.
\end{definition}

\begin{remark}
    Note that with this filtration, the complex $\left(\Omega^\bullet(M), d \right)$ is \emph{not} a filtered complex. Cf. \cite[Section~2.3]{diff_geom_of_weightings}.
\end{remark}

\begin{lemma}\label{lemma:integral_over_td_sections_vanish_order}
    Let $\sigma_t \in I_j^{\Gamma_{td}(E)}$. Then $\int_0^1\sigma_t dt \in I_j^{\Gamma(E)}$.
\end{lemma}
\begin{proof}
    This follows by computation in local coordinates.
\end{proof}

\begin{lemma}\label{lemma:pullbacks_preserve_filtrations}
    Let $M, M'$ be manifolds, $A, A'$ submanifolds of them, $I,I'$ the vanishing ideals, and $I_j, I_j'$ the $I,I'$-adic filtrations, respectively. Let $g:M \to M'$ be a smooth map with $g(A) \subset A'$. Let $E' \to M'$ be a vector bundle and $g^*E' \to M$ the pullback bundle.
    Then
    \begin{enumerate}
        \item The pullback of functions $g^*:C^\infty(M') \to C^\infty(M)$,
        \item the pullback of sections $g_s^*:\Gamma(E') \to \Gamma(g^*E')$,
        \item the pullback of forms $g^*:\Omega^i(M') \to \Omega^i(M)$
    \end{enumerate}
    all preserve the respective induced $I,I'$-adic filtrations.
\end{lemma}
\begin{proof}
    This follows by computation in local coordinates.
\end{proof}

\begin{notation}
    Denote by $\mathfrak{X}(M) \coloneq \Gamma(TM)$ the set of smooth vector fields on $M$. Denote by $\mathfrak{X}_{td}(M)$ the set of smooth time-dependent vector fields on $M$.
\end{notation}

\begin{lemma}\label{lemma:vector_field_vanish_to_second_order_differential}
    Let $(X_t)\in \mathfrak{X}_{td}(M)$ and assume that $(X_t)\in I_2^{\mathfrak{X}_{td}(M)}$.
    Then the differential of the flow $D\Phi_{(X_t)}^{t_0}\Big\lvert_A:TM \to TM$ restricts on $A$ to
    \begin{equation*}
        \Id:TM\Big\lvert_A \to TM\Big\lvert_A,
    \end{equation*}
    whenever it is defined.
\end{lemma}
\begin{proof}
    Let $p\in A$ and let $U$ be a neighbourhood of $p$ in $M$, with coordinates $(x,y)$, such that $p$ is identified with $0$ and $U\cap A$ is identified with $\{y=0\}$. Then
    \begin{equation*}
        \left(x,y ;\; v_x, v_y \right)
    \end{equation*}
    are coordinates on $TU$ and
    \begin{equation*}
        \left(x,y, v_x, v_y; u_x, u_y, u_{v_x}, u_{v_y}\right)
    \end{equation*}
    are coordinates on $TTU$.

    Now, $X_t(x,y)$ has the form
    \begin{equation*}
        X_t(x,y) = \left(x,y ;\; X_t^x(x,y), X_t^y(x,y) \right),
    \end{equation*}
    where, for all $t$,
    \begin{align*}
        &X_t^x(x,0)=0, &DX_t^x(x,0)= 0, \\
        &X_t^y(x,0)=0, &DX_t^y(x,0)= 0.
    \end{align*}
    It follows that
    \begin{equation*}
        DX_t(x, 0; v_x, v_y) = \left(x, 0, 0, 0;\; v_x, v_y, 0, 0 \right).
    \end{equation*}
    Lastly, we check that
    \begin{equation*}
        \Id:TM\Big\lvert_A \to TM\Big\lvert_A \qquad \forall t_0\in \R
    \end{equation*}
    satisfies the differential equation
    \begin{equation*}
        \frac{d}{ds}\bigg\lvert_{t_0}\left(D\Phi^{s}_{(X_t)}(v) \right) = \left(\kappa_M\circ DX_{t_0} \circ D\Phi_{X_{t_0}}^{t_0}\right)(v)
    \end{equation*}
    which characterizes $D\Phi^{t_0}_{(X_t)}$. Here $\kappa_M$ is \emph{canonical flip}, defined by
    \begin{equation*}
        \kappa_M\left(x,y, v_x, v_y; u_x, u_y, u_{v_x}, u_{v_y}\right) = \left(x,y, u_x, u_y; v_x, v_y, u_{v_x}, u_{v_y}\right)
    \end{equation*}
    See \cite[Sections~(8.12-8.19)]{michor2008topics}. Indeed,
    \begin{align*}
        \frac{d}{dt}\bigg\lvert_{t_0}\bigl( \Id\left(x,0;\; v_x, v_y\right)\bigr) &= \frac{d}{dt}\bigg\lvert_{t_0}\left( x,0;\; v_x, v_y\right)\\
        &= \left(x,0,v_x,v_y ;\; 0,0,0,0 \right)
    \end{align*}
    and
    \begin{align*}
        \left(\kappa_M\circ DX_t \circ \Id\right)\left(x,0,v_x,v_y\right) &= \kappa_M\left(x,0,0,0 ;\; v_x,v_y,0,0 \right) \\
        &= \left(x,0,v_x,v_y ;\; 0,0,0,0 \right).
    \end{align*}
\end{proof}

\begin{proof}[Proof of \Cref{thm:moser_trick_with_a_strong_deformation_retraction}]
    We follow the proofs of \Cref{lemma:relative_poincare,lemma:using_beta_for_moser} and use the same notation, with the additional assumption that the deformation retraction $R:[0,1]\times V \to V$ is strong. We want to prove property~\cref{property:strong_deformation_id_differential}, i.e., that the diffeomorphism $\G:U' \to U''$ that we construct in the proofs of these lemmas satisfies $D\G\Big\lvert_A = \Id$.

    Recall that $\G$ is the time one flow of a time-dependent vector field $(X_t)$ defined by ${\i_{X_t}\w_t = \beta}$, where
    \begin{equation*}
        \beta \coloneq \int_0^1 i^*_t \i_{\frac{\partial}{\partial t}} R^* \alpha
    \end{equation*}
    and $\alpha \coloneq \w_0 - \w_1$. We claim that it is enough to prove that $\beta \in I_2^{\Omega^1(V)}$.
    Indeed, $\beta \in I_2^{\Omega^1(V)}$ implies $(X_t)\in I_2^{\mathfrak{X}_{td}(V)}$, and we finish by \Cref{lemma:vector_field_vanish_to_second_order_differential}.

    Using the formula for $\beta$ and \Cref{lemma:integral_over_td_sections_vanish_order,lemma:pullbacks_preserve_filtrations}, it is enough to show that
    \begin{equation*}
        \i_{\frac{\partial}{\partial t}} R^* \alpha \in I_2^{\Omega^1\bigl((0,1) \times V\bigr)}
    \end{equation*}
    for the induced $I\Bigl((0,1) \times V, (0,1) \times A\Bigr)$-adic filtration.

    Let $p\in A$, $t_0\in (0,1)$ and $v\in T_{(t_0, p)}\Bigl((0,1) \times V \Bigr)$. Then
    \begin{equation}\label{equation:iota_d_dt_R_alpha_expression}
        \begin{aligned}
            \left(\i_{\frac{\partial}{\partial t}} R^* \alpha \right)_{p}(v)
            &= \left(R^* \alpha\right)_{(t_0, p)} \left(\frac{\partial}{\partial t}, v \right) \\
            &= \alpha_{R(t_0, p)}\left(D_{(t_0, p)}R\left(\frac{\partial}{\partial t}\right), D_{(t_0, p)}R (v)\right).
        \end{aligned}
    \end{equation}

    We consider the pullback bundles
    \begin{enumerate}
        \item $R^*\left(\bigwedge^2 T^*V \right) \to (0,1) \times V$,
        \item $R^*\left(TV\right) \to (0,1) \times V$,
        \item $R^*\left(T^*V\right) \to (0,1) \times V$.
    \end{enumerate}
    Since $\alpha \in I_1^{\Omega^2(V)}$, \Cref{lemma:pullbacks_preserve_filtrations} implies that the section $R_s^*\alpha$ of the bundle ${R^*\left(\bigwedge^2 T^*V \right)}$ satisfies
    \begin{equation*}
        R_s^*\alpha \in I_1^{\Gamma\left(R^*\left(\bigwedge^2 T^*V \right)\right)}.
    \end{equation*}
    Since the deformation retraction is strong, the section $DR\left(\frac{\partial}{\partial t} \right)$ of the pullback bundle $R^*\left(TV\right)$ satisfies
    \begin{equation*}
        DR\left(\frac{\partial}{\partial t} \right) \in I_1^{\Gamma\left(R^*\left(TV\right)\right)}.
    \end{equation*}

    The point-wise contraction on the pullback bundles
    \begin{equation*}
        \i:\Gamma\left(R^*\left(\bigwedge^k T^*V \right)\right) \times \Gamma\left(R^*\left(TV\right)\right) \to  \Gamma\left(R^*\left(\bigwedge^{k-1} T^*V\right) \right)
    \end{equation*}
    is $C^\infty\left( (0,1) \times V\right)$-bilinear; it follows that
    \begin{equation*}
        \i\Bigl( I_1^{\Gamma\left(R^*\left(\bigwedge^2 T^*V \right)\right)} \times I_1^{\Gamma\left(R^*\left(TV\right)\right)}\Bigr) \subset I_2^{\Gamma\left(R^*\left(T^*V\right)\right)}.
    \end{equation*}
    In particular,
    \begin{equation*}
        \i\left(R_s^*\alpha, DR\left(\frac{\partial}{\partial t} \right)\right) \in I_2^{\Gamma\left(R^*\left(T^*V\right)\right)}.
    \end{equation*}
    Therefore, for all $\tilde{Y} \in \Gamma\left(R^*\left(TV\right)\right)$, we have
    \begin{equation}\label{equation:double_contraction}
        \begin{aligned}
            \i\left(\i\left(R_s^*\alpha, DR\left(\frac{\partial}{\partial t} \right)\right), \tilde{Y} \right) &= R_s^*\alpha \left(DR\left(\frac{\partial}{\partial t}\right), \tilde{Y} \right) \\
            &\in I_2 \Bigl( (0,1) \times V, (0,1) \times A \Bigr).
        \end{aligned}
    \end{equation}

    For all $Y \in \mathfrak{X}\Bigl((0,1) \times V \Bigr)$ consider $DR\left(Y\right) \in \Gamma\left(R^*\left(TV\right)\right)$.
    By equation~\eqref{equation:iota_d_dt_R_alpha_expression}, we have
    \begin{equation*}
        \left(\i_{\frac{\partial}{\partial t}} R^* \alpha\right)(Y) =  R_s^*\alpha \left(DR\left(\frac{\partial}{\partial t}\right), DR\left(Y\right) \right).
    \end{equation*}
    Equation~\eqref{equation:double_contraction} therefore implies that
    $\left( \i_{\frac{\partial}{\partial t}}R^*\alpha \right)(Y) \in I_2 \left( (0,1) \times V, (0,1) \times A \right)$.
    Finally, since this holds for all $Y \in \mathfrak{X}\Bigl((0,1) \times V \Bigr)$, we obtain $\i_{\frac{\partial}{\partial t}} R^* \alpha \in I_2^{\Omega^1\bigl((0,1) \times V\bigr)}$.
\end{proof}

\bibliographystyle{plain}
\bibliography{main}

\end{document}